\documentclass[oneside,american]{amsart}
\pdfoutput=1
\usepackage[T1]{fontenc}
\usepackage[latin9]{inputenc}
\usepackage{babel}
\usepackage{verbatim}
\usepackage{refstyle}
\usepackage{mathrsfs}
\usepackage{enumitem}
\usepackage{amstext}
\usepackage{amsthm}
\usepackage{amssymb}
\usepackage{graphicx}
\usepackage{multirow}
\usepackage[unicode=true,pdfusetitle,
 bookmarks=true,bookmarksnumbered=false,bookmarksopen=false,
 breaklinks=false,pdfborder={0 0 1},backref=false,colorlinks=false]
 {hyperref}
\usepackage{breakurl}

\makeatletter

\AtBeginDocument{\providecommand\propref[1]{\ref{prop:#1}}}
\AtBeginDocument{\providecommand\thmref[1]{\ref{thm:#1}}}
\AtBeginDocument{\providecommand\subsecref[1]{\ref{subsec:#1}}}
\AtBeginDocument{\providecommand\lemref[1]{\ref{lem:#1}}}
\AtBeginDocument{\providecommand\corref[1]{\ref{cor:#1}}}
\AtBeginDocument{\providecommand\figref[1]{\ref{fig:#1}}}
\RS@ifundefined{subsecref}
  {\newref{subsec}{name = \Rsubsecxt}}
  {}
\RS@ifundefined{thmref}
  {\def\RSthmtxt{theorem~}\newref{thm}{name = \RSthmtxt}}
  {}
\RS@ifundefined{lemref}
  {\def\RSlemtxt{lemma~}\newref{lem}{name = \RSlemtxt}}
  {}

\numberwithin{equation}{section}
\numberwithin{figure}{section}
\theoremstyle{plain}
\newtheorem{thm}{\protect\theoremname}[section]
  \theoremstyle{plain}
  \newtheorem{prop}[thm]{\protect\propositionname}
  \theoremstyle{definition}
  \newtheorem{defn}[thm]{\protect\definitionname}
  \theoremstyle{remark}
  \newtheorem*{rem*}{\protect\remarkname}
  \theoremstyle{plain}
  \newtheorem{lem}[thm]{\protect\lemmaname}
  \theoremstyle{plain}
  \newtheorem{cor}[thm]{\protect\corollaryname}
   \theoremstyle{plain}

\usepackage{graphicx}
\usepackage{color}

\newref{prop}{name=Proposition~}
\newref{thm}{name=Theorem~}
\newref{lem}{name=Lemma~}
\newref{cor}{name=Corollary~}
\newref{conj}{name=Conjecture~}
\newref{fig}{name=Figure~}
\newref{subsec}{name=Subsection~}
\newref{sec}{name=Section~}

\@ifundefined{showcaptionsetup}{}{%
 \PassOptionsToPackage{caption=false}{subfig}}
\usepackage{subfig}
\makeatother

\providecommand{\corollaryname}{Corollary}
\providecommand{\definitionname}{Definition}
\providecommand{\lemmaname}{Lemma}
\providecommand{\propositionname}{Proposition}
\providecommand{\remarkname}{Remark}
\providecommand{\theoremname}{Theorem}

\DeclareMathOperator{\diag}{diag}
\DeclareMathOperator{\interior}{int}

\begin{document}

\title[Invasion of open space by two competitors]{Invasion of open space by two competitors: 
    spreading properties of monostable two-species competition--diffusion systems}

\author{L\'{e}o Girardin$^{1}$, King-Yeung Lam$^{2}$}

\thanks{This work has been carried out in the framework of the NONLOCAL project
(ANR-14-CE25-0013) funded by the French National Research Agency (ANR). 
The second author is partially supported by National Science Foundation under grant DMS-1853561.\\
$^{1}$: Laboratoire de Math\'{e}matiques d'Orsay, Universit\'{e} Paris Sud, CNRS,
Universit\'{e} Paris-Saclay, 91405 Orsay cedex, France. \\
$^{2}$: Department of Mathematics, The Ohio State University, Columbus, OH, USA}

\email{$^{1}$: leo.girardin@math.u-psud.fr, $^{2}$: lam.184@osu.edu}

\begin{abstract}
This paper is concerned with some spreading properties of monostable
Lotka--Volterra two-species competition--diffusion
systems when the initial values are null or exponentially decaying in a right
half-line. Thanks to a careful construction of super-solutions
and sub-solutions, we improve previously known results and settle
open questions. In particular, we show that if the weaker competitor
is also the faster one, then it is able to evade 
the stronger and slower competitor by invading first into unoccupied territories. The pair of speeds
depends on the initial values. If these are null in a right half-line,
then the first speed is the KPP speed of the fastest competitor and
the second speed is given by an exact formula describing the possibility
of nonlocal pulling. Furthermore, the unbounded set of pairs of speeds
achievable with exponentially decaying initial values is characterized,
up to a negligible set. 
\end{abstract}



\maketitle
\tableofcontents{}

\section{Introduction}

In this paper, we are interested in some spreading properties of the
classical monostable Lotka\textendash Volterra two-species competition\textendash diffusion
system

\begin{equation}
\left\{ \begin{matrix}\partial_{t}u-\partial_{xx}u=u\left(1-u-av\right) & \text{in }\left(0,+\infty\right)\times\mathbb{R}\\
\partial_{t}v-d\partial_{xx}v=rv\left(1-v-bu\right) & \text{in }\left(0,+\infty\right)\times\mathbb{R}\\
u\left(0,x\right)=u_{0}\left(x\right) & \text{for all }x\in\mathbb{R}\\
v\left(0,x\right)=v_{0}\left(x\right) & \text{for all }x\in\mathbb{R}
\end{matrix}\right.\label{eq:competition_diffusion_system}
\end{equation}
 with $d>0$, $a\in\left(0,1\right)$, $b>1$, $r>0$ and $u_{0},v_{0}\in\mathscr{C}\left(\mathbb{R},\left[0,1\right]\right)\backslash\left\{ 0\right\} $.
The assumptions on $a$ and $b$ mean that $u$ and $v$ are respectively
the stronger and the weaker competitor.

Recall from the classical literature \cite{Aronson_Weinbe,Fisher_1937,KPP_1937}
that the scalar Fisher-KPP equation
\[
\left\{ \begin{matrix}\partial_{t}w-\delta\partial_{xx}w=\rho w\left(1-w\right) & \text{in }\left(0,+\infty\right)\times\mathbb{R}\\
w\left(0,x\right)=w_{0}\left(x\right) & \text{for all }x\in\mathbb{R}
\end{matrix}\right.
\]
 with $\delta,\rho>0$ and $w_{0}\in\mathscr{C}_{b}\left(\mathbb{R}\right)$
with nonempty support included in $\left(-\infty,0\right]$ has the
following spreading property: there exists a unique $c_{\textup{KPP}}>0$ satisfying
\[
\left\{ \begin{matrix}\lim\limits _{t\to+\infty}\sup\limits _{\left|x\right|<ct}\left|w\left(t,x\right)-1\right|=0\text{ for each }c<c_{\textup{KPP}}\\
\lim\limits _{t\to+\infty}\sup\limits _{ct<x}\left|w\left(t,x\right)\right|=0\text{ for each }c>c_{\textup{KPP}}
\end{matrix}\right..
\]
These asymptotics describe the invasion of the unstable state $0$
by the stable state $1$ and $c_{\textup{KPP}}$ is consequently referred to
as the spreading speed of this invasion. Furthermore, $c_{\textup{KPP}}$ coincides
with the minimal speed of the traveling wave solutions, which are
particular entire solutions of the form $w:\left(t,x\right)\mapsto\varphi\left(x-ct\right)$
with $\varphi\geq0$, $\varphi\left(-\infty\right)=1$ and $\varphi\left(+\infty\right)=0$.
A striking result is the so-called \textit{linear determinacy} property:
there exists such a pair $\left(\varphi,c\right)$ if and only if
the linear equation 
\[
-\delta\varphi''-c\varphi'=\rho\varphi,
\]
namely, the linearization at $\varphi=0$ of the semilinear equation
satisfied by $\varphi$, admits a positive solution in $\mathbb{R}$. Consequently,
$c_{\textup{KPP}}=2\sqrt{\rho\delta}$. As far as the system (\ref{eq:competition_diffusion_system})
is concerned, this result shows that in the absence of the competitor,
$u$ and $v$ respectively spread at speed $2$ and $2\sqrt{rd}$. 

Recall also from the collection of works due to Lewis, Li and Weinberger
\cite{Lewis_Weinberg,Li_Weinberger_} that the competition\textendash diffusion
system
\begin{equation}
\left\{ \begin{matrix}\partial_{t}u-\partial_{xx}u=u\left(1-u-av\right) & \text{in }\left(0,+\infty\right)\times\mathbb{R}\\
\partial_{t}v-d\partial_{xx}v=rv\left(1-v-bu\right) & \text{in }\left(0,+\infty\right)\times\mathbb{R}\\
u\left(0,x\right)=\tilde{u}_{0}\left(x\right) & \text{for all }x\in\mathbb{R}\\
v\left(0,x\right)=1-\tilde{v}_{0}\left(x\right) & \text{for all }x\in\mathbb{R}
\end{matrix}\right.\label{eq:related_competition_diffusion_system}
\end{equation}
with $\tilde{u}_{0}$ and $\tilde{v}_{0}$ compactly supported and
$\tilde{u}_{0}$ nonnegative nonzero, has an analogous spreading property: there
exists a unique $c_{\textup{LLW}}>0$ satisfying
\[
\left\{ \begin{matrix}\lim\limits _{t\to+\infty}\sup\limits _{\left|x\right|<ct}\left(\left|u\left(t,x\right)-1\right|+\left|v\left(t,x\right)\right|\right)=0\text{ for each }c<c_{\textup{LLW}}\\
\lim\limits _{t\to+\infty}\sup\limits _{ct<\left|x\right|}\left(\left|u\left(t,x\right)\right|+\left|v\left(t,x\right)-1\right|\right)=0\text{ for each }c>c_{\textup{LLW}}
\end{matrix}\right.
\]
and describing the invasion of the unstable state $\left(0,1\right)$
by the stable state $\left(1,0\right)$. As in the KPP case, the spreading
speed $c_{\textup{LLW}}$ is the minimal speed of the monotonic traveling wave
solutions; linearizing at $\left(0,1\right)$, it is easily deduced
that the linear speed is $2\sqrt{1-a}$ and that $c_{\textup{LLW}}\geq2\sqrt{1-a}$.
However, contrarily to the KPP case, the converse inequality $c_{\textup{LLW}}\leq2\sqrt{1-a}$
is only sometimes true. Depending on the parameters, linear determinacy 
($c_{\textup{LLW}}=2\sqrt{1-a}$) holds true
in some cases (we refer for instance to Lewis--Li--Weinberger \cite{Lewis_Weinberg}, 
Huang \cite{Huang_2010}, Alhasanat--Ou \cite{Alhasanat_Ou}) but fails 
($c_{\textup{LLW}}>2\sqrt{1-a}$) in other cases (for instance, Huang--Han 
\cite{Huang_Han_2011}, Alhasanat--Ou \cite{Alhasanat_Ou}). Let us point out that a sharp necessary
and sufficient condition on the parameters for linear determinacy is unknown; to this day,
this is one of the greatest open problems on this system.
We also recall the related notions of \textit{pulled front} (roughly speaking, a front 
that is driven by its exponential tail) 
and \textit{pushed front} (roughly speaking, a front that is pushed by its back) as well as
the partially proved conjecture of Roques--Hosono--Bonnefon--Boivin \cite{Roques_Hosono}
stating the equivalence between the dichotomy ``pulled or pushed front'' and the dichotomy
``linear determinacy or failure of linear determinacy'' (namely, a front is pulled if and only
if linear determinacy holds true).

Independently of this linear determinacy issue, a rough upper estimate of 
$c_{\textup{LLW}}$ can be obtained by comparison
with the KPP equation satisfied by $u$ in the absence of $v$: $c_{\textup{LLW}}\leq2$
(the competition always slows down the invasion of $u$). The strict
inequality $c_{\textup{LLW}}<2$ is expected but, as far as we know, cannot
be established easily when linear determinacy fails. 

We focus now on the system (\ref{eq:competition_diffusion_system})
and observe that, when $u_{0}$ and $v_{0}$ are both null or exponentially decaying
in $\left[0,+\infty\right)$, the long-time behavior in $\left(0,+\infty\right)$ is unclear. It is the purpose of this
paper to address this question. 

If $rd>1$ and $u_{0}$ and $v_{0}$ are compactly supported, then for all small $\epsilon>0$,
\[
\lim\limits _{t\to+\infty}\sup\limits _{\left|x\right|<(c_{\textup{LLW}} - \epsilon)t}\left(\left|u\left(t,x\right)-1\right|+\left|v\left(t,x\right)\right|\right)=0,
\]
\[
\lim\limits _{t\to+\infty}\sup\limits _{(2+\epsilon)t<\left|x\right|<(2\sqrt{rd}-\epsilon)t}\left(\left|u\left(t,x\right)\right|+\left|v\left(t,x\right)-1\right|\right)=0,
\]
\[
\lim\limits _{t\to+\infty}\sup\limits _{(2\sqrt{rd} + \epsilon)t<\left|x\right|}\left(\left|u\left(t,x\right)\right|+\left|v\left(t,x\right)\right|\right)=0.
\]
This fact, which we are going to prove
in the forthcoming pages (see \propref{Hair-trigger_effect_vs_extinction})
by adapting very slightly arguments from the related literature and
which is therefore not really new, basically means the following: the open
space is first invaded by the faster competitor $v$ at speed $2\sqrt{rd}$
and then the replacement of $v$ by the stronger competitor $u$ occurs
somewhere in the area $c_{\textup{LLW}}\leq\frac{x}{t}\leq2$. In particular,
as far as spreading speeds are concerned, the first invasion ($\left(0,0\right)$
by $\left(0,1\right)$) is not influenced by the second invasion ($\left(0,1\right)$
by $\left(1,0\right)$): the competition exerted by the exponential
tail of $u$ in the area $2<\frac{x}{t}$ is negligible. 

It is then natural to investigate whether the converse statement is
true: is the second invasion influenced by the first one? Is it possible
to show that the speed $c_{2}$ of the second invasion is exactly
$c_{\textup{LLW}}$, or is there on the contrary a possibility of speed enhancement,
namely $c_{2}>c_{\textup{LLW}}$?

Heuristically, two distinct spreading phenomena occur simultaneously and might influence
the second speed.
\begin{enumerate}
    \item The first phenomenon occurs at the interface between $(1,0)$ and $(0,1)$.  
	As explained previously, its front is either pulled or pushed, with speed
	$c_{\textup{LLW}}$. This is a purely local phenomenon.
    \item The second phenomenon occurs far away on the right, at the interface between 
	$(0,1)$ and $(0,0)$. Here, an exponentially small quantity of 
	$u$ spreads in an environment where $v\simeq\mathbf{1}_{\{x \leq 2\sqrt{rd}t\}}$ 
	does not depend on $u$ anymore. The equation satisfied by $u$ here is a KPP equation with
	a uniformly positive heterogeneous intrinsic growth rate 
	$1-a\mathbf{1}_{\{x \leq 2\sqrt{rd}t\}}$: its front is 
	strongly expected to be a pulled front with linearly determined speed. However,
	this front spreads at most at speed $2$ whereas the surrounding environment spreads
	at speed $2\sqrt{rd}>2$. Hence this spreading phenomenon is nonlocal, in some sense, and 
	accordingly its front will hereafter be referred to as \textit{nonlocally pulled} and 
	its speed will be denoted $c_{\textup{nlp}}$. 
\end{enumerate}

In the first work on staged invasions of two competitors, due to
Shigesada and Kawasaki in 1997 \cite{Kawasaki_Shige}, the spreading 
speeds of the two competitors were estimated based on conjectures of linear 
determinacy and local determinacy.
They, however, noted that the conjectures might need to be revised, in view of the 
numerical results of Hosono \cite{Hosono_1989} illustrating the possible failure
of linear determinacy. Nevertheless, the purpose of the investigation of 
Shigesada--Kawasaki was more to raise
interesting mathematical problems than to rigorously solve them.
Rigorous analysis started more recently, in the last decade, with Carr\`{e}re 
\cite{Carrere_2017} and Lin and Li \cite{Lin_Li_2012}.

Carr\`{e}re studied the bistable case ($a>1$, $b>1$). She proved that
the second invasion has the speed
of the unique bistable traveling wave connecting $\left(0,1\right)$
to $\left(1,0\right)$: the two invasions are indeed independent, nonlocal pulling 
does not occur. However we point out that the bistable case is quite different from
the monostable one (based on the uniqueness 
of traveling wave speed and profile in the bistable case,  
the arguments used on the left of the second transition can
be used again on its right). 

As a matter of fact, Lin and Li investigated
the monostable case with stable coexistence ($a<1$, $b<1$) and the second speed
remained elusive. All three monostable cases (stable coexistence, stable $\left(1,0\right)$,
stable $\left(0,1\right)$) being handled quite similarly (see Lewis\textendash Li\textendash Weinberger
\cite{Lewis_Weinberg} for instance), the technical obstacles they
encountered should not depend on the sign of $b-1$. 

In the present paper, we adopt a new point of view: we aim directly
for the construction of (almost) optimal pairs of super-solutions
and sub-solutions. This point of view turns out to be highly fruitful. Indeed,
the forthcoming \thmref{Compactly_supported} states that, when the support of $u_{0}$
is included in a left half-line ($u_{0}$ is Heavyside-like
or compactly supported) and $v_{0}$ is compactly supported,
the actual invasion speed of $u$ is simply the maximum of $c_{\textup{LLW}}$ and $c_{\textup{nlp}}$
(the sign of $c_{\textup{nlp}}-c_{\textup{LLW}}$ can vary). 
Therefore speed enhancement does occur in some cases (depending on the parameters).
By taking into account the possibility of failure
of linear determinacy and by showing the possibility of nonlocal pulling, our work
completely settles the mathematical questions raised by Shigesada--Kawasaki \cite{Kawasaki_Shige}. 

In addition to this first result, our approach also delivers a general existence
and nonexistence result related to \textit{propagating terraces} (succession
of compatible traveling waves with decreasingly ordered speeds, first
described by Fife and McLeod \cite{Fife_McLeod_19}) having the unstable
steady state $\left(0,1\right)$ as intermediate steady state and corresponding
to exponentially decaying initial data. As
far as scalar terraces for reaction--diffusion equations
are concerned, Ducrot, Giletti and Matano \cite{Ducrot_Giletti_Matano}
showed quite generically that all intermediate states are stable from
below (see also Pol\'{a}\v{c}ik \cite{Polacik_2018} for a complete account in the general setting). 
In more sophisticated contexts (reaction--diffusion
systems, nonlocal equations, etc.), propagating terraces with unstable
intermediate states are observed numerically 
\cite[among others]{Faye_Holzer_20,Holzer_Scheel,Iida_Lui_Ninomiya,Nadin_Perthame_Tang,Petrovskii_Kawasaki_Takasu_Shigesada,Sherratt_1998,Venegas_Ortiz_}. 
Rigorous analytical studies are however very difficult and have only been carried 
out in simple cases and with Heavyside-like or compactly supported initial data 
\cite{Holzer_Scheel,Iida_Lui_Ninomiya}.
In this regard, the present paper is, to the best of our knowledge, unprecedented. 

Our approach relies heavily upon the comparison principle. Therefore it might be appropriate for
some cooperative systems of arbitrary size (let us recall however that a fully 
coupled cooperative system, 
namely a cooperative system where the positivity of any one component implies the 
positivity of all the others, necessarily has a single spreading speed). Unfortunately,
our approach cannot be adapted to settings devoid of comparison principle. 

Finally, let us point out that our forthcoming results would still
hold true if $u\left(1-u\right)$ and $rv\left(1-v\right)$ were replaced
by more general KPP reaction terms. In order to ease the reading,
however, we focus on the traditional logistic form. 

\subsection{Main results}\label{subsec:1.1}

Define the auxiliary function
\begin{equation}\label{eq:f}
\begin{matrix}f: & \left[2\sqrt{1-a},+\infty\right) & \to & \left(2\sqrt{a},2\left(\sqrt{1-a}+\sqrt{a}\right)\right]\\
 & c & \mapsto & c-\sqrt{c^{2}-4\left(1-a\right)}+2\sqrt{a}
\end{matrix}.
\end{equation}
 This function is decreasing and bijective and satisfies in particular
\[
f\left(2\right)=2,
\]
\[
f^{-1}:\tilde{c}\mapsto\frac{\tilde{c}}{2}-\sqrt{a}+\frac{2\left(1-a\right)}{\tilde{c}-2\sqrt{a}}.
\]

\subsubsection{Spreading properties of initially localized solutions}
\begin{thm}
\label{thm:Compactly_supported} Let $u_{0}\in\mathscr{C}\left(\mathbb{R},\left[0,1\right]\right)\backslash\left\{ 0\right\} $
with support included in a left half-line and $v_{0}\in\mathscr{C}\left(\mathbb{R},\left[0,1\right]\right)\backslash\left\{ 0\right\} $
with compact support. Let $\left(u,v\right)$ be the solution of \eqref{competition_diffusion_system}.
\begin{enumerate}
\item Assume $2\sqrt{rd}<2$. Then
\[
\lim_{t\to+\infty}\sup_{x\geq 0}\left|v\left(t,x\right)\right|=0,
\]
\[
\lim\limits _{t\to+\infty}\sup\limits _{0\leq x<\left(2-\varepsilon\right)t}\left|u\left(t,x\right)-1\right|=0\text{ for each }\varepsilon\in\left(0,2\right),
\]
\[
\lim\limits _{t\to+\infty}\sup\limits _{\left(2+\varepsilon\right)t<x}\left|u\left(t,x\right)\right|=0\text{ for each }\varepsilon>0.
\]
\item Assume $2\sqrt{rd}\in\left(2,f\left(c_{\textup{LLW}}\right)\right)$ and define
\[
c_{\textup{nlp}}=f^{-1}\left(2\sqrt{rd}\right)=\sqrt{rd}-\sqrt{a}+\frac{1-a}{\sqrt{rd}-\sqrt{a}}\in\left(c_{\textup{LLW}},2\right).
\]
Then
\[
\lim\limits _{t\to+\infty}\sup\limits _{0\leq x<\left(c_{\textup{nlp}}-\varepsilon\right)t}\left(\left|u\left(t,x\right)-1\right|+\left|v\left(t,x\right)\right|\right)=0\text{ for each }\varepsilon\in\left(0,c_{\textup{nlp}}\right),
\]
\[
\lim\limits _{t\to+\infty}\sup\limits _{\left(c_{\textup{nlp}}+\varepsilon\right)t<x<\left(2\sqrt{rd}-\varepsilon\right)t}\left(\left|u\left(t,x\right)\right|+\left|v\left(t,x\right)-1\right|\right)=0\text{ for each }\varepsilon\in\left(0,\frac{2\sqrt{rd}-c_{\textup{nlp}}}{2}\right),
\]
\[
\lim\limits _{t\to+\infty}\sup\limits _{\left(2\sqrt{rd}+\varepsilon\right)t<x}\left(\left|u\left(t,x\right)\right|+\left|v\left(t,x\right)\right|\right)=0\text{ for each }\varepsilon>0.
\]
\item Assume $2\sqrt{rd}\geq f\left(c_{\textup{LLW}}\right)$. Then
\[
\lim\limits _{t\to+\infty}\sup\limits _{0\leq x<\left(c_{\textup{LLW}}-\varepsilon\right)t}\left(\left|u\left(t,x\right)-1\right|+\left|v\left(t,x\right)\right|\right)=0\text{ for each }\varepsilon\in\left(0,c_{\textup{LLW}}\right),
\]
\[
\lim\limits _{t\to+\infty}\sup\limits _{\left(c_{\textup{LLW}}+\varepsilon\right)t<x<\left(2\sqrt{rd}-\varepsilon\right)t}\left(\left|u\left(t,x\right)\right|+\left|v\left(t,x\right)-1\right|\right)=0\text{ for each }\varepsilon\in\left(0,\frac{2\sqrt{rd}-c_{\textup{LLW}}}{2}\right),
\]
\[
\lim\limits _{t\to+\infty}\sup\limits _{\left(2\sqrt{rd}+\varepsilon\right)t<x}\left(\left|u\left(t,x\right)\right|+\left|v\left(t,x\right)\right|\right)=0\text{ for each }\varepsilon>0.
\]
\end{enumerate}
\end{thm}

In the first case, $v$ goes extinct. In the second case, $v$ invades
first at speed $2\sqrt{rd}$ and is then replaced by $u$ at speed
$c_{\textup{nlp}}>c_{\textup{LLW}}$. In the third case, $v$ invades first at speed
$2\sqrt{rd}$ and is then replaced by $u$ at speed $c_{\textup{LLW}}$.

Notice that the limits above are chiefly concerned with $x\geq0$. This
is intentional, for the sake of brevity and clarity. In $\left(-\infty,0\right)$,
two behaviors are possible, depending on whether $u_{0}$ is compactly
supported or Heavyside-like. In the former case, all inequalities
above hold with $x$ replaced by $\left|x\right|$ (and this claim
is proved simply by symmetry). In the latter case, $\left(u,v\right)$
converges uniformly to $\left(1,0\right)$ in $\left(-\infty,0\right)$
(and this claim can be proved by a standard comparison argument).

\subsubsection{The set of admissible pairs of speeds for more general initial data}

Define the auxiliary function

\[
\begin{matrix}\lambda_{v}: & \left[2\sqrt{rd},+\infty\right) & \to & \left(0,\sqrt{\frac{r}{d}}\right]\\
 & c & \mapsto & \frac{1}{2d}\left(c-\sqrt{c^{2}-4rd}\right)
\end{matrix}.
\]

\begin{thm}
\label{thm:Nonexistence}  Let $c_{1}\in\left[2\sqrt{rd},+\infty\right)$ and $c_{2}\in\left[c_{\textup{LLW}},c_{1}\right]$. 
Let $\left(u,v\right)$ be a solution of (\ref{eq:competition_diffusion_system}) such that
\[
c_2=\sup\left\{ c>0\ |\ \lim\limits _{t\to+\infty}\sup\limits _{0\leq x\leq ct}\left(\left|u\left(t,x\right)-1\right|+\left|v\left(t,x\right)\right|\right)=0\right\}
\]
and such that at least one of the following two properties holds true:
\begin{enumerate}
\item $x\mapsto v\left(0,x\right)\textup{e}^{\lambda_{v}\left(c_{1}\right)x}$ is bounded in $\mathbb{R}$; or 
\item $c_1$ satisfies
\[
c_1\geq\inf\left\{ c>0\ |\ \lim\limits _{t\to+\infty}\sup\limits _{x\geq ct}\left|v\left(t,x\right)\right|=0\right\}.
\]
\end{enumerate}
Then $f\left( c_2 \right) \leq c_1$. 
\end{thm}

The assumption on $c_2$ basically means that $u$ spreads at speed $c_2$. However, in general, the spreading speed is ill-defined: 
the \textit{minimal spreading speed} of $u$,
\[
\sup\left\{ c>0\ |\ \lim\limits _{t\to+\infty}\sup\limits _{0\leq x\leq ct}\left|u\left(t,x\right)-1\right|=0\right\},
\]
might very well be smaller than its \textit{maximal spreading speed},
\[
\inf\left\{ c>0\ |\ \lim\limits _{t\to+\infty}\sup\limits _{x\geq ct}\left|u\left(t,x\right)\right|=0\right\}.
\]
On this problem, we refer to Hamel--Nadin \cite{Hamel_Nadin_2012}.

The properties $(1)$ and $(2)$ above are more or less equivalent. Indeed, on one hand, $(1)$ directly implies $(2$) by standard
comparison; on the 
other hand, if $(2)$ holds, then for all $\lambda\in\left(0,\lambda_{v}\left(c_1\right)\right)$, 
there exists $T_\lambda$ such that $x\mapsto v\left(T_{\lambda},x\right)\textup{e}^{\lambda x}$ is bounded in $\mathbb{R}$.
However the proof of the latter implication is difficult. In fact, instead of establishing it, we will directly prove the 
result in each case. 
We emphasize that although $(2)$ might be easier to understand in that it directly relates $c_1$ to the spreading of $v$, 
$(1)$ has the advantage of being easier to apply since it only requires knowledge of the initial condition.

In short, this theorem means that if $v$ spreads no faster than $c_1$ and if $u$ spreads at speed $c_2$, then 
$f\left( c_2 \right)\leq c_1$. The next theorem shows the sharpness of this threshold: any 
$c_1 > f\left( c_2 \right)$ can actually be achieved.

\begin{thm}
\label{thm:Generalized_terraces} Let $c_{1}\in\left(2\sqrt{rd},+\infty\right)$
and $c_{2}\in\left(c_{\textup{LLW}},c_{1}\right)$. 

Assume $c_{1}>f\left(c_{2}\right)$.
Then there exists $\left(u_{c_{1},c_{2}},v_{c_{1},c_{2}}\right)\in\mathscr{C}\left(\mathbb{R},\left[0,1\right]^{2}\right)$
such that the solution $\left(u,v\right)$ of (\ref{eq:competition_diffusion_system})
with initial value $\left(u_{0},v_{0}\right)=\left(u_{c_{1},c_{2}},v_{c_{1},c_{2}}\right)$
satisfies
\[
\lim\limits _{t\to+\infty}\sup\limits _{x<\left(c_{2}-\varepsilon\right)t}\left(\left|u\left(t,x\right)-1\right|+\left|v\left(t,x\right)\right|\right)=0\text{ for each }\varepsilon\in\left(0,c_{2}\right),
\]
\[
\lim\limits _{t\to+\infty}\sup\limits _{\left(c_{2}+\varepsilon\right)t<x<\left(c_{1}-\varepsilon\right)t}\left(\left|u\left(t,x\right)\right|+\left|v\left(t,x\right)-1\right|\right)=0\text{ for each }\varepsilon\in\left(0,\frac{c_{1}-c_{2}}{2}\right),
\]
\[
\lim\limits _{t\to+\infty}\sup\limits _{\left(c_{1}+\varepsilon\right)t<x}\left(\left|u\left(t,x\right)\right|+\left|v\left(t,x\right)\right|\right)=0\text{ for each }\varepsilon>0.
\]
\end{thm}

\begin{figure}\subfloat[ $c_{\textup{LLW}}=2\sqrt{1-a}$, $2\sqrt{rd}<f\left(c_{\textup{LLW}}\right)$. In this case \thmref{Compactly_supported}(2) applies with $(c_1,c_2) = (2\sqrt{rd}, c_{\textup{nlp}})$, where $c_{\textup{nlp}} > c_{\textup{LLW}}$.]{\resizebox{.65\hsize}{!}{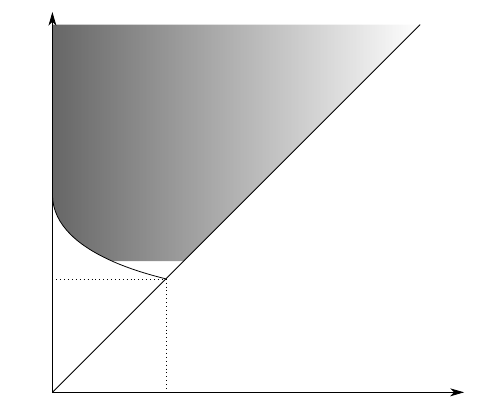}

}

\subfloat[$c_{\textup{LLW}}>2\sqrt{1-a}$, $2\sqrt{rd}>f\left(c_{\textup{LLW}}\right)$. In this case \thmref{Compactly_supported}(3) applies with $(c_1,c_2) = (2\sqrt{rd}, c_{\textup{LLW}})$.]{\resizebox{.65\hsize}{!}{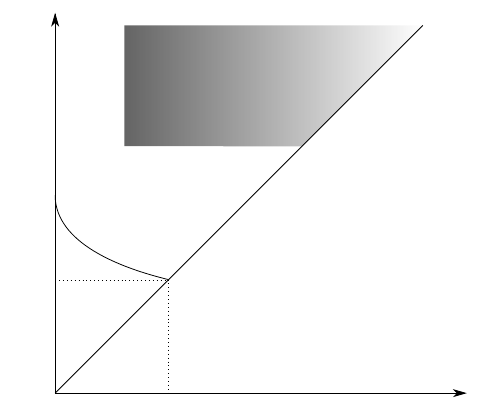}

}\caption{Examples of sets of admissible pairs of speeds $\left(c_{1},c_{2}\right)$}\label{fig:Set_of_speeds}

\end{figure}

Let us point out that this solution $\left(u,v\right)$ is not a 
proper propagating terrace in the sense of Ducrot\textendash Giletti\textendash Matano
\cite{Ducrot_Giletti_Matano}: the locally uniform convergence of
the profiles is missing (as in Carr\`{e}re \cite{Carrere_2017}). 

The fact that the set of admissible speeds is not always the maximal set 
\[
\left\{ \left(c_{1},c_{2}\right)\in\left[2\sqrt{rd},+\infty\right)\times\left[c_{\textup{LLW}},+\infty\right)\ |\ c_{1}>c_{2}\right\} .
\]
settles completely a question raised by the first author \cite{Girardin_2017}.

\subsubsection{The super-solutions and sub-solutions}\label{subsec:1.1.3}

The preceding theorems will be proved thanks to the following
three propositions, which are of independent interest and concern
existence results for super-solutions and sub-solutions (the precise
definition of these will be recalled in the next section).

Let 
\[
\begin{matrix}\lambda: & \left[2\sqrt{1-a},+\infty\right) & \to & \left(0,\sqrt{1-a}\right]\\
 & c & \mapsto & \frac{1}{2}\left(c-\sqrt{c^{2}-4\left(1-a\right)}\right)
\end{matrix},
\]
\[
\Lambda:\left(c,\tilde{c}\right)\mapsto\frac{1}{2}\left(\tilde{c}-\sqrt{\tilde{c}^{2}-4\left(\lambda\left(c\right)\left(\tilde{c}-c\right)+1\right)}\right).
\]
 The domain of $\Lambda$ is the set of all $\left(c,\tilde{c}\right)$
such that $c\geq2\sqrt{1-a}$ and $\tilde{c}\geq\max\left(c,f\left(c\right)\right)$. 

For all $c\geq c_{\textup{LLW}}$ and $\tilde{c}\geq\max\left(c,f\left(c\right)\right)$,
$\overline{w_{c,\tilde{c}}}$ denotes the function
\[
\overline{w_{c,\tilde{c}}}:\left(t,x\right)\mapsto\textup{e}^{-\lambda\left(c\right)\left(\tilde{c}-c\right)t}\textup{e}^{-\Lambda\left(c,\tilde{c}\right)\left(x-\tilde{c}t\right)}.
\]
\begin{prop}
\label{prop:Sub-solution_nonexistence} Let $c_{1}\geq2\sqrt{rd}$,
$c_{2}\geq c_{\textup{LLW}}$ and assume $c_{2}<c_{1}<f\left(c_{2}\right)$. 

There exist $c>c_{2}$, $\tilde{c}\in\left(c_{1},f\left(c\right)\right)$,
$L>0$ and $\delta^{\star}>0$ such that, for all $\delta\in\left(0,\delta^{\star}\right)$, all 
$\kappa\in\left(0,\min\left(\frac{1-a}{2},\frac{\delta}{2}\right)\right)$ and
all $\zeta>L$, 
there exists $R_{\delta}>0$ and a sub-solution 
$\left(\underline{u_{\delta,\zeta,\kappa}},\overline{v_{\delta,\zeta}}\right)$ 
of (\ref{eq:competition_diffusion_system}) 
satisfying the following properties:
\begin{enumerate}
\item $\underline{u_{\delta,\zeta,\kappa}}\left(0,x\right)\leq1-a$ for
all $x\in\mathbb{R}$;
\item the support of $x\mapsto\underline{u_{\delta,\zeta,\kappa}}\left(0,x\right)$
is included in $\left[0,L+\zeta+2R_{\delta}\right]$;
\item $\underline{u_{\delta,\zeta,\kappa}}\left(0,x\right)\leq\kappa$ for
all $x\in\left[L,L+\zeta+2R_{\delta}\right]$;
\item there exists $X>0$ such that $\underline{u_{\delta,\zeta,\kappa}}$ satisfies
\[
\partial_{t}\underline{u_{\delta,\zeta,\kappa}}-\partial_{xx}\underline{u_{\delta,\zeta,\kappa}}\leq\left(1-\delta\right)\underline{u_{\delta,\zeta,\kappa}}\text{ in }\left\{\left(t,x\right)\in\left[0,+\infty\right)\times\mathbb{R}\ |\ x>X+\tilde{c}t\right\}.
\]
\item there exists $C_{\delta}>0$ depending only on $\delta$ such that
\[
\overline{v}_{\delta,\zeta}\left(0,x\right)\geq\min\left(1,C_{\delta}\textup{e}^{-\lambda_{v}\left(\tilde{c}\right)\left(x-\zeta\right)}\right)\text{ for all }x\in\mathbb{R};
\]
\item the following spreading property holds true:
\[
\lim\limits _{t\to+\infty}\sup\limits _{L\leq x<\left(c-\varepsilon\right)t}\left|\underline{u_{\delta,\zeta,\kappa}}\left(t,x\right)-\frac{1-a}{2}\right|=0\text{ for all }\varepsilon\in\left(0,c\right).
\]
\end{enumerate}
\end{prop}

\begin{prop}
\label{prop:Super-solution_compactly_supported} Let $c_{2}\in\left(\max\left(c_{\textup{LLW}},f^{-1}\left(2\sqrt{rd}\right)\right),2\right)$. 

There exists $\delta^{\star}>0$ and $\left(c_{1}^{\delta},c_{2}^{\delta}\right)_{\delta\in\left(0,\delta^{\star}\right)}$
such that 
\[
c_{2}<c_{2}^{\delta}<c_{1}^{\delta}<2\sqrt{rd}\text{ for all }\delta\in\left(0,\delta^{\star}\right),
\]
\[
\lim_{\delta\to0}\left(c_{1}^{\delta},c_{2}^{\delta}\right)=\left(2\sqrt{rd},c_{2}\right),
\]
 and, for all $\delta\in\left(0,\delta^{\star}\right)$, there exists
a super-solution  $\left(\overline{u_{\delta}},\underline{v_{\delta}}\right)$ of (\ref{eq:competition_diffusion_system}) satisfying the following properties:
\begin{enumerate}
\item there exists $y_0\in\mathbb{R}$ such that, for all $y\geq y_0$ and $t\geq0$, 
\[
\overline{u_{\delta}}\left(0,x-y-\frac{\left(\Lambda\left(c_{2},2\sqrt{rd}\right)\right)^{2}+1}{\Lambda\left(c_{2},2\sqrt{rd}\right)}t\right)\geq\min\left(1,\overline{w_{c_{2},2\sqrt{rd}}}\left(t,x\right)\right)\text{ for all }x\in\mathbb{R};
\]
\item $x\mapsto\underline{v_{\delta}}\left(0,x\right)$ is compactly supported;
\item $\underline{v_{\delta}\left(0,x\right)}\leq1-\delta$ for all $x\in\mathbb{R}$;
\item the following spreading property holds true:
\[
\lim\limits _{t\to+\infty}\sup\limits _{\left(c_{2}^{\delta}+\varepsilon\right)t<x<\left(c_{1}^{\delta}-\varepsilon\right)t}\left(\left|\overline{u_{\delta}}\left(t,x\right)\right|+\left|\underline{v_{\delta}}\left(t,x\right)-\left(1-2\delta\right)\right|\right)=0\text{ for all }\varepsilon\in\left(0,\frac{c_{1}^{\delta}-c_{2}^{\delta}}{2}\right).
\]
\end{enumerate}
\end{prop}

\begin{prop}
\label{prop:Super-sub-solution_terraces} Let $c_{1}>2\sqrt{rd}$,
$c_{2}>c_{\textup{LLW}}$ and assume $c_{1}>\max\left(c_{2},f\left(c_{2}\right)\right)$. 

There exists $\delta^{\star}>0$ and $\left(c_{2}^{\delta}\right)_{\delta\in\left(0,\delta^{\star}\right)}$
such that 
\[
c_{2}^{\delta}>c_{2}\text{ for all }\delta\in\left(0,\delta^{\star}\right),
\]
\[
\lim_{\delta\to0}c_{2}^{\delta}=c_{2},
\]
 and, for all $\delta\in\left(0,\delta^{\star}\right)$, there exists
a super-solution of (\ref{eq:competition_diffusion_system}) $\left(\overline{u_{\delta}},\underline{v_{\delta}}\right)$
and a sub-solution of (\ref{eq:competition_diffusion_system}) $\left(\underline{u_{\delta}},\overline{v_{\delta}}\right)$
satisfying the following properties:
\begin{enumerate}
\item there exists $y_0\in\mathbb{R}$ such that, for all $y\geq y_0$ and $t\geq0$, 
\[
\overline{u_{\delta}}\left(0,x-y-\frac{\left(\Lambda\left(c_{2},c_{1}\right)\right)^{2}+1}{\Lambda\left(c_{2},c_{1}\right)}t\right)\geq\min\left(1,\overline{w_{c_{2},c_{1}}}\left(t,x\right)\right)\text{ for all }x\in\mathbb{R};
\]
\item the support of $\underline{v_{\delta}}$ is a right half-line and there exists $\left(y,z\right)\in\mathbb{R}^{2}$
such that
\[
\frac{1}{2}\leq\textup{e}^{\lambda_{v}\left(c_{1}\right)\left(x-c_{1}t\right)}\underline{v_{\delta}}\left(0,x-y\right)\leq1\text{ for all }t\geq 0\text{ and }x\geq z;
\]
\item $\underline{u_{\delta}}\left(0,x\right)\leq\overline{u_{\delta}}\left(0,x\right)$
and $\underline{v_{\delta}}\left(0,x\right)\leq\overline{v_{\delta}}\left(0,x\right)$
for all $x\in\mathbb{R}$;
\item the following spreading properties hold true:
\[
\lim\limits _{t\to+\infty}\sup\limits _{x<\left(c_{2}-\varepsilon\right)t}\left|\underline{u_{\delta}}\left(t,x\right)-\left(1-a\right)\right|=0\text{ for all }\varepsilon\in\left(0,c_{2}\right),
\]
\[
\lim\limits _{t\to+\infty}\sup\limits _{\left(c_{2}^{\delta}+\varepsilon\right)t<x<\left(c_{1}-\varepsilon\right)t}\left(\left|\overline{u_{\delta}}\left(t,x\right)\right|+\left|\underline{v_{\delta}}\left(t,x\right)-\left(1-2\delta\right)\right|\right)=0\text{ for all }\varepsilon\in\left(0,\frac{c_{1}-c_{2}^{\delta}}{2}\right).
\]
\[
\lim\limits _{t\to+\infty}\sup\limits _{\left(c_{1}+\varepsilon\right)t<x}\left(\left|\overline{u_{\delta}}\left(t,x\right)\right|+\left|\overline{v_{\delta}}\left(t,x\right)\right|\right)=0\text{ for all }\varepsilon>0,
\]
\end{enumerate}
\end{prop}

The forms of the super- and sub-solutions of \propref{Sub-solution_nonexistence} and  \propref{Super-solution_compactly_supported} are illustrated in \figref{1.4} and \figref{1.5} respectively. Those of \propref{Super-sub-solution_terraces} are illustrated in  \figref{1.6a} and \figref{1.6b}.

\subsection{The quantities $f\left(c_2\right)$, $\lambda\left(c_2\right)$, $\Lambda\left(c_2,c_1\right)$}

Let us explain by a heuristic argument how these three quantities come out naturally in the problem 
and what is their ecological meaning. 

Assume that $v$ invades the uninhabited territory at some speed $c_1\geq 2\sqrt{rd}$ and that $u$ chases $v$ at some speed 
$c_2\in\left[c_{\textup{LLW}},c_1\right)$. In the area where $v\simeq 1$, $u$ looks like the exponential tail of the monostable 
traveling wave connecting $\left(0,1\right)$ to $\left(1,0\right)$ at speed $c_2$, that is 
\[
u\left(t,x\right)\simeq \textup{e}^{-\lambda\left(c_2\right)\left(x-c_2 t\right)}.
\]
Accordingly, in a neighborhood of $x=\tilde{c}t$ with $\tilde{c}\in\left(c_2,c_1\right)$, we can observe non-negligible 
quantities only if we consider the rescaled function 
\[
w:\left(t,x\right)\mapsto u\left(t,x\right)\textup{e}^{\lambda\left(c_2\right)\left(x-c_2 t\right)}
\]
instead of $u$ itself.

Yet, in a neighborhood of $x=\tilde{c}t$ with $\tilde{c}>c_1$, where $\left(u,v\right)\simeq\left(0,0\right)$, $w$ satisfies
at the first order 
\[
\partial_{t}w - \partial_{xx}w=\left(1+\lambda\left(c_2\right)\left(\tilde{c}-c_2\right)\right)w
\]
whence the exponential ansatz $w\left(t,x\right)=\textup{e}^{-\Lambda\left(x-\tilde{c} t\right)}$
leads to the equation
\[
\Lambda^2-\tilde{c}\Lambda+\left(1+\lambda\left(c_2\right)\left(\tilde{c}-c_2\right)\right)=0.
\]
The minimal zero of this equation being precisely 
\[
\Lambda\left(c_2,\tilde{c}\right)=\frac{1}{2}\left(\tilde{c}-\sqrt{\tilde{c}^{2}-4\left(\lambda\left(c_2\right)\left(\tilde{c}-c_2\right)+1\right)}\right),
\]
we deduce then that $\tilde{c}$ has to satisfy
\[
\tilde{c}^{2}-4\left(\lambda\left(c_2\right)\left(\tilde{c}-c_2\right)+1\right)\geq 0
\]
that is $\tilde{c}\geq f\left(c_2\right)$. Passing to the limit $\tilde{c}\to c_1$, we find indeed $c_1\geq f\left(c_2\right)$.

\subsection{Organization of the paper}

In Section 2, we recall the comparison principle for (\ref{eq:competition_diffusion_system})
and define super-solutions and sub-solutions.

In Section 3, we prove \thmref{Compactly_supported}, \thmref{Nonexistence} and \thmref{Generalized_terraces}
assuming \propref{Sub-solution_nonexistence}, \propref{Super-solution_compactly_supported}
and \propref{Super-sub-solution_terraces} are true. 

In Section 4, we prove \propref{Sub-solution_nonexistence}, \propref{Super-solution_compactly_supported}
and \propref{Super-sub-solution_terraces}. These constructions are
rather delicate and require several objects and preliminary lemmas, which we summarize in a table at the beginning of \subsecref{4.1}.

In Section 5, we comment on the results and provide some future perspectives.

\section{Competitive comparison principle}

\subsection{Competitive comparison principle}

In what follows, vectors in $\mathbb{R}^{2}$ are always understood
as column vectors.

We define the competitive ordering $\preceq$ in $\mathbb{R}^{2}$ 
as follows: for all $\left(u_{1},v_{1}\right)\in\mathbb{R}^{2}$,
$\left(u_{2},v_{2}\right)\in\mathbb{R}^{2}$, 
\[
\left(u_{1},v_{1}\right)\preceq\left(u_{2},v_{2}\right)\text{ if }u_{1}\leq u_{2}\text{ and }v_{1}\geq v_{2}.
\]
 The strict competitive ordering $\prec$ is defined by
\[
\left(u_{1},v_{1}\right)\prec\left(u_{2},v_{2}\right)\text{ if }u_{1}<u_{2}\text{ and }v_{1}>v_{2}.
\]

We define also the operators
\[
P:\left(u,v\right)\mapsto\partial_{t}\left(u,v\right)-\diag\left(1,d\right)\partial_{xx}\left(u,v\right),
\]
\[
F:\left(u,v\right)\mapsto\left(\begin{matrix}u\left(1-u-av\right)\\
rv\left(1-v-bu\right)
\end{matrix}\right).
\]
With these notations, (\ref{eq:competition_diffusion_system}) can be written as 
\[
\left\{ \begin{matrix}P\left(u,v\right)=F\left(u,v\right) & \text{in }\left(0,+\infty\right)\times\mathbb{R}\\
\left(u,v\right)\left(0,x\right)=\left(u_{0},v_{0}\right)\left(x\right) & \text{for all }x\in\mathbb{R}
\end{matrix}\right..
\]
\begin{defn}
A classical super-solution of (\ref{eq:competition_diffusion_system})
is a pair 
\[
\left(\overline{u},\underline{v}\right)\in\mathscr{C}^{1}\left(\left(0,+\infty\right),\mathscr{C}^{2}\left(\mathbb{R},\left[0,1\right]^{2}\right)\right)\cap\mathscr{C}\left(\left[0,+\infty\right)\times\mathbb{R},\left[0,1\right]^{2}\right)
\]
 satisfying 
\[
P\left(\overline{u},\underline{v}\right)\succeq F\left(\overline{u},\underline{v}\right)\text{ in }\left(0,+\infty\right)\times\mathbb{R}.
\]

A classical sub-solution of (\ref{eq:competition_diffusion_system})
is a pair 
\[
\left(\underline{u},\overline{v}\right)\in\mathscr{C}^{1}\left(\left(0,+\infty\right),\mathscr{C}^{2}\left(\mathbb{R},\left[0,1\right]^{2}\right)\right)\cap\mathscr{C}\left(\left[0,+\infty\right)\times\mathbb{R},\left[0,1\right]^{2}\right)
\]
 satisfying 
\[
P\left(\underline{u},\overline{v}\right)\preceq F\left(\underline{u},\overline{v}\right)\text{ in }\left(0,+\infty\right)\times\mathbb{R}.
\]
 The unbounded domain $\left(0,+\infty\right)\times\mathbb{R}$ can
be replaced in the above definition by a bounded parabolic cylinder
$\left(0,T\right)\times\left(-R,R\right)$. In such a case, the required
regularity is $\mathscr{C}^{1}\left(\left(0,T\right),\mathscr{C}^{2}\left(\left(-R,R\right),\left[0,1\right]^{2}\right)\right)\cap\mathscr{C}\left(\left[0,T\right]\times\left[-R,R\right],\left[0,1\right]^{2}\right)$.
\end{defn}

We also recall that it is possible to extend the theory of super-
and sub-solutions to Sobolev spaces. The full extension is outside
of the scope of this reminder, however a very partial extension will
be necessary later on. More precisely, we will use the following notion
of Lipschitz-continuous super- and sub-solutions, acting against smooth compactly supported
test functions.
In what follows, $\mathscr{C}^{0,1}$ and $\mathscr{D}$ denote as usual 
the sets of Lipschitz-continuous and smooth compactly supported functions respectively 
and $\circ$ denotes the Hadamard product
$\left(u_{1},v_{1}\right)\circ\left(u_{2},v_{2}\right)=\left(u_{1}u_{2},v_{1}v_{2}\right)$.

\begin{defn}
A generalized super-solution of (\ref{eq:competition_diffusion_system})
is a pair 
\[
\left(\overline{u},\underline{v}\right)\in\mathscr{C}^{0,1}\left(\left(0,+\infty\right)\times\mathbb{R},\left[0,1\right]^{2}\right)
\]
 satisfying, for all $\left(U,V\right)\in\mathscr{D}\left(\left(0,+\infty\right)\times\mathbb{R},\left[0,1\right]^{2}\right)$,
\[
\int\partial_{t}\left(\overline{u},\underline{v}\right)\circ\left(U,V\right)+\diag\left(1,d\right)\partial_{x}\left(\overline{u},\underline{v}\right)\circ\partial_{x}\left(U,V\right)\succeq\int F\left(\overline{u},\underline{v}\right)\circ\left(U,V\right).
\]

A generalized sub-solution of (\ref{eq:competition_diffusion_system})
is a pair 
\[
\left(\underline{u},\overline{v}\right)\in\mathscr{C}^{0,1}\left(\left(0,+\infty\right)\times\mathbb{R},\left[0,1\right]^{2}\right)
\]
 satisfying, for all $\left(U,V\right)\in\mathscr{D}\left(\left(0,+\infty\right)\times\mathbb{R},\left[0,1\right]^{2}\right)$,
\[
\int\partial_{t}\left(\underline{u},\overline{v}\right)\circ\left(U,V\right)+\diag\left(1,d\right)\partial_{x}\left(\underline{u},\overline{v}\right)\circ\partial_{x}\left(U,V\right)\preceq\int F\left(\underline{u},\overline{v}\right)\circ\left(U,V\right).
\]
\end{defn}

Again, the unbounded domain $\left(0,+\infty\right)\times\mathbb{R}$
can be replaced by a bounded parabolic cylinder $\left(0,T\right)\times\left(-R,R\right)$.
The following important theorem, that will be used repeatedly thereafter,
actually uses the local definition.
\begin{thm}
\label{thm:Generalized_super-solutions}Let $R>0$, $T>0$, $Q=\left(0,T\right)\times\left(-R,R\right)$
and
\[
\left(\overline{u}_1,\overline{u}_2,\underline{v}_1,\underline{v}_2\right)\in\mathscr{C}^{1}\left(\left[0,T\right],\mathscr{C}^{2}\left(\left[-R,R\right],\left[0,1\right]^{4}\right)\right)\cap\mathscr{C}\left(\left[0,T\right]\times\left[-R,R\right],\left[0,1\right]^{4}\right).
\]
\begin{enumerate}
\item Assume that $\left(\overline{u}_1,\underline{v}_1\right)$ and $\left(\overline{u}_1,\underline{v}_2\right)$ are both classical
super-solutions in $Q$. Then 
$\left(\overline{u}_1, \max\left(\underline{v}_1,\underline{v}_2\right)\right)$ is a generalized super-solution in $Q$. 
\item Assume that $\left(\overline{u}_1,\underline{v}_1\right)$ and $\left(\overline{u}_2,\underline{v}_1\right)$ are both classical
super-solutions in $Q$. Then
$\left(\min\left(\overline{u}_1,\overline{u}_2\right),\underline{v}_1\right)$ is a generalized super-solution in $Q$. 
\end{enumerate}
\end{thm}

\begin{rem*}
We state this theorem in a bounded parabolic cylinder in order to
be able to construct later on more complex super- and sub-solutions,
for instance super-solutions $\left(\overline{u},\underline{v}\right)$
with $\overline{u}$ of the form
\[
\overline{u}\left(t,x\right)=\left\{ \begin{matrix}\overline{u}_{1}\left(t,x\right) & \text{if }x<x\left(t\right)\\
\overline{u}_{2}\left(t,x\right) & \text{if }x\in\left[x\left(t\right),y\left(t\right)\right]\\
\overline{u}_{3}\left(t,x\right) & \text{if }x>y\left(t\right)
\end{matrix}\right.,
\]
where $x\left(t\right)<y\left(t\right)$ and $\overline{u}_{1}$,
$\overline{u}_{2}$ and $\overline{u}_{3}$ are such that $\overline{u}_{1}\left(t,x\right)\leq\overline{u}_{2}\left(t,x\right)$
if $x<x\left(t\right)$, $\overline{u}_{2}\left(t,x\right)\leq\overline{u}_{1}\left(t,x\right)$
in a right-sided neighborhood of $x\left(t\right)$, $\overline{u}_{2}\left(t,x\right)\leq\overline{u}_{3}\left(t,x\right)$
in a left-sided neighborhood of $y\left(t\right)$ and $\overline{u}_{3}\left(t,x\right)\leq\overline{u}_{2}\left(t,x\right)$
if $x>y\left(t\right)$. Although we do not have any global information
on $\overline{u}_{1}-\overline{u}_{2}$, $\overline{u}_{1}-\overline{u}_{3}$
and $\overline{u}_{2}-\overline{u}_{3}$, the local theorem shows
that the construction is still valid.
\end{rem*}
\begin{proof}
Since the second statement is proved similarly, we only prove the
first one. 

For simplicity, we only consider the special case where 
$\Gamma=\left(\underline{v}_1-\underline{v}_2\right)^{-1}\left(\left\{ 0\right\} \right)$ 
is a smooth hypersurface, which is always satisfied for our purposes. 
A proof that does not require such a regularity assumption can be found for instance 
in \cite{Sattinger_1972}.

Define $\underline{v}=\max\left(\underline{v}_1,\underline{v}_2\right)$ and let 
$\left(U,V\right)\in\mathscr{D}\left(\overline{Q},\left[0,1\right]^{2}\right)$.
On one hand, 
\[
\partial_{t}\overline{u}_1-\partial_{xx}\overline{u}_1\geq F_{1}\left(\overline{u}_1,\underline{v}\right)
\]
 is satisfied in the classical sense (using for instance $-a\overline{u}_1\underline{v}_1\geq-a\overline{u}_1\underline{v}$).
On the other hand, 
we have assumed that $\Gamma=\left(\underline{v}_1-\underline{v}_2\right)^{-1}\left(\left\{ 0\right\} \right)$
is a smooth hypersurface, so that we may integrate by parts. Denoting 
$Q_{1}=\left(\underline{v}_1-\underline{v}_2\right)^{-1}\left(\left[0,1\right]\right)$,
$Q_{2}=\left(\underline{v}_2-\underline{v}_1\right)^{-1}\left(\left[0,1\right]\right)$, $\nu$
the outward unit normal to $Q_{1}$, we find $\Gamma=\partial Q_{1}\backslash\partial Q=\partial Q_{2}\backslash\partial Q$
and $\displaystyle 
\left(\partial_{x}\underline{v}_1-\partial_{x}\underline{v}_2\right)\nu\leq0$ on $\Gamma$, 
whence
\begin{align*}
&\quad \int_{Q}\partial_{t}\underline{v}V+d\partial_{x}\underline{v}\partial_{x}V \\
&=\int_{Q_{1}}\partial_{t}\underline{v}_1V+d\partial_{x}\underline{v}_1\partial_{x}V+\int_{Q_{2}}\partial_{t}\underline{v}_2V+d\partial_{x}\underline{v}_2\partial_{x}V\\
 & =\int_{Q_{1}}\left(\partial_{t}\underline{v}_1-d\partial_{xx}\underline{v}_1\right)V+\int_{\partial Q_{1}}\partial_{x}\underline{v}_1V\nu+\int_{Q_{2}}\left(\partial_{t}\underline{v}_2-d\partial_{xx}\underline{v}_2\right)V+\int_{\partial Q_{2}}\partial_{x}\underline{v}_1V\left(-\nu\right)\\
 & \leq\int_{Q_{1}}F_{2}\left(\overline{u}_1,\underline{v}_1\right)V+\int_{Q_{2}}F_{2}\left(\overline{u}_1,\underline{v}_2\right)V+\int_{\Gamma}\left(\partial_{x}\underline{v}_1-\partial_{x}\underline{v}_2\right)V\nu\\
 & \leq\int_{Q}F_{2}\left(\overline{u}_1,\underline{v}\right)V.
\end{align*}
This completes the proof. 
\end{proof}
An inversion of the roles yields a similar statement on sub-solutions. 
\begin{thm}
\label{thm:Generalized_sub-solutions}Let $R>0$, $T>0$, $Q=\left(0,T\right)\times\left(-R,R\right)$
and
\[
\left(\underline{u}_1,\underline{u}_2,\overline{v}_1,\overline{v}_2\right)\in\mathscr{C}^{1}\left(\left[0,T\right],\mathscr{C}^{2}\left(\left[-R,R\right],\left[0,1\right]^{4}\right)\right)\cap\mathscr{C}\left(\left[0,T\right]\times\left[-R,R\right],\left[0,1\right]^{4}\right).
\]
\begin{enumerate}
\item Assume that $\left(\underline{u}_1,\overline{v}_1\right)$ and $\left(\underline{u}_1,\overline{v}_2\right)$ are both classical
sub-solutions in $Q$. Then 
$\left(\underline{u}_1, \min\left(\overline{v}_1,\overline{v}_2\right)\right)$ is a generalized sub-solution in $Q$. 
\item Assume that $\left(\underline{u}_1,\overline{v}_1\right)$ and $\left(\underline{u}_2,\overline{v}_1\right)$ are both classical
sub-solutions in $Q$. Then
$\left(\max\left(\underline{u}_1,\underline{u}_2\right),\overline{v}_1\right)$ is a generalized sub-solution in $Q$. 
\end{enumerate}
\end{thm}

Since a classical super- or sub-solution is \textit{a fortiori} a
generalized super- or sub-solution respectively, from now on, we omit
the adjectives classical and generalized and always have in mind the
generalized notion.

The comparison principle for (\ref{eq:competition_diffusion_system}),
directly derived from the comparison principle for cooperative systems
(see Protter\textendash Weinberger \cite{Protter_Weinberger}) via
the transformation $w=1-v$, reads as follows.
\begin{thm}
\label{thm:Competitive_comparison_principle}Let $\left(\overline{u},\underline{v}\right)$
and $\left(\underline{v},\overline{u}\right)$ be respectively a super-solution
and a sub-solution of (\ref{eq:competition_diffusion_system}). Assume
that 
\[
\left(\overline{u},\underline{v}\right)\left(0,x\right)\succeq\left(\underline{u},\overline{v}\right)\left(0,x\right)\text{ for all }x\in\mathbb{R}.
\]

Then 
\[
\left(\overline{u},\underline{v}\right)\succeq\left(\underline{u},\overline{v}\right)\text{ in }\left[0,+\infty\right)\times\mathbb{R}.
\]

Furthermore, if there exists $\left(T,x\right)\in\left(0,+\infty\right)\times\mathbb{R}$
such that $\overline{u}\left(T,x\right)=\underline{u}\left(T,x\right)$
or $\underline{v}\left(T,x\right)=\overline{v}\left(T,x\right)$,
then 
\[
\left(\overline{u},\underline{v}\right)=\left(\underline{u},\overline{v}\right)\text{ in }\left[0,T\right]\times\mathbb{R}.
\]
\end{thm}

In other words, $\left(\overline{u},\underline{v}\right)\succ\left(\underline{u},\overline{v}\right)$
holds at $t=0$ if and only if it holds at all $t\geq0$.

Finally, we recall an important existence--comparison result that will be used
later on.
\begin{thm}
\label{thm:Existence_by_comparison}Let $\left(\overline{u},\underline{v}\right)$
and $\left(\underline{v},\overline{u}\right)$ be respectively a super-solution
and a sub-solution of (\ref{eq:competition_diffusion_system}). Assume
that for some $(u_0,v_0) \in C(\mathbb{R}, [0,1]^2)$ we have
$$
(\overline{u}, \underline{v})(0,x) \succeq (u_0,v_0)(x)\succeq (\underline{u}, \overline{v})(0,x) \quad \text{ for all }x \in \mathbb{R},
$$
then the solution $\left(u,v\right)$ of (\ref{eq:competition_diffusion_system}) with initial data $(u_0,v_0)$ satisfies
\[
\left(\overline{u},\underline{v}\right)\succeq\left(u,v\right)\succeq\left(\underline{u},\overline{v}\right)\text{ in }\left[0,+\infty\right)\times\mathbb{R}.
\]
\end{thm}

\section{Proofs of \thmref{Compactly_supported}, \thmref{Nonexistence} and \thmref{Generalized_terraces}}

In this section, we assume \propref{Sub-solution_nonexistence}, \propref{Super-solution_compactly_supported}
and \propref{Super-sub-solution_terraces} are true.

\subsection{Proof of \thmref{Nonexistence} }\label{subsec:3.1}
\begin{proof}
Let $c_{1}\geq2\sqrt{rd}$ and $c_{2}\geq c_{\textup{LLW}}$ such that $c_{1}\geq c_{2}$
and $c_{1}<f\left(c_{2}\right)$. 

First, we consider the case where $x\mapsto v_{0}\left(x\right)\textup{e}^{\lambda_{v}\left(c_{1}\right)x}$ is globally bounded.
By contradiction, assume the existence of a solution $\left(u,v\right)$ such that both the boundedness of
$x\mapsto v_{0}\left(x\right)\textup{e}^{\lambda_{v}\left(c_{1}\right)x}$
and the equality 
\[
c_{2}=\sup\left\{ c>0\ |\ \lim\limits _{t\to+\infty}\sup\limits _{0\leq x\leq ct}\left(\left|u\left(t,x\right)-1\right|+\left|v\left(t,x\right)\right|\right)=0\right\} 
\]
are true. 

Define $c$, $\tilde{c}$, $\delta^{\star}$, $\delta=\frac{\delta^{\star}}{2}$,
$R_{\delta}$, $L$, $\underline{u_{\delta,\zeta,\kappa}}$, $\overline{v_{\delta,\eta}}$ as in \propref{Sub-solution_nonexistence}. Note that $c>c_2$, $\tilde{c}>c_1$.

In view of the equality satisfied by $c_{2}$, there exists $T\geq\frac{2L}{c_{2}}$
such that, for all $x\in\left[0,\frac{c_{2}}{2}T\right]$, $u\left(T,x\right)\geq1-\frac{a}{2}$.

We claim that 
$\left(t,x\right)\mapsto v\left(t,x\right)\textup{e}^{\lambda_{v}\left(c_{1}\right)\left(x-c_{1}t\right)}$
is globally bounded in $\left[0,+\infty\right)\times\mathbb{R}$. To see this, it suffices to observe that, by definition of $\lambda_v(c_1)$, $C e^{-\lambda_v(c_1)(x-c_1t)}$ is a supersolution of the equation of $v$ for any constant $C>0$. Hence standard comparison implies that
$$
v(t,x) \leq \left(\sup_{x \in \mathbb{R}} v(0,x) e^{\lambda_v(c_1)x}\right)e^{-\lambda_v(c_1)(x-c_1 t)}.
$$

Since $c_1 < \tilde{c}$ and $\lambda_v$ is decreasing, we have $\lambda_{v}\left(c_{1}\right)>\lambda_{v}\left(\tilde{c}\right)$.
Hence, there exists $\zeta>L$ such that
\[
v\left(T,x\right)\leq\overline{v_{\delta,\zeta}}\left(0,x\right)\text{ for all }x\in\mathbb{R}.
\]

Now, we fix 
\[
\kappa=\frac{1}{2}\min\left(\min\left(\frac{1-a}{2},\frac{\delta}{2}\right),\min_{x\in\left[L,L+\zeta+2R_{\delta}\right]}u\left(T,x\right)\right).
\]
It follows that 
\[
u(T,x) \geq 
\left\{
\begin{array}{ll}
1-a  &\text{ for }x \in [0,L],\\
\kappa &\text{ for }x \in (L, L+\zeta+ 2R_\delta],\\
0 &\text{ for  }x \in \mathbb{R} \setminus [0, L+\zeta+ 2R_\delta],
\end{array}
\right.
\]
whence $u(T,x) \geq \underline{u_{\delta,\zeta,\kappa}}\left(0,x\right)$ for $x \in \mathbb{R}$. Then 
\[
\left(\underline{u},\overline{v}\right):\left(t,x\right)\mapsto\left(\underline{u_{\delta,\zeta,\kappa}}\left(t-T,x\right),\overline{v}\left(t-T,x\right)\right)
\]
 is a sub-solution of (\ref{eq:competition_diffusion_system}) which
satisfies $\left(\underline{u},\overline{v}\right)\preceq\left(u,v\right)$
at $t=T$, whence by the comparison principle of \thmref{Competitive_comparison_principle}
it satisfies the same inequality at any time $t\geq T$. 

Now, due to the spreading property satisfied by $\underline{u}$,
for all $\varepsilon\in\left(0,c\right)$, there exists $T_{\varepsilon}\geq T$
such that, for all $t\geq T_{\varepsilon}$, 
\[
\inf_{L\leq x<\left(c-\varepsilon\right)t}u\left(t,x\right)\geq\frac{1-a}{4}.
\]

Assume now the existence of sequences $\left(t_n\right)_{n\in\mathbb{N}}$, 
$\left(x_n\right)_{n\in\mathbb{N}}$ and of $\hat{c} \in (0,c-\varepsilon)$ such that, 
as $n\to+\infty$, $t_n\to+\infty$, $\frac{x_n}{t_n}\to \hat{c}$ and 
\[
    \limsup \left( \left|u\left( t_n,x_n \right)-1\right| + \left|v\left(t_n, x_n\right)\right|\right) >0.
\]

Denote $\left(c_n\right)_{n\in\mathbb{N}}=\left(\frac{x_n}{t_n}\right)_{n\in\mathbb{N}}$.

For all $n\in\mathbb{N}$, define $\tau_n=\frac{c_n}{\hat{c}}t_n=\frac{x_n}{\hat{c}}$ and
\[
    \left( u_n,v_n \right):\left(t,x\right)\mapsto \left(u,v\right)\left(t + \frac{\hat{c}}{c_n}\tau_n, x + \hat{c} \tau_n\right).
\]

By classical parabolic estimates (see Lieberman \cite{Lieberman_2005}),
$\left(\left(u_n,v_n\right)\right)_{n\in\mathbb{N}}$ converges up to
a diagonal extraction in $\mathscr{C}_{loc}\left(\mathbb{R}^2,\left[0,1\right]\right)$
to a limit $\left(u_\infty,v_\infty\right)$ which is an entire solution of 
(\ref{eq:competition_diffusion_system}) and satisfies $u_\infty\geq\frac{1-a}{4}$.

By comparison of $\left(u_\infty,v_\infty\right)$ with the spatially homogeneous
sub-solution $\left(\underline{U},\overline{V}\right)$ satisfying, for any arbitrary 
$t_0\in\mathbb{R}$, the system 
\[
\left\{ \begin{matrix}\underline{U}'\left(t\right)=\underline{U}\left(t\right)\left(1-\underline{U}\left(t\right)-a\overline{V}\left(t\right)\right) & \text{for all }t\in\left(T_{\varepsilon},+\infty\right)\\
\overline{V}'\left(t\right)=r\overline{V}\left(t\right)\left(1-\overline{V}\left(t\right)-b\underline{U}\left(t\right)\right) & \text{for all }t\in\left(T_{\varepsilon},+\infty\right)\\
\underline{U}\left(t_0\right)=\frac{1-a}{4}\\
\overline{V}\left(t_0\right)=1
\end{matrix}\right.,
\]
whose convergence to $\left(1,0\right)$ is well-known, we find 
$\left( u_\infty,v_\infty \right)=\left( 1,0 \right)$, which directly contradicts the existence
of $\left( t_n \right)$, $\left( x_n \right)$ and $\left( c_n \right)$. 

Therefore
\[
    \lim\limits _{t\to+\infty}\sup\limits _{\varepsilon t\leq x<\left(c-2\varepsilon\right)t}\left(\left|u\left(t,x\right)-1\right|+\left|v\left(t,x\right)\right|\right)=0\text{ for all }\varepsilon\in\left(0,\frac{c}{3}\right).
\]

%%%%%%%%%%%%%%%%%%%
%
%Thus, by comparison of $\left(u,v\right)$ with the spatially homogeneous
%sub-solution $\left(\underline{U},\overline{V}\right)$ satisfying
%the system 
%\[
%\left\{ \begin{matrix}\underline{U}'\left(t\right)=\underline{U}\left(t\right)\left(1-\underline{U}\left(t\right)-a\overline{V}\left(t\right)\right) & \text{for all }t\in\left(T_{\varepsilon},+\infty\right)\\
%\overline{V}'\left(t\right)=r\overline{V}\left(t\right)\left(1-\overline{V}\left(t\right)-b\underline{U}\left(t\right)\right) & \text{for all }t\in\left(T_{\varepsilon},+\infty\right)\\
%\underline{U}\left(T_{\varepsilon}\right)=\frac{1-a}{4}\\
%\overline{V}\left(T_{\varepsilon}\right)=1
%\end{matrix}\right.,
%\]
%whose convergence to $\left(1,0\right)$ is well-known, we find 
%\[
%\lim\limits _{t\to+\infty}\sup\limits _{0\leq x<\left(c-\varepsilon\right)t}\left(\left|u\left(t,x\right)-1\right|+\left|v\left(t,x\right)\right|\right)=0\text{ for all }\varepsilon\in\left(0,c\right).
%\]
%
%%%%%%%%%%%%%%%%%%%

This means $c_2 \geq c$, and directly contradicts the choice of $c>c_{2}$ made at the beginning of the proof.

Next, we consider the case where 
\[
c_1\geq\inf\left\{ c>0\ |\ \lim\limits _{t\to+\infty}\sup\limits _{x\geq ct}\left|v\left(t,x\right)\right|=0\right\}.
\]
Since the proof is mostly the same, we only sketch it. Again, we argue by contradiction and use \propref{Sub-solution_nonexistence}.
Using the assumption on the spreading of $v$, we can establish the following estimate:
\[
v\left(t,x-\hat{c}t\right)\leq \mathbf{1}_{y\leq y_0}\left(x-\hat{c}t\right) + \frac{\delta}{2a}\mathbf{1}_{y\geq y_0}\left(x-\hat{c}t\right),
\]
for some $y_0\in\mathbb{R}$ and with $\hat{c}=\frac{\tilde{c}+c_1}{2}$. Thanks to this, we can directly use $\eta\underline{u_{\delta,\zeta,\kappa}}$, for some small $\eta>0$, as sub-solution
for $u$ and deduce a contradiction. We point out that in this case, we do not use the competitive comparison principle but 
instead use the scalar one. 
\end{proof}

\subsection{Proof of \thmref{Compactly_supported}}

\subsubsection{Hair-trigger effect and extinction}
\begin{prop}
\label{prop:Hair-trigger_effect_vs_extinction}Let $u_{0}\in\mathscr{C}\left(\mathbb{R},\left[0,1\right]\right)\backslash\left\{ 0\right\} $
with support included in a left half-line and $v_{0}\in\mathscr{C}\left(\mathbb{R},\left[0,1\right]\right)\backslash\left\{ 0\right\} $
with compact support. Let $\left(u,v\right)$ be the solution of \ref{eq:competition_diffusion_system}. 
\begin{enumerate}
\item If $2\sqrt{rd}>2$, then
\[
\lim\limits _{t\to+\infty}\sup\limits _{0\leq x<\left(c_{\textup{LLW}}-\varepsilon\right)t}\left(\left|u\left(t,x\right)-1\right|+\left|v\left(t,x\right)\right|\right)=0\text{ for all }\varepsilon\in\left(0,c_{\textup{LLW}}\right),
\]
\[
\lim\limits _{t\to+\infty}\sup\limits _{\left(2\sqrt{rd}+\varepsilon\right)t<x}\left(\left|u\left(t,x\right)\right|+\left|v\left(t,x\right)\right|\right)=0\text{ for all }\varepsilon>0,
\]
\[
\lim\limits _{t\to+\infty}\sup\limits _{\left(2+\varepsilon\right)t<x<\left(2\sqrt{rd}-\varepsilon\right)t}\left(\left|u\left(t,x\right)\right|+\left|v\left(t,x\right)-1\right|\right)=0\text{ for all }\varepsilon\in\left(0,\frac{2\sqrt{rd}-c_{\textup{LLW}}}{2}\right).
\]
\item If $2\sqrt{rd}<2$, then
\[
\lim_{t\to+\infty}\sup_{x \in \mathbb{R}}\left|v\left(t,x\right)\right|=0,
\]
\[
\lim\limits _{t\to+\infty}\sup\limits _{0\leq x<\left(2-\varepsilon\right)t}\left|u\left(t,x\right)-1\right|=0\text{ for all }\varepsilon\in\left(0,2\right),
\]
\[
\lim\limits _{t\to+\infty}\sup\limits _{\left(2+\varepsilon\right)t<x}\left|u\left(t,x\right)\right|=0\text{ for all }\varepsilon>0.
\]
\end{enumerate}
\end{prop}

\begin{rem*} 
The inequality regarding $\left(2+\varepsilon\right)t<x<\left(2\sqrt{rd}-\varepsilon t\right)$
is by far the more interesting and the less trivial. It basically
means that $u$ does not exert any competition far ahead of its own
territory. It was first proved by Ducrot, Giletti and Matano \cite{Ducrot_Giletti_Matano_2}
in the case of predator\textendash prey interactions (the conclusion
being then that no predation occurs far ahead of the territory of
the predator), and by Lin and Li \cite{Lin_Li_2012} in case of two-species competition (the conclusion being that the region of coexistence falls behind the territory where the faster diffuser dominates). 
The proof of Ducrot \textit{et al.} was sufficiently robust and generic to
be reused by Carr\`{e}re \cite{Carrere_2017} in the bistable competitive
case and to be reused again here, in the monostable case. Although
it would certainly be interesting to write the result of Ducrot \textit{et al.} 
in the most general form possible (with more than two species and
minimal assumptions on the interactions), this is far beyond the scope
of this paper. Therefore we simply adapt the main idea of their proof. 

\end{rem*}
\begin{proof}
First, applying the comparison principle with the solution of 
\[
\left\{ \begin{matrix}\partial_{t}u_{\textup{KPP}}-\partial_{xx}u_{\textup{KPP}}=u_{\textup{KPP}}\left(1-u_{\textup{KPP}}\right) & \text{in }\left(0,+\infty\right)\times\mathbb{R}\\
u_{\textup{KPP}}\left(0,x\right)=u_{0}\left(x\right) & \text{for all }x\in\mathbb{R}
\end{matrix}\right.,
\]
 we find directly $u\leq u_{\textup{KPP}}$, whence
\[
\lim\limits _{t\to+\infty}\sup\limits _{x>\left(2+\varepsilon\right)t}u\left(t,x\right)=0\text{ for all }\varepsilon>0.
\]
Similarly,
\[
\lim\limits _{t\to+\infty}\sup\limits _{x>\left(2\sqrt{rd}+\varepsilon\right)t}v\left(t,x\right)=0\text{ for all }\varepsilon>0.
\]
 Furthermore, $\left(u,v\right)$ satisfies also $\left(u,v\right)\succeq \left(u_{\textup{LLW}},v_{\textup{LLW}}\right)$, 
where $\left(u_{\textup{LLW}},v_{\textup{LLW}}\right)$
is the solution of (\ref{eq:competition_diffusion_system}) with initial
data $\left(u_{0},1\right)$, and by Lewis\textendash Li\textendash Weinberger
\cite{Lewis_Weinberg}, this yields 
\[
\lim\limits _{t\to+\infty}\sup\limits _{0\leq x<\left(c_{\textup{LLW}}-\varepsilon\right)t}\left|u\left(t,x\right)-1\right|+\left|v\left(t,x\right)\right|=0\text{ for all }\varepsilon\in\left(0,c_{\textup{LLW}}\right).
\]

Next, let us prove that if $2\sqrt{rd}<2$ and provided
\[
\lim\limits _{t\to+\infty}\sup\limits _{\left(2\sqrt{rd}+\varepsilon\right)t<x<\left(2-\varepsilon\right)t}\left(\left|u\left(t,x\right)-1\right|+\left|v\left(t,x\right)\right|\right)=0\text{ for all }\varepsilon\in\left(0,\frac{2-2\sqrt{rd}}{2}\right),
\]
then in fact the above limit can be reinforced as 
\[
\lim\limits _{t\to+\infty}\sup\limits _{0\leq x<\left(2-\varepsilon\right)t}\left(\left|u\left(t,x\right)-1\right|+\left|v\left(t,x\right)\right|\right)=0\text{ for all }\varepsilon\in\left(0,2\right).
\]
 Let $\varepsilon\in\left(0,\frac{2-2\sqrt{rd}}{3}\right)$. It is
well-known (see Du\textendash Lin \cite{Du_Lin_2010,Du_Lin_2010_er})
that there exists a unique solution of 
\[
\left\{ \begin{matrix}-\underline{\varphi}''=\underline{\varphi}\left(1-a-\underline{\varphi}\right) & \text{in }\left(0,+\infty\right)\\
\underline{\varphi}\left(0\right)=0\\
\underline{\varphi}\left(x\right)>0 & \text{for all }x>0
\end{matrix}\right..
\]
Furthermore, this solution is increasing in $\left(0,+\infty\right)$
and converges to $1-a$ at $+\infty$. In view of the assumption on
the limit of $\left(u,v\right)$ in $\left(2\sqrt{rd}+\varepsilon\right)t<x<\left(2-\varepsilon\right)t$,
there exists $T\geq0$ and $x_{0}>0$ such that 
\[
u\left(t, y +\left(2-2\varepsilon\right)t\right)>1-a > \underline\varphi(y + (2-2\varepsilon)t - x_0)\text{ for }(y,t) \in \{0\} \times [T,+\infty),
\]
\[
u(t,y + (2-2\varepsilon)t) \geq \underline{\varphi}(y + (2-2\varepsilon)t - x_0) \text{ for } (y,t) \in (-\infty,0]\times \{T\},
\]
\[
\underline\varphi(y + (2-2\varepsilon)t - x_0) >0 \text{ when }(y,t)=(0,T).
\]
Let $\tilde{u}(t,y)=u(t,y + (2-2\varepsilon)t)$ and $\underline{\tilde u}(t,y) = \underline\varphi(y + (2-2\varepsilon)t - x_0)$.  Then they satisfy for all $(t,y) \in (0,+\infty) \times \mathbb{R}$,
\[
\partial_t \underline{\tilde u} - \partial_{xx}\underline{\tilde u} -(2-2\varepsilon)\partial_x\underline{\tilde u} - \underline{\tilde u} (1- \underline{\tilde u}  - av) \leq 0 = \partial_{xx}{\tilde u} - (2-2\varepsilon)\partial_x \tilde{u} - {\tilde u} (1- {\tilde u}  - av), 
\]
it follows by virtue of the scalar comparison principle 
and of a change of variable $y=x-\left(2-2\varepsilon\right)t$ 
that 
\[
\underline{u}\left(t,x\right)\leq u\left(t,x\right)\text{ for all }t\geq T\text{ and }x\leq\left(2-2\varepsilon\right)t.
\]
Consequently, 
\[
\liminf_{t\to+\infty}\inf_{\varepsilon t<x<\left(2-2\varepsilon\right)t}u\left(t,x\right)\geq1-a,
\]
whence there exists $T'\geq0$ such that 
\[
\inf_{\varepsilon t<x<\left(2-2\varepsilon\right)t}u\left(t,x\right)\geq\frac{1-a}{2}>0\text{ for all }t\geq T'.
\]

Repeating a limiting argument developed earlier in the course of the proof of 
\thmref{Nonexistence} showing the locally uniform convergence of $(u,v)$ to $(1,0)$ 
in the cone defined by $2\varepsilon t<x<\left( 2-3\varepsilon \right)t$ (with suitable
$\varepsilon$) and recalling that we have the estimate
\[
\lim\limits _{t\to+\infty}\sup\limits _{0\leq x<\left(c_{\textup{LLW}}-\varepsilon\right)t}\left|u\left(t,x\right)-1\right|+\left|v\left(t,x\right)\right|=0\text{ for all }\varepsilon\in\left(0,c_{\textup{LLW}}\right),
\]
the claim is now proved.

%%%%%%%%%%%%%%%%%%
%
%By virtue of \thmref{Competitive_comparison_principle}, the solution
%$\left(U,V\right)$ of (\ref{eq:competition_diffusion_system}) with
%constant initial values $\left(\frac{1-a}{2},1\right)$ satisfies
%\[
%\left(U,V\right)\left(t-T',x\right)\preceq\left(u,v\right)\left(t,x\right)\text{ for all }t\geq T'\text{ and }\varepsilon t<x<\left(2-2\varepsilon\right)t.
%\]
%Since $\left(U,V\right)$ coincides with the solution of the ODE system
%\[
%\left\{ \begin{matrix}U'=U\left(1-U-aV\right)\\
%V'=rV\left(1-V-bU\right)\\
%\left(U,V\right)\left(0\right)=\left(\frac{1-a}{2},1\right)
%\end{matrix}\right.,
%\]
% standard theory on such systems shows that $\left(U,V\right)$ converges
%to $\left(1,0\right)$, whence $\left(u,v\right)$ itself converges
%to $\left(1,0\right)$ uniformly in $\varepsilon t<x<\left(2-2\varepsilon\right)t$.
%Recalling that we also have the estimate
%\[
%\lim\limits _{t\to+\infty}\sup\limits _{0\leq x<\left(c_{\textup{LLW}}-\varepsilon\right)t}\left|u\left(t,x\right)-1\right|+\left|v\left(t,x\right)\right|=0\text{ for all }\varepsilon\in\left(0,c_{\textup{LLW}}\right),
%\]
%the claim is proved.
%
%%%%%%%%%%%%%%%%%%%

It now remains to prove the most difficult part, namely
\[
 \lim\limits _{t\to+\infty}\sup\limits _{\left(2+\varepsilon\right)t<x<\left(2\sqrt{rd}-\varepsilon\right)t}\left(\left|u\left(t,x\right)\right|+\left|v\left(t,x\right)-1\right|\right)=0, \quad \text{ if }2\sqrt{rd}>2,
\]
and 
\[
\lim\limits _{t\to+\infty}\sup\limits _{\left(2\sqrt{rd}+\varepsilon\right)t<x<\left(2-\varepsilon\right)t}\left(\left|u\left(t,x\right)-1\right|+\left|v\left(t,x\right)\right|\right)=0, \quad \text{ if }2\sqrt{rd} < 2.
\]

Since this is a symmetric statement and since the forthcoming proof
does not rely upon the assumptions $a<1$ and $b>1$, we only do the
case $2\sqrt{rd}>2$ (when $v$ spreads faster than $u$) and the proof will be valid for the other case (when $u$ spreads faster than $v$) 
as well.

\textbf{Step 1:} Let $\bar{u}:\left(t,x\right)\mapsto \min\left(1,e^{-\left(x-2t-x_1\right)}\right)$, where $x_1$ is chosen such that $\bar{u}\left(0,x\right) \geq u\left(0,x\right)$ for all $x\in \mathbb{R}$. Then, by standard scalar comparison, 
\[
u\left(t,x\right) \leq \bar{u}\left(t,x\right) \quad \text{ for all }t\geq 0\text{ and }x\in \mathbb{R}.
\]

\textbf{Step 2:} We show that for each $c \in (2,2\sqrt{rd})$, there exist positive constants $\delta,x_2,\eta_1,R$ such that 
\begin{equation}\label{eq:hair-trigger_prep}
v\left(t',x + x_2+ct\right) \geq \eta_1 \quad \text{ for all }t\geq 1, x\in\left(-2R,2R\right) \text{ and }t'\in \left[\left(1-\delta\right)t,\left(1+\delta\right)t\right].  
\end{equation}

To show (\ref{eq:hair-trigger_prep}), 
fix $c \in \left(2,2\sqrt{dr}\right)$ and fix $\delta$ so small that 
\[
2 < \frac{c}{1+\delta} <\frac{c}{1-\delta}  < 2\sqrt{dr}.
\]

Let $\eta>0$, $R>0$, $x_{2}\in\mathbb{R}$, $\tilde{c} \in \left[\frac{c}{1+\delta},\frac{c}{1-\delta}\right]$ and define 
\[
\underline{v}^{\tilde{c}}:\left(t,x\right)\mapsto\eta\textup{e}^{-\frac{\tilde{c}}{2d}\left(x-\tilde{c}t\right)}\psi_{4R}\left(x-\tilde{c}t-x_{2}\right),
\]
where $\left(\lambda_{4R},\psi_{4R}\right)$ is the Dirichlet principal
eigenpair defined by
\[
\left\{ \begin{matrix}-d\psi_{4R}''=\lambda_{4R}\psi_{4R} & \text{in }\left(-4R,4R\right)\\
\psi_{4R}\left(\pm4R\right)=0\\
\psi_{4R}\left(x\right)>0 & \text{for all }x\in\left(-4R,4R\right)\\
\max\psi_{4R}=1
\end{matrix}\right..
\]
 The principal eigenvalue $\lambda_{4R}$ is positive, vanishes as
$R\to+\infty$ and $\psi_{4R}$ is extended into $\mathbb{R}$ by
setting $\psi_{4R}\left(x\right)=0$ if $\left|x\right|>4R$. 

Obviously, 
\[
\partial_{t}\underline{v}^{\tilde{c}} - d\partial_{xx}\underline{v}^{\tilde{c}} - r\underline{v}^{\tilde{c}}\left(1-\underline{v}^{\tilde{c}}-bu\right)\leq\partial_{t}\underline{v}^{\tilde{c}} - d\partial_{xx}\underline{v}^{\tilde{c}} - r\underline{v}^{\tilde{c}}\left(1-\underline{v}^{\tilde{c}}-b\overline{u}\right),
\]
whence the left-hand side above divided by $\eta\textup{e}^{-\frac{\tilde{c}}{2d}\left(x-\tilde{c}t\right)}$ is 
\textit{a fortiori} smaller than or equal to
\begin{align*}
&\quad \frac{\tilde{c}^{2}}{2d}\psi_{4R}-\tilde{c}\psi_{4R}'-d\left(\psi_{4R}''+\frac{\tilde{c}^{2}}{4d^{2}}\psi_{4R}-\frac{\tilde{c}}{d}\psi_{4R}'\right)-r\psi_{4R}\left(1-\underline{v}^{\tilde{c}}-b\overline{u}\right) \\
& \leq\left(\frac{\tilde{c}^{2}}{4d}+\lambda_{4R}-r+r\left(\underline{v}^{\tilde{c}}+b\overline{u}\right)\right)\psi_{4R} \\
 & \leq\left(\lambda_{4R}+r\left(\underline{v}^{\tilde{c}}+b\overline{u} - \gamma\right)\right)\psi_{4R},
\end{align*}
where the last inequality holds provided we choose the  constant $\gamma>0$ so small that 
\[
2\sqrt{r(1-\gamma)d} > \frac{c}{1-\delta}  \geq \tilde{c}.
\]

Therefore, by choosing $R$ so large that $\lambda_{4R} < r\frac{\gamma}{4}$, $x_2$ so large that 
\[
\bar{u}\left(t,x\right) \leq \frac{\gamma}{4b}\text{ for all }t\geq 0\text{ and }x\geq 2t+x_2 - 4R
\]
(which is possible by Step 1), and $\eta$ so small that 
\[
\eta\sup_{\hat{c}\in\left[\frac{c}{1+\delta}, \frac{c}{1-\delta} \right]}\sup_{\xi\in\left(-4R+x_2,4R+x_2\right)}\left(\textup{e}^{-\frac{\tilde{c}}{2d}\xi}\psi_{4R}\left(\xi-x_{2}\right)\right)\leq\frac{\gamma}{4},
\]
\[
\eta\sup_{\hat{c}\in\left[\frac{c}{1+\delta}, \frac{c}{1-\delta} \right]}\left(\textup{e}^{-\frac{\hat{c}}{2d}\left(x-\hat{c}\right)}\psi_{4R}\left(x-\hat{c}-x_{2}\right)\right)\leq v\left(1,x\right)\text{ for all }x\in\mathbb{R},
\]
we deduce that $\underline{v}^{\tilde{c}}$ is a sub-solution for the single parabolic equation satisfied by $v$. 

By scalar comparison, $v(t,x) \geq \underline{v}^{\tilde{c}}(t,x)$ for all $t \geq 1$ and $x \in \mathbb{R}$. It follows 
then that 
\begin{align*}
v\left(\frac{c}{\tilde{c}}t, x + x_2 + ct\right) &\geq \underline{v}^{\tilde{c}}\left(\frac{c}{\tilde{c}}t, x + x_2 + ct\right)\\
&= \eta e^{- \frac{\tilde c}{2d}(x + x_2)} \psi_{4R}(x)\\
&\geq \eta e^{-\frac{\tilde{c}}{2d}x_2} e^{-\frac{\tilde c}{d}R} \min_{[-2R, 2R]} \psi_{4R}
\end{align*}
for all $t \geq 1$ and $x \in \left[ -2R, 2R\right]$.
Noticing that the last expression on the right-hand side above is constant and denoting
\[
\eta_1=\eta \min_{[-2R, 2R]} \psi_{4R}\inf_{\hat{c}\in\left[\frac{c}{1+\delta}, \frac{c}{1-\delta} \right]}\left(e^{-\frac{\hat{c}}{2d}x_2} e^{-\frac{\hat c}{d}R}\right),
\]
we may take the infimum over all $\tilde{c} \in \left[ \frac{c}{1+\delta}, \frac{c}{1-\delta}\right]$ and obtain indeed (\ref{eq:hair-trigger_prep}).

\textbf{Step 3:} We are now in position to show that, for any small $\varepsilon>0$, 
\[
\lim_{t \to +\infty} \sup_{(2+\varepsilon)t < x < (2\sqrt{rd} - \varepsilon)t} |v\left(t,x\right)-1| = 0.
\]
%%%%%%%%%%%%%%%%%%%%%%%
%We will proceed with a limiting argument similar to the one used twice already and 
%detailed in the proof of \thmref{Nonexistence}.
%%%%%%%%%%%%%%%%%%%%%%%%

Assume by contradiction the existence of sequences $\left(t_n\right)_{n\in\mathbb{N}}$, 
$\left(x_n\right)_{n\in\mathbb{N}}$ and of $c \in (2,2\sqrt{rd})$ such that, as $n\to+\infty$,
$t_n\to+\infty$, $\frac{x_n}{t_n}\to c$ and $\limsup v\left(t_n, x_n\right) <1$.

Denote $\left(c_n\right)_{n\in\mathbb{N}}=\left(\frac{x_n}{t_n}\right)_{n\in\mathbb{N}}$,  assume without loss of generality that $\left| \frac{c}{c_n} -1\right| < \delta/2$, where $\delta = \delta(c)$ is specified in Step 2.

For all $n\in\mathbb{N}$, define $\tau_n=\frac{c_n}{c}t_n=\frac{x_n}{c}$ and
\[
v_n:\left(t,x\right)\mapsto v\left(t + \frac{c}{c_n}\tau_n, x + x_2+ c \tau_n\right).
\]
By Step 2 (with $t'=t+\frac{c}{c_n}\tau_n$ and $t = \tau_n$), we deduce that
\[
v_n(t,x) \geq \eta_1 \text{ if }  |x| \leq 2R\text{ and }\left| t + \left(\frac{c}{c_n}-1\right)\tau_n\right| < \delta \tau_n, 
\]
and hence (using $\left| \frac{c}{c_n} -1\right| < \frac{\delta}{2}$) if $|x - x_2| \leq 2R$ 
and $\left| t\right| < \frac{\delta}{2} \tau_n$.

By classical parabolic estimates (see Lieberman \cite{Lieberman_2005}),
$\left(v_{n}\right)_{n\in\mathbb{N}}$ converges up to
a diagonal extraction in $\mathscr{C}_{loc}\left(\mathbb{R}^2,\left[0,1\right]\right)$
to a limit $v_{\infty}$ which satisfies (using the fact that
$u(t+\frac{c}{c_n}\tau_n, x+ c\tau_n) \to 0$ in $\mathscr{C}_{loc}\left(\mathbb{R}^2,\left[0,1\right]\right)$ by Step 1, since $(c\tau_n)/(\frac{c}{c_n}\tau_n) = c_n \geq \frac{c}{1+\delta} > 2$ for all $n$) 
\[
\partial_{t}v_{\infty}-d\partial_{xx}v_{\infty}-rv_{\infty}\left(1-v_{\infty}\right)=0\text{ in }\mathbb{R}^{2}
\]
and, in view of the above estimates, 
\[
v_{\infty}\left(t,x+x_{2}\right)\geq\eta_{1}\text{ for all }t \in \mathbb{R} \text{ and }x\in\left[-2R,2R\right].
\]
By standard classification of the entire solutions of the KPP equation, this implies $v_\infty\equiv 1$. In particular,
$$
v(t_n,x_n) = v\left( \frac{c}{c_n}\tau_n,c\tau_n\right) = v_n(0,-x_2) \to 1.
$$
This directly contradicts $\limsup v(t_n,x_n) <1$.
\end{proof}

In view of \propref{Hair-trigger_effect_vs_extinction}, in order to prove \thmref{Compactly_supported},
we only have to prove that for each sufficiently small $\varepsilon>0$,
\[
\lim\limits _{t\to+\infty}\sup\limits _{x<\left(c^{\star}-\varepsilon\right)t}\left(\left|u\left(t,x\right)-1\right|+\left|v\left(t,x\right)\right|\right)=0,
\]
\[
\lim\limits _{t\to+\infty}\sup\limits _{\left(c^{\star}+\varepsilon\right)t<x<\left(2\sqrt{rd}-\varepsilon\right)t}\left(\left|u\left(t,x\right)\right|+\left|v\left(t,x\right)-1\right|\right)=0,
\]
where
\[
c^{\star}=\max\left(c_{\textup{LLW}},f^{-1}\left(2\sqrt{rd}\right)\right).
\]

\subsubsection{Proof of \thmref{Compactly_supported}}\label{subsec:3.2.2}

We begin with an algebraic lemma.
\begin{lem}
\label{lem:Estimate_speed_w} Let $c_{2}\geq2\sqrt{1-a}$ and $c_{1}>c_{2}$
such that $c_{1}\geq f\left(c_{2}\right)$. Then 
\[
\frac{\left(\Lambda\left(c_{2},c_1\right)\right)^{2}+1}{\Lambda\left(c_{2},c_1\right)}<c_{1}.
\]
\end{lem}

\begin{proof}
First, $\Lambda(c_2,c_1)$ is well-defined as $c_1 \geq \max\{c_2, f(c_2)\}$. Noticing that 
\[
\left(\Lambda\left(c_{2},c_{1}\right)\right)^{2}-c_{1}\Lambda\left(c_{2},c_{1}\right)+\lambda\left(c_{2}\right)\left(c_{1}-c_{2}\right)+1=0,
\]
we find that the claimed inequality is equivalent to 
$-\lambda\left(c_{2}\right)\left(c_{1}-c_{2}\right)<0.$
\end{proof}

Now we prove the remaining part of \thmref{Compactly_supported}.
Assume $2\sqrt{rd}>2$, define $c^{\star}$ as above, fix $u_{0}\in\mathscr{C}\left(\mathbb{R},\left[0,1\right]\right)\backslash\left\{ 0\right\} $
with support included in a left half-line and $v_{0}\in\mathscr{C}\left(\mathbb{R},\left[0,1\right]\right)\backslash\left\{ 0\right\} $
with compact support and let $\left(u,v\right)$ be the solution of
\eqref{competition_diffusion_system}. 
\begin{proof}
By virtue of \thmref{Nonexistence} and \propref{Hair-trigger_effect_vs_extinction}, 
\[
\sup\left\{ c>0\ |\ \lim\limits _{t\to+\infty}\sup\limits _{x\leq ct}\left(\left|u\left(t,x\right)-1\right|+\left|v\left(t,x\right)\right|\right)=0\right\} \geq c^{\star}.
\]
It remains to verify that the quantity 
\[
\overline{c}=\inf\left\{ 0 < c <2\sqrt{rd}\ |\ \lim\limits _{t\to+\infty}\sup\limits _{ct\leq x\leq\frac{2\sqrt{rd}+c}{2}t}\left(\left|u\left(t,x\right)\right|+\left|v\left(t,x\right)-1\right|\right)=0\right\}
\]
satisfies $\overline{c} \leq c^{\star}$. Notice that by \propref{Hair-trigger_effect_vs_extinction}, $\overline{c}\leq 2$.

Assume by contradiction $\overline{c} \in \left(c^{\star},2\right]$ and let $c_{2}\in\left(c^{\star},\overline{c}\right)$.
Define $\delta^{\star}$ as in \propref{Super-solution_compactly_supported}, let 
$\delta\in\left(0,\delta^{\star}\right)$ so small that $c_{2}^{\delta}<\overline{c}$ and define subsequently 
$\left(\overline{u}_{\delta},\underline{v}_{\delta}\right)$.

By standard comparison, 
\[
u\left(t,x\right)\leq\min\left(1,\overline{w_{c_{2},2\sqrt{rd}}}\left(t,x\right)\right)\text{ for all }\left(t,x\right)\in\left(0,+\infty\right)\times\mathbb{R}.
\]

By virtue of \propref{Super-solution_compactly_supported},
there exists $y_0\in\mathbb{R}$ such that, for all $y\geq y_0$ and $t\geq0$, 
\[
\overline{u_{\delta}}\left(0,x-y-\frac{\left(\Lambda\left(c_{2},2\sqrt{rd}\right)\right)^{2}+1}{\Lambda\left(c_{2},2\sqrt{rd}\right)}t\right)\geq\min\left(1,\overline{w_{c_{2},2\sqrt{rd}}}\left(t,x\right)\right)\text{ for all }x\in\mathbb{R},
\]

Since 
\[
c_{2}<\overline{c}\leq2<2\sqrt{rd},
\]
\[
c_{2}>c^{\star}=\max\left(c_{\textup{LLW}},f^{-1}\left(2\sqrt{rd}\right)\right)\geq \max\left(2\sqrt{1-a},f^{-1}\left(2\sqrt{rd}\right)\right),
\]
\lemref{Estimate_speed_w} yields
\[
\frac{\left(\Lambda\left(c_{2},2\sqrt{rd}\right)\right)^{2}+1}{\Lambda\left(c_{2},2\sqrt{rd}\right)}<2\sqrt{rd}.
\]

Choose $c>0$ such that
\[
\max\left(\frac{1}{2}\left(2\sqrt{rd}+\frac{\left(\Lambda\left(c_{2},2\sqrt{rd}\right)\right)^{2}+1}{\Lambda\left(c_{2},2\sqrt{rd}\right)}\right),2\right) < c < 2 \sqrt{rd}.
\]

By virtue of \propref{Super-solution_compactly_supported}, 
$x\mapsto \underline{v_{\delta}}\left(0,x-y-\frac{\left(\Lambda\left(c_{2},2\sqrt{rd}\right)\right)^{2}+1}{\Lambda\left(c_{2},2\sqrt{rd}\right)}t\right)$ 
is compactly supported for all $y\geq y_0$ and $t\geq 0$. 
Since also $2 < c < 2\sqrt{rd}$,  
by virtue of \propref{Hair-trigger_effect_vs_extinction}, there exists $T_0\geq 0$ such that, for all $T\geq T_0$,

\[
\underline{v_{\delta}}\left(0,x-cT\right)\leq v\left(T,x\right)\text{ for all }x\in\mathbb{R}.
\]

Now, relating the parameters $y$ and $T$ as follows, 
\[
cT=y+\frac{\left(\Lambda\left(c_{2},2\sqrt{rd}\right)\right)^{2}+1}{\Lambda\left(c_{2},2\sqrt{rd}\right)}T, \quad \text{ where we have }c > \frac{\left(\Lambda\left(c_{2},2\sqrt{rd}\right)\right)^{2}+1}{\Lambda\left(c_{2},2\sqrt{rd}\right)},
\]
we find the existence of $y\geq y_0$ and $T\geq T_0$ such that
\[
\underline{v_{\delta}}\left(0,x-y-\frac{\left(\Lambda\left(c_{2},2\sqrt{rd}\right)\right)^{2}+1}{\Lambda\left(c_{2},2\sqrt{rd}\right)}T\right)\leq v\left(T,x\right)\text{ for all }x\in\mathbb{R}.
\]
 Then 
\[
\left(\overline{u},\underline{v}\right):\left(t,x\right)\mapsto\left(\overline{u_{\delta}},\underline{v_{\delta}}\right)\left(t-T,x-y-\frac{\left(\Lambda\left(c_{2},2\sqrt{rd}\right)\right)^{2}+1}{\Lambda\left(c_{2},2\sqrt{rd}\right)}T\right)
\]
 is a super-solution of (\ref{eq:competition_diffusion_system}) which
satisfies $\left(u,v\right)\preceq\left(\overline{u},\underline{v}\right)$
at $t=T$, whence by the comparison principle of \thmref{Competitive_comparison_principle}
it satisfies the same inequality at any time $t\geq T$. 

A contradiction follows from \propref{Super-solution_compactly_supported} and $c_{2}^{\delta}<\overline{c}$, as in
the proof of \thmref{Nonexistence}. 
\end{proof}

\subsection{Proof of \thmref{Generalized_terraces}}\label{subsec:3.3}

Let $c_{1}>2\sqrt{rd}$, $c_{2}>c_{\textup{LLW}}$ and assume $c_{1}>\max\left(c_{2},f\left(c_{2}\right)\right)$. 
\begin{proof}
Fix $\delta^{\star}$, $\delta=\frac{\delta^{\star}}{2}$ and $c_{2}^{\delta}$, 
and define the super- and sub-solutions $(\overline{u_\delta},\underline{v_\delta})$ and $(\underline{u_\delta},\overline{v_\delta})$ as in \propref{Super-sub-solution_terraces}. 

First, let $\left(u_{0},v_{0}\right)\in\mathscr{C}\left(\mathbb{R},\left[0,1\right]^{2}\right)$
be a pair such that 
\[
\left(\underline{u_{\delta}},\overline{v_{\delta}}\right)\left(0,x\right)\preceq\left(u_{0},v_{0}\right)\left(x\right)\preceq\left(\overline{u_{\delta}},\underline{v_{\delta}}\right)\left(0,x\right)\text{ for all }x\in\mathbb{R}
\]
and satisfying also 
\[
u_{0}\left(x\right)\le\overline{w_{c_{2},c_{1}}}\left(0,x\right)\text{ for all }x\in\mathbb{R}.
\]
 By virtue of \thmref{Existence_by_comparison}, there exists a (unique)
solution $\left(u,v\right)$ of (\ref{eq:competition_diffusion_system})
such that 
\[
\left(u,v\right)\left(0,x\right)=\left(u_{0},v_{0}\right)\left(x\right)\text{ for all }x\in\mathbb{R},
\]
\[
\left(\underline{u_{\delta}},\overline{v_{\delta}}\right)\left(t,x\right)\preceq\left(u,v\right)\left(t,x\right)\preceq\left(\overline{u_{\delta}},\underline{v_{\delta}}\right)\left(t,x\right)\text{ for all }t\in\left(0,+\infty\right)\text{ and }x\in\mathbb{R},
\]
\[
u\left(t,x\right)\le\min\left(1,\overline{w_{c_{2},c_{1}}}\left(t,x\right)\right)\text{ for all }t\in\left(0,+\infty\right)\text{ and }x\in\mathbb{R}.
\]

Next, in view of the spreading properties of the super-solution and
the sub-solution and thanks to the comparison argument with the ODE
system detailed in the proof of \thmref{Nonexistence},
it only remains to show that the quantity
\[
\bar{c}=\inf\left\{ c>0\ |\ \lim\limits _{t\to+\infty}\sup\limits _{ct\leq x\leq\frac{c_{1}+c}{2}t}\left(\left|u\left(t,x\right)\right|+\left|v\left(t,x\right)-1\right|\right)=0\right\}
\]
satisfies $\bar{c}\leq c_2$.

Now, the choice of super- and sub-solutions above proves that $\bar{c} \in [c_2, c^\delta_2]$. Suppose to the contrary that $\bar{c} > c_2$. Then we can fix a sufficiently small $\delta' \in (0,\delta)$ such that $c^{\delta'}_2 \in (c_2,\bar{c})$. (This is possible since $c^{\delta'}_2 \searrow c_2$ as $\delta' \searrow 0$, by \propref{Super-sub-solution_terraces}.) 
Then, thanks to
\begin{itemize}
\item the estimate $u\le\min\left(1,\overline{w_{c_{2},c_{1}}}\right)$,
\item \lemref{Estimate_speed_w} which controls from above the speed of
$\overline{w_{c_{2},c_{1}}}$,
\item the control from below of the exponential tail of $v$, 
\end{itemize}
we can use the super-solution $\left(\overline{u_{\delta'}},\underline{v_{\delta'}}\right)$
associated with a sufficiently small $\delta'\in\left(0,\delta\right)$
as barrier after some large time $T_{\delta'}$ to slow down the invasion
of $u$ in an impossible way. More precisely, just as in the proof of \thmref{Compactly_supported},
there exist large $T'$ and $y_0$ such that, for all $x\in\mathbb{R}$,
\[
u(T',x) \leq \min\left(1,\overline{w_{c_{2},c_{1}}}(T',x)\right) \leq \overline{u_{\delta'}}(0, x - y_0)
\]
and
\[
v(T',x) \geq \underline{v_{\delta'}}(0,x-y_0).
\]
This implies that 
\[
\left(u\left(t,x\right), v\left(t,x\right)\right) \preceq \left(\overline{u_{\delta'}}\left(t-T', x - y_0\right),\underline{v_{\delta'}}\left(t-T',x-y_0\right)\right)\text{ for all }t \geq T' \text{ and }x \in \mathbb{R} 
\]
and  
$\bar{c} \leq c^{\delta'}_2$, which is a contradiction. This ends the proof.
\end{proof}

\section{Proofs of \propref{Sub-solution_nonexistence}, \propref{Super-solution_compactly_supported}
and \propref{Super-sub-solution_terraces}}

%In Subsection \ref{subsec:4.1} We define various objects which will be used in the construction of super- and sub-solutions. The lemmas concerning the properties of these objects are proved in Subsection \ref{subsec:4.2}. 

%For ease of reading, we suggest the readers to skip Subsection \ref{subsec:4.1} and only refer to it when a specific object is being used.

\subsection{Several useful objects}\label{subsec:4.1}

In this subjection, we will define components which will be used for our later constructions. For ease of reading, we suggest the readers to skip Subsection \ref{subsec:4.1} and only refer to it when a specific object is being used.

{\small

\noindent \begin{tabular}{ |p{2cm}||p{3.4cm}|p{2.9cm}|p{3.2cm}|  }
 \hline
 \multicolumn{4}{|c|}{\textbf{List of Objects}} \\
 \hline
\textit{Object(s)} & \textit{Defined in} & \textit{Used in} & \textit{Property}\\
 \hline
 
$f(c)$   & Sect. \ref{subsec:1.1}, (\ref{eq:f})    &  &   $f(c) = c  + 2\sqrt{a}\quad\quad$   $\qquad$ $ \,\qquad -\sqrt{c^2 - 4(1-a)} $\\

\hline

$c_{\textup{nlp}}$   & \thmref{Compactly_supported}(2)    &  &  $c_{\textup{nlp}} > c_{\textup{LLW}}$\\

\hline
 
 $ \lambda(c)$ & Sect. \ref{subsec:1.1.3} &  & $\lambda(c)= \lambda_\delta(c)\big|_{\delta=0}$\\
 \hline
 
  $ \Lambda(c,\tilde{c})$ & Sect. \ref{subsec:1.1.3}; \lemref{Estimate_speed_w} &  & $\Lambda(c,\tilde{c})= \Lambda_\delta(c,\tilde{c})\big|_{\delta=0}$\\
 \hline

  $\overline{w_{c,\tilde{c}}}$ & Sect. \ref{subsec:1.1.3} &  & \\
 \hline

 $\underline{u_{\delta,\zeta,\kappa}},\overline{v_{\delta,\zeta}}$ & \propref{Sub-solution_nonexistence} & Sect. \ref{subsec:3.1}  & Fig. \ref{fig:1.4}\\
 \hline

 $\overline{u_{\delta}},\underline{v_{\delta}}$ & \propref{Super-solution_compactly_supported} & Sect. \ref{subsec:3.2.2}  & Fig. \ref{fig:1.5}\\
 \hline

 $(\overline{u_{\delta}},\underline{v_{\delta}}),(\underline{u_{\delta}},\overline{v_{\delta}})$ & \propref{Super-sub-solution_terraces} & Sect. \ref{subsec:3.3}  & Fig. \ref{fig:1.6a} and Fig. \ref{fig:1.6b}\\
 \hline

$\lambda_v(c)$   & Sect. \ref{subsec:4.1.1}    & Sect. \ref{subsec:4.3.2} and \ref{subsec:4.5.1}&   $d\lambda^2 -c\lambda + r =0$\\

\hline

$a_\delta$ &   Sect. \ref{subsec:4.1.2}  &    & $a_\delta \to a$ as $\delta \to 0$ \\

\hline

$\lambda_\delta(c)$ &Sect. \ref{subsec:4.1.3} & & $\lambda^2 - c\lambda + (1-a_\delta)=0$ \\

\hline

 $c^{\delta}_{\textup{LLW}}$     &Sect. \ref{subsec:4.1.4}; \lemref{Continuity_of_underline_c_delta} & &  Minimal wave speed of $P(u,v) = F_\delta(u,v)$\\

\hline

 $F_\delta$&   Sect. \ref{subsec:4.1.4}  & & \\

\hline

 $(\overline{\varphi_{\delta,c}}, \underline{\psi_{\delta,c}})$ & Sect. \ref{subsec:4.1.5}  & Sect. \ref{subsec:4.3.1} and \ref{subsec:4.4.1}   & Monotone wave profile of $P(u,v) = F_\delta(u,v)$ \\

\hline

 $\underline{\theta_{\delta,c,A}}$ & Sect. \ref{subsec:4.1.6}; \lemref{theta_sub-solution}  & Sect. \ref{subsec:4.3.1} and \ref{subsec:4.4.1} &  (\ref{eq:theta_sub-solution})\\

\hline

 $\underline{\omega_{\delta,R}}$, $R^{\omega}_{\delta}$, $x_{\delta,R}$ & Sect. \ref{subsec:4.1.7}; \lemref{estimates_omega}  & Sect. \ref{subsec:4.4.1}& (\ref{eq:4.1.7}) \\
 \hline
 
 $\underline{\pi_{\delta,c,h}}$, $h^\star$ & Sect. \ref{subsec:4.1.8}; \lemref{pi_subsolution_and_estimate} & Sect. \ref{subsec:4.3.1} & (\ref{eq:4.1.8})\\
 
  \hline
 
 $\underline{\beta_{c,B,\eta}}, \xi_\beta, K_\beta$ & Sect. \ref{subsec:4.1.9}; \lemref{beta_subsolution} & Sect. \ref{subsec:4.3.1} & Exp. decay of $\underline{v}$ at $+\infty$\\
 
  \hline
  
 $\underline{\alpha_l}, L_\alpha, x_l$ & Sect. \ref{subsec:4.1.10}; \lemref{estimates_alpha} & Sect. \ref{subsec:4.5.1} & (\ref{eq:4.1.10})\\  
 \hline
 
 $\underline{\chi_c}$ & Sect. \ref{subsec:4.1.11} & Sect. \ref{subsec:4.3.2} and \ref{subsec:4.5.1} & (\ref{eq:4.1.11}) \\
 \hline
 
 $f_\delta(c), \Lambda_\delta(c,\tilde{c})$ & Sect. \ref{subsec:4.1.12} & Beginning, Sect. \ref{subsec:4.4} & (\ref{eq:Lambda_delta}) \\
 \hline

 $\overline{w_{\delta,c,\tilde{c}}}$ & Sect. \ref{subsec:4.1.13}; \lemref{w_super-solution} & Sect. \ref{subsec:4.3.1} and \ref{subsec:4.4.1} & (\ref{eq:w_super-solution}) \\
 \hline

 $\underline{w_{c,\tilde{c},A,\eta}}, X_w$ & Sect. \ref{subsec:4.1.14}; \lemref{w_sub-solution} & Sect. \ref{subsec:4.3.2} & (\ref{eq:w_sub-solution}) \\ 
 
 \hline
 
 $\underline{z_{\delta,c,\tilde{c}}}, R_z$ & Sect. \ref{subsec:4.1.15}; \lemref{z_sub-solution} & Sect. \ref{subsec:4.5.1} & (\ref{eq:z_sub-solution}) \\ 
 
 \hline
 
 $\lambda^{-\infty}(c)$ & \lemref{Exponential_estimates_super-critical_waves} & & \\ 
 
 \hline

 \end{tabular}
 
 }
 
 \smallskip
 
 {\small
 
\noindent  \begin{tabular}{ |p{1.4cm}||p{1.8cm}|p{2.1cm}|p{6.2cm}|  }
 \hline
 \multicolumn{4}{|c|}{\textbf{List of Intersection Points}} \\
 \hline
\textit{Symbol} & \textit{Defined in} & \textit{Used in} & \textit{Relation}\\
 \hline
$x_0(t)$, $\zeta_0$ & \lemref{interface_0} & Sect. \ref{subsec:4.3.2} &  $\underline{\chi_c}(x_0(t) - ct + \zeta_0) = \underline{w_{c,\tilde{c},A,\eta}}(t,x_0(t))$ \\ 

\hline 

$\xi_{1,\kappa}, \zeta_{1,\kappa}$, $A_\kappa$ & \lemref{interface_1} & Sect. \ref{subsec:4.3.1} and Sect. \ref{subsec:4.4.1} &  $\underline{\theta_{\delta,c,A_\kappa}}(\xi_{1,\kappa}) = \underline{\psi_{\delta,c}}(\xi_{1,\kappa}-\zeta_{1,\kappa})$ \\

\hline

$x_2(t), \zeta_2$ & \lemref{interface_2} & Sect. \ref{subsec:4.3.1} and Sect. \ref{subsec:4.4.1} &  $\overline{\varphi_{\delta,c}}(x_2(t) - ct)= \overline{w_{\delta,c,\tilde{c}}}(t, x_2(t) - \zeta_2)$ \\

\hline

\multirow{2}{*}{$x_3(t), \zeta_3$}& \lemref{interface_3_small_c_1} & Sect. \ref{subsec:4.4.1} & 
 ${\underline{\psi_{\delta,c}}(\hat{x}_3(t) - ct)} \hspace{5cm}$ $=\underline{\omega_{\delta,R_\delta}}(\hat{x}_3(t) - (2\sqrt{r(1-2\delta)d} - \delta) t - \hat{\zeta}_3)$ 
\\

\cline{2-4}

    & \lemref{interface_3_large_c_1} & Sect. \ref{subsec:4.3.1} &  $\underline{\psi_{\delta,c}}(x_3(t) - ct) = \underline{\pi_{\delta,\tilde{c},h}}(x_3(t) - \tilde{c}t - \zeta_3)$ \\

\hline

$\xi_4, \zeta_4$ & \lemref{interface_4} & Sect. \ref{subsec:4.3.1} &  $\underline{\pi_{\delta,c,h}}(\xi_4) = \underline{\beta_{c,B,\eta}}(\xi_4 - \zeta_4)$\\

\hline

$x_{0,\kappa}(t), \zeta_0$ & \lemref{interface_0_nonexistence} & Sect. \ref{subsec:4.5.1} & $\underline{\alpha_L}(x_{0,\kappa}(t)) = \underline{\chi_c}(x_{0,\kappa}(t) - ct - \zeta_0)$\\

\hline

$x_1(t)$ & \lemref{interface_1_nonexistence} & Sect. \ref{subsec:4.5.1} & $\underline{\chi_c}(x_1(t)-ct) = \frac{\underline{\chi_c}(\zeta)}{\underline{z_{c,\tilde{c},\delta}}(0,X_z)}\underline{z_{c,\tilde{c},\delta}}(t, x_1(t) - \zeta)$\\

\hline
 
\end{tabular}
}

\subsubsection{$\lambda_{v}$}\label{subsec:4.1.1}

The function $\lambda_{v}$ is defined as

\[
\begin{matrix}\lambda_{v}: & \left[2\sqrt{rd},+\infty\right) & \to & \left(0,\sqrt{\frac{r}{d}}\right]\\
 & c & \mapsto & \frac{1}{2d}\left(c-\sqrt{c^{2}-4rd}\right)
\end{matrix}.
\]

\subsubsection{$a_{\delta}$} \label{subsec:4.1.2}
For all $\delta\in\left[0,\frac{1}{2}\right)$, we denote
\[
a_{\delta}=\frac{\left(1-2\delta\right)a}{1+\delta}.
\]
Notice that $a_0=a$ and that $\delta\mapsto a_{\delta}$ is decreasing.

\subsubsection{$\lambda_{\delta}$} \label{subsec:4.1.3}

For all $\delta\in\left[0,\frac{1}{2}\right]$, the function $\lambda_{\delta}$
is defined as
\[
\begin{matrix}\lambda_{\delta}: & \left[2\sqrt{1-a_{\delta}},+\infty\right) & \to & \left(0,\sqrt{1-a_{\delta}}\right]\\
 & c & \mapsto & \frac{1}{2}\left(c-\sqrt{c^{2}-4\left(1-a_{\delta}\right)}\right)
\end{matrix}.
\]
The family $\left(\lambda_{\delta}\right)_{\delta\in\left[0,\frac{1}{2}\right]}$
is continuous and increasing in $\delta$. Note that $\lambda_{0}=\lambda$, the latter being introduced in \subsecref{1.1.3}.

\subsubsection{$c_{\textup{LLW}}^{\delta}$}\label{subsec:4.1.4}

For all $\delta\in\left[0,\frac{1}{2}\right)$, $c_{\textup{LLW}}^{\delta}$
denotes the minimal wave speed of the problem $P\left(u,v\right)=F_{\delta}\left(u,v\right)$,
where 
\[
F_{\delta}:\left(u,v\right)\mapsto\left(\begin{matrix}u\left(1+\delta-u-av\right)\\
rv\left(1-2\delta-v-bu\right)
\end{matrix}\right).
\]

Notice that $\left(u,v\right)$ is a solution of $P\left(u,v\right)=F_{\delta}\left(u,v\right)$
if and only if 
\[
\left(U,V\right):\left(t,x\right)\mapsto\left(\frac{u}{1+\delta},\frac{v}{1-2\delta}\right)\left(\frac{t}{1+\delta},\frac{x}{\sqrt{1+\delta}}\right)
\]
 is a solution of 
\[
P\left(U,V\right)=\left(\begin{matrix}U\left(1-U-\frac{\left(1-2\delta\right)a}{1+\delta}V\right)\\
\frac{\left(1-2\delta\right)r}{1+\delta}V\left(1-V-\frac{\left(1+\delta\right)b}{1-2\delta}U\right)
\end{matrix}\right).
\]
Therefore
\[
c_{\textup{LLW}}^{\delta} = \sqrt{1+\delta}\hat{c}_{\textup{LLW}}^{\delta},
\]
where $\hat{c}_{\textup{LLW}}^{\delta}$ is the minimal wave speed of \eqref{related_competition_diffusion_system}
where $\left(r,a,b\right)$ is replaced by $\left(\frac{\left(1-2\delta\right)r}{1+\delta},a_{\delta},\frac{\left(1+\delta\right)b}{1-2\delta}\right)$.
As such, $c^\delta_{\textup{LLW}}$ satisfies 
\[
2\sqrt{\left(1-a_{\delta}\right)\left(1+\delta\right)}\leq c_{\textup{LLW}}^{\delta}\leq 2\sqrt{1+\delta}.
\]

\subsubsection{$\left(\overline{\varphi_{\delta,c}},\underline{\psi_{\delta,c}}\right)$}\label{subsec:4.1.5}

For all $\delta\in\left[0,\frac{1}{2}\right)$ and $c\geq c_{\textup{LLW}}^{\delta}$,
$\left(\overline{\varphi_{\delta,c}},\underline{\psi_{\delta,c}}\right)$
denotes a component-wise monotonic profile of traveling wave with
speed $c$ for the problem $P\left(u,v\right)=F_{\delta}\left(u,v\right)$, connecting $\left(1+\delta,0\right)$ to 
$\left(0,1-2\delta\right)$ and satisfying the normalization $\underline{\psi_{\delta,c}}\left(0\right)=\frac{1-2\delta}{2}$.

The existence of such a profile is well-known (and proved for instance in \cite{Li_Weinberger_}). 
In fact, in the appendix, we will prove that any profile of traveling wave is component-wise monotonic and show that the 
condition 
\[
\frac{c+\sqrt{c^2+4rd}}{2d}\geq\frac{c-\sqrt{c^2-4\left(1-a\right)}}{2}
\]
implies the uniqueness, up to translation, of the profile associated with a particular speed $c\geq c_{\textup{LLW}}$. However
these results are not actually required here. What is required indeed is the forthcoming \lemref{Exponential_estimates_super-critical_waves}.

\subsubsection{$\underline{\theta_{\delta,c,A}}$}\label{subsec:4.1.6}

For all $\delta\in\left[0,\frac{1}{2}\right)$, $c\geq c_{\textup{LLW}}^{\delta}$
and $A>0$, define the function 
\[
\underline{\theta_{\delta,c,A}}:\xi\mapsto A\textup{e}^{\frac{1}{2d}\left(\sqrt{c^{2}+4rd\left(b-1+\delta\right)}-c\right)\left(\xi-\xi_{\theta}\right)}-\textup{e}^{\frac{1}{2d}\left(-\sqrt{c^{2}+4rd\left(b-1+\delta\right)}-c\right)\left(\xi-\xi_{\theta}\right)}
\]
where the constant 
\[
\xi_{\theta}=\frac{d\ln A}{\sqrt{c^{2}+4rd\left(b-1+\delta\right)}}
\]
is fixed so that $\underline{\theta_{\delta,c,A}}\left(0\right)=0$.
This function is increasing in $\mathbb{R}$.

\subsubsection{$\underline{\omega_{\delta,R}}$ and $R_{\delta}^{\omega}$}\label{subsec:4.1.7}

For all $\delta\in\left[0,1\right)$ and all $R>0$ large enough,
$\underline{\omega_{\delta,R}}:\left[-R,R\right]\to\left[0,+\infty\right)$
denotes the unique nonnegative nonzero solution of
\begin{equation}\label{eq:4.1.7}
\left\{ \begin{matrix}-d\omega''-\left(2\sqrt{r\left(1-\delta\right)d}-\delta\right)\omega'=r\omega\left(1-\delta-\omega\right) & \text{in }\left(-R,R\right)\\
\omega\left(\pm R\right)=0
\end{matrix}\right..
\end{equation}
It is well-known that this problem admits a solution if and only if
$R$ is larger than or equal to some $R_{\delta}^{\omega}>0$. 

We extend the definition of $\underline{\omega_{\delta,R}}$ into
the whole real line by setting $\underline{\omega_{\delta,R}}\left(\xi\right)=0$
if $\left|\xi\right|>R$.

\subsubsection{$\underline{\pi_{\delta,c,h}}$}\label{subsec:4.1.8}

For all $\delta\in\left[0,1\right)$, $c\geq 2\sqrt{rd}$ and all $h \in \mathbb{R}$, $\underline{\pi_{\delta,c,h}}$ denotes 
\begin{equation}\label{eq:4.1.8}
\underline{\pi_{\delta,c,h}}:(\xi)\mapsto \underline{\pi_{\delta,c}}(\xi)+ h\xi,
\end{equation}
where $\underline{\pi_{\delta,c}}$ 
denotes the unique (decreasing) profile of traveling wave solution of 
\[
\partial_{t}v-d\partial_{xx}v=rv\left(1-\delta-v\right)
\]
connecting $0$ to $1-\delta$ at speed $c$ and satisfying $\underline{\pi_{\delta,c}}\left(0\right)=\frac{1-\delta}{2}$.

\subsubsection{$\underline{\beta_{c,B,\eta}}$}\label{subsec:4.1.9}

For all $c>2\sqrt{rd}$, $\eta\in\left(0,\frac{1}{d}\sqrt{c^{2}-4rd}\right)$
and $B>0$, $\underline{\beta_{c,B,\eta}}$ denotes 
\[
\underline{\beta_{c,B,\eta}}:\xi\mapsto\max\left(0,\textup{e}^{-\lambda_{v}\left(c\right)\left(\xi+\xi_{\beta}\right)}-K_{\beta}\textup{e}^{-\left(\lambda_{v}\left(c\right)+\eta\right)\left(\xi+\xi_{\beta}\right)}\right),
\]
where the constants 
\[
K_{\beta}=\frac{r\left(1+bB\right)}{\eta\left(\sqrt{c^{2}-4rd}-d\eta\right)}
\]
and $\xi_{\beta}=\frac{\ln K_{\beta}}{\eta}$ are fixed so that $\underline{\beta_{c,B,\eta}}$
is positive in $\left(0,+\infty\right)$ and null elsewhere.

\subsubsection{$\underline{\alpha_{l}}$ and $L_{\alpha}$}\label{subsec:4.1.10}

Similarly to the construction of $\underline{\omega_{\delta,R}}$
and $R_{\delta}^{\omega}$, we define $\underline{\alpha_{l}}:\mathbb{R}\to\left[0,+\infty\right)$
and $L_{\alpha}>0$ such that, for all $l\geq L_{\alpha}$, 
\begin{equation}\label{eq:4.1.10}
\left\{ \begin{matrix}-\underline{\alpha_{l}}''=\underline{\alpha_{l}}\left(1-a-\underline{\alpha_{l}}\right) & \text{in }\left(0,l\right)\\
\underline{\alpha_{l}}\left(0\right)=\underline{\alpha_{l}}\left(l\right)=0\\
\underline{\alpha_{l}}\left(x\right)=0 & \text{if }\left|x-\frac{l}{2}\right|>\frac{l}{2}
\end{matrix}\right..
\end{equation}

\subsubsection{$\underline{\chi_{c}}$}\label{subsec:4.1.11}

For all $c\geq2\sqrt{\frac{1-a}{2}}$, $\underline{\chi_{c}}$ denotes
the unique (decreasing) profile of traveling wave solution of 
\begin{equation}\label{eq:4.1.11}
\partial_{t}u-\partial_{xx}u=u\left(\frac{1-a}{2}-u\right)
\end{equation}
connecting $0$ to $\frac{1-a}{2}$ at speed $c$ and satisfying $\underline{\chi_{c}}\left(0\right)=\frac{1-a}{4}$.

\subsubsection{$f_{\delta}\left(c\right)$ and $\Lambda_{\delta}\left(c,\tilde{c}\right)$}\label{subsec:f_delta_Lambda_delta}\label{subsec:4.1.12}

For all $\delta\in\left[0,\frac{1}{2}\right)$, $f_{\delta}$ denotes
the function 
\[
\begin{matrix}f_{\delta}: & \left[2\sqrt{1-a_{\delta}},+\infty\right) & \to & \left(0,2\left(\sqrt{1-a_{\delta}}+\sqrt{a_{\delta}}\right)\right]\\
 & c & \mapsto & c-\sqrt{c^{2}-4\left(1-a_{\delta}\right)}+2\sqrt{a_{\delta}}
\end{matrix}.
\]
 Notice right now that provided $\tilde{c}-f_{\delta}\left(c\right)>-4\sqrt{a_{\delta}}$,
$\tilde{c}-f_{\delta}\left(c\right)$ has exactly the sign of $\tilde{c}^{2}-4\left(\lambda_{\delta}\left(c\right)\left(\tilde{c}-c\right)+1\right)$.
Indeed, by the fact that $\left(\lambda_{\delta}\left(c\right)\right)^2 - c\lambda_{\delta}\left(c\right) + 1-a_{\delta}=0$,
\begin{align*}
\tilde{c}^{2}-4\left(\lambda_{\delta}\left(c\right)\left(\tilde{c}-c\right)+1\right) & =\left(\tilde{c}-2\lambda_{\delta}\left(c\right)\right)^{2}-4\left(1-\lambda_{\delta}\left(c\right)c+\left(\lambda_{\delta}\left(c\right)\right)^{2}\right)\\
 & =\left(\tilde{c}-c+\sqrt{c^{2}-4\left(1-a_{\delta}\right)}\right)^{2}-4a_{\delta}\\
 & =\left(\tilde{c}-f_{\delta}\left(c\right)+2\sqrt{a_{\delta}}\right)^{2}-\left(2\sqrt{a_{\delta}}\right)^2\\
 & =\left(\tilde{c}-f_{\delta}\left(c\right)\right)\left(\tilde{c}-f_{\delta}\left(c\right)+4\sqrt{a_{\delta}}\right).
\end{align*}

For all $\delta\in\left[0,\frac{1}{2}\right)$, $\Lambda_{\delta}$
denotes the function
\[
\Lambda_{\delta}:\left(c,\tilde{c}\right)\mapsto\frac{1}{2}\left(\tilde{c}-\sqrt{\tilde{c}^{2}-4\left(\lambda_{\delta}\left(c\right)\left(\tilde{c}-c\right)+1\right)}\right).
\]
 Its domain is the set of all $\left(c,\tilde{c}\right)$ such that
$c\geq2\sqrt{1-a_{\delta}}$ and $\tilde{c}\geq\max\left(c,f_{\delta}\left(c\right)\right)$
and it is decreasing with respect to both $c$ and $\tilde{c}$. As
a function of $c$ only, with a fixed $\tilde{c}$, it bijectively maps $\left[2\sqrt{1-a_{\delta}},+\infty\right)$ onto 
\[
\left(\frac{1}{2}\left(\tilde{c}-\sqrt{\tilde{c}^{2}-4\left(a_{\delta}+1\right)}\right),\frac{1}{2}\left(\tilde{c}-\sqrt{\tilde{c}^{2}-4\left(\tilde{c}\sqrt{1-a_{\delta}}+2a_{\delta}-1\right)}\right)\right].
\] 
The family $\left(\Lambda_{\delta}\right)_{\delta\in\left[0,\frac{1}{2}\right)}$
is increasing. Recalling the earlier definition of $\Lambda$, we find $\Lambda_0 = \Lambda$. 

Finally, by construction, for all $\left(c,\tilde{c}\right)$
such that $c\geq2\sqrt{1-a_{\delta}}$ and $\tilde{c}\geq\max\left(c,f_{\delta}\left(c\right)\right)$,
$\Lambda_{\delta}\left(c,\tilde{c}\right)$ satisfies
\begin{equation}
\left(\Lambda_{\delta}\left(c,\tilde{c}\right)\right)^{2}-\tilde{c}\Lambda_{\delta}\left(c,\tilde{c}\right)+\lambda_{\delta}\left(c\right)\left(\tilde{c}-c\right)+1=0.\label{eq:Lambda_delta}
\end{equation}

\subsubsection{$\overline{w_{\delta,c,\tilde{c}}}$}\label{subsec:4.1.13}

For all $\delta\in\left[0,\frac{1}{2}\right)$, $c\geq c_{\textup{LLW}}^{\delta}$
and $\tilde{c}\geq\max\left(c,f_{\delta}\left(c\right)\right)$, $\overline{w_{\delta,c,\tilde{c}}}$
denotes the function
\[
\overline{w_{\delta,c,\tilde{c}}}:\left(t,x\right)\mapsto\textup{e}^{-\lambda_{\delta}\left(c\right)\left(\tilde{c}-c\right)t}\textup{e}^{-\Lambda_{\delta}\left(c,\tilde{c}\right)\left(x-\tilde{c}t\right)}.
\]
 In view of (\ref{eq:Lambda_delta}),
\[
\overline{w_{\delta,c,\tilde{c}}}\left(t,x\right)=\textup{e}^{\left(\left(\Lambda_{\delta}\left(c,\tilde{c}\right)\right)^{2}+1\right)t}\textup{e}^{-\Lambda_{\delta}\left(c,\tilde{c}\right)x}\text{ for all }\left(t,x\right)\in\left[0,+\infty\right)\times\mathbb{R}.
\]
Recalling the earlier definition of $\overline{w_{c,\tilde{c}}}$, we find $\overline{w_{c,\tilde{c}}}=\overline{w_{0,c,\tilde{c}}}$.

\subsubsection{$\underline{w_{c,\tilde{c},A,\eta}}$}\label{subsec:4.1.14}

For all $c\geq c_{\textup{LLW}}$, $\tilde{c}\geq c$ such that $\tilde{c}>f\left(c\right)$,
$\eta\in\left(0,\min\left(\Lambda\left(c,\tilde{c}\right),\sqrt{\tilde{c}^{2}-4\left(\lambda\left(c\right)\left(\tilde{c}-c\right)+1\right)}\right)\right)$
and $A>0$, $\underline{w_{c,\tilde{c},A,\eta}}$ denotes 
\[
\underline{w_{c,\tilde{c},A,\eta}}:\left(t,x\right)\mapsto\textup{e}^{-\lambda\left(c\right)\left(\tilde{c}-c\right)t}\max\left(0,\textup{e}^{-\Lambda\left(c,\tilde{c}\right)\left(x-\tilde{c}t+x_{w}\right)}-K_{w}\textup{e}^{-\left(\Lambda\left(c,\tilde{c}\right)+\eta\right)\left(x-\tilde{c}t+x_{w}\right)}\right),
\]
where 
\[
K_{w}=\max\left(1,\frac{1+aA}{\eta\left(\sqrt{\tilde{c}^{2}-4\left(\lambda\left(c\right)\left(\tilde{c}-c\right)+1\right)}-\eta\right)}\right)=\max\left(1,\frac{1+aA}{\eta\left(\tilde{c}-\eta-2\Lambda_{0}\left(c,\tilde{c}\right)\right)}\right)
\]
and $x_{w}=\frac{\ln K_{w}}{\eta}$ is fixed so that, for all $t\geq0$,
$x\mapsto\underline{w_{c,\tilde{c},A,\eta}}\left(t,x\right)$ is positive
in $\left(\tilde{c}t,+\infty\right)$, null elsewhere, increasing
in $\left(\tilde{c}t,\frac{\ln\left(\Lambda\left(c,\tilde{c}\right)+\eta\right)-\ln\left(\Lambda\left(c,\tilde{c}\right)\right)}{\eta}+\tilde{c}t\right)$
and decreasing in $\left(\frac{\ln\left(\Lambda\left(c,\tilde{c}\right)+\eta\right)-\ln\left(\Lambda\left(c,\tilde{c}\right)\right)}{\eta}+\tilde{c}t,+\infty\right)$.
Hereafter, the point where the global maximum is attained at $t=0$
is denoted $X_{w}$.

\subsubsection{$\underline{z_{\delta,c,\tilde{c}}}$}\label{subsec:4.1.15}

For all $c\geq c_{\textup{LLW}}$, $\tilde{c}\in\left(f\left(c\right)-4\sqrt{a},f\left(c\right)\right)$
and $\delta\in\left[0,\frac{1}{4}\left(-\tilde{c}^{2}+4\left(\lambda\left(c\right)\left(\tilde{c}-c\right)+1\right)\right)\right)$,
$\underline{z_{c,\tilde{c},\delta}}$ denotes the function defined
by
\[
\underline{z_{\delta,c,\tilde{c}}}\left(t,x\right)=\left\{ \begin{matrix}\textup{e}^{-\lambda\left(c\right)\left(\tilde{c}-c\right)t}\textup{e}^{-\frac{\tilde{c}}{2}\left(x-\tilde{c}t\right)}\sin\left(\frac{\pi}{2R_{z}}\left(x-\tilde{c}t\right)\right) & \text{if }x-\tilde{c}t\in\left[0,2R_{z}\right]\\
0 & \text{otherwise}
\end{matrix}\right..
\]
 where 
\[
R_{z}=\frac{\pi}{\sqrt{-\tilde{c}^{2}+4\left(\lambda\left(c\right)\left(\tilde{c}-c\right)+1-\delta\right)}}.
\]
Hereafter, the point where the global maximum is attained at $t=0$
is denoted $X_{z}$.

\subsection{Several useful lemmas}
\begin{lem}
\label{lem:Exponential_estimates_super-critical_waves} Let $c>c_{\textup{LLW}}$ and $\left(\varphi,\psi\right)$ be a profile 
of traveling wave solution of (\ref{eq:competition_diffusion_system}) with speed $c$. 

Then there exist $A>0$ and $B>0$ such that
\[
\varphi\left(\xi\right)=A\textup{e}^{-\lambda\left(c\right)\xi}+\text{h.o.t. as }\xi\to+\infty
\]
and
\[
\psi\left(\xi\right)=B\textup{e}^{\lambda^{-\infty}\left(c\right)\xi}+\text{h.o.t. as }\xi\to-\infty
\]
where
\[
\lambda^{-\infty}\left(c\right)=\frac{1}{2d}\left(\sqrt{c^{2}+4rd\left(b-1\right)}-c\right),
\]
\end{lem}
The proof of this lemma is quite lengthy. Therefore it is postponed to the appendix (see 
\corref{decays_near_stable_state} and \corref{Precise_decays_near_unstable_state}).

\begin{lem}
\label{lem:Continuity_of_underline_c_delta} The function $\delta\mapsto c_{\textup{LLW}}^{\delta}$
is continuous and nondecreasing in $\left[0,\frac{1}{2}\right)$.
\end{lem}

\begin{proof}
Recalling that 
\[
c_{\textup{LLW}}^{\delta} = \sqrt{1+\delta}\hat{c}_{\textup{LLW}}^{\delta},
\]
where $\hat{c}_{\textup{LLW}}^{\delta}$ is the minimal wave speed of the system (\ref{eq:related_competition_diffusion_system})
where $\left(r,a,b\right)$ is replaced by $\left(\frac{\left(1-2\delta\right)r}{1+\delta},a_{\delta},\frac{\left(1+\delta\right)b}{1-2\delta}\right)$, the continuity of $\delta\mapsto c_{\textup{LLW}}^{\delta}$ follows directly
from the theorem due to Kan-on \cite{Kan_On_1997} establishing the continuity of the spreading speed of 
(\ref{eq:related_competition_diffusion_system}) with respect to the parameters $\left(r,a,b\right)$.

The monotonicity follows from the comparison principle.
\end{proof}

\begin{lem}
\label{lem:estimates_omega}Let $\delta\in\left[0,1\right)$. Then
for all $R\geq R_{\delta}^{\omega}$,
\[
\max_{\left[-R,R\right]}\underline{\omega_{\delta,R}}<1-\delta.
\]

Furthermore, if $\delta>0$, then there exists $R_{\delta}\geq R_{\delta}^{\omega}$
such that, for all $R\geq R_{\delta}$,
\[
\max_{\left[-R,R\right]}\underline{\omega_{\delta,R}}\geq1-2\delta,
\]
and there exists a unique $x_{\delta,R}\in\left(-R,R\right)$ such
that $\underline{\omega_{\delta,R}}$ is increasing in $\left[-R,x_{\delta,R}\right]$,
decreasing in $\left[x_{\delta,R},R\right]$ and maximal at $x_{\delta,R}$.
\end{lem}

\begin{proof}
The first inequality follows very classically from the first and second
order conditions at any local maximum and from the strong maximum
principle.

The second inequality comes from the locally uniform convergence of
$\underline{\omega_{\delta,R}}$ to $1-\delta$ as $R\to+\infty$.
This fact is also well-known and its proof is not detailed here. 

Finally, the piecewise strict monotonicity comes from the inequality
\[
-d\underline{\omega_{\delta,R}}''-\left(2\sqrt{r\left(1-\delta\right)d}-\delta\right)\underline{\omega_{\delta,R}}'>0\text{ in }\left(-R,R\right),
\]
which implies the nonexistence of local minima. 
\end{proof}
The function $\underline{\alpha_{l}}$ satisfies of course a similar
property. 
\begin{lem}
\label{lem:estimates_alpha} For all $l\geq L_{\alpha}$,
\[
\max_{\left[0,l\right]}\underline{\alpha_{l}}<1-a.
\]

Furthermore, there exists $L\geq L_{\alpha}$ such that, for all $l\ge L$,
\[
\max_{\left[0,l\right]}\underline{\alpha_{l}}\geq\frac{1-a}{2},
\]
and there exists a unique $x_{l}\in\left(-l,l\right)$ such that $\underline{\alpha_{l}}$
is increasing in $\left[0,x_{l}\right]$, decreasing in $\left[x_{l},l\right]$
and maximal at $x_{l}$.
\end{lem}

\begin{lem}
\label{lem:w_super-solution} For all $\delta\in\left[0,\frac{1}{2}\right)$,
$c\geq c_{\textup{LLW}}^{\delta}$ and $\tilde{c}\geq\max\left(c,f_{\delta}\left(c\right)\right)$,
$\overline{w_{\delta,c,\tilde{c}}}$ satisfies
\[
\partial_{t}\overline{w_{\delta,c,\tilde{c}}}-\partial_{xx}\overline{w_{\delta,c,\tilde{c}}}=\overline{w_{\delta,c,\tilde{c}}}\text{ in }\left[0,+\infty\right)\times\mathbb{R}.
\]
\end{lem}

\begin{proof}
The following equality being straightforward, 
\begin{equation}\label{eq:w_super-solution}
\partial_{t}\overline{w_{\delta,c,\tilde{c}}}-\partial_{xx}\overline{w_{\delta,c,\tilde{c}}}-\overline{w_{\delta,c,\tilde{c}}}=\left(-\lambda_{\delta}\left(c\right)\left(\tilde{c}-c\right)+\tilde{c}\Lambda_{\delta}\left(c,\tilde{c}\right)-\left(\Lambda_{\delta}\left(c,\tilde{c}\right)\right)^{2}-1\right)\overline{w_{\delta,c}},
\end{equation}
the conclusion follows from (\ref{eq:Lambda_delta}).
\end{proof}
Quite similarly, we have the following lemma.
\begin{lem}
\label{lem:z_sub-solution} For all $c\geq c_{\textup{LLW}}$, $\tilde{c}\in\left(f\left(c\right)-4\sqrt{a},f\left(c\right)\right)$
and $\delta\in\left[0,\frac{1}{4}\left(-\tilde{c}^{2}+4\left(\lambda\left(c\right)\left(\tilde{c}-c\right)+1\right)\right)\right)$,
$\underline{z_{\delta,c,\tilde{c}}}$ satisfies
\begin{equation}\label{eq:z_sub-solution}
\partial_{t}\underline{z_{\delta,c,\tilde{c}}}-\partial_{xx}\underline{z_{\delta,c,\tilde{c}}}=\left(1-\delta\right)\underline{z_{\delta,c,\tilde{c}}}\text{ in }\left[0,+\infty\right)\times\mathbb{R}.
\end{equation}
\end{lem}

\begin{proof}
It suffices to verify 
\[
-\lambda\left(c\right)\left(\tilde{c}-c\right)+\frac{\tilde{c}^{2}}{4}+\left(\frac{\pi}{2R_{z}}\right)^{2}-\left(1-\delta\right)=0,
\]
which, in view of the definition of $R_{z}$, is equivalent to
\[
-\lambda\left(c\right)\left(\tilde{c}-c\right)+\frac{\tilde{c}^{2}}{4}+\left(\frac{-\tilde{c}^{2}+4\lambda\left(c\right)\left(\tilde{c}-c\right)+4\left(1-\delta\right)}{4}\right)-\left(1-\delta\right)=0.
\]
The last statement obviously holds.
\end{proof}

\begin{lem}
\label{lem:w_sub-solution} For all $c\geq c_{\textup{LLW}}$, $\tilde{c}\geq c$
such that $\tilde{c}>f\left(c\right)$, $\eta\in\left(0,\min\left(\Lambda\left(c,\tilde{c}\right),\sqrt{\tilde{c}^{2}-4\left(\lambda\left(c\right)\left(\tilde{c}-c\right)+1\right)}\right)\right)$
and $A>0$, the function $\underline{w_{c,\tilde{c},A,\eta}}$ satisfies,
for all $\sigma\geq\eta$,
\begin{equation}\label{eq:w_sub-solution}
\partial_{t}\underline{w_{c,\tilde{c},A,\eta}}-\partial_{xx}\underline{w_{c,\tilde{c},A,\eta}}\leq\underline{w_{c,\tilde{c},A,\eta}}\left(1-\underline{w_{c,\tilde{c},A,\eta}}-aA\textup{e}^{-\sigma\left(x-\tilde{c}t+x_{w}\right)}\right).
\end{equation}
\end{lem}

\begin{rem*}
The above inequality is to be understood in the weak sense associated with
generalized sub-solutions.
\end{rem*}
\begin{proof}
For $x - \tilde{c}t  <0$, $\underline{w_{c,\tilde{c},A,\eta}}$ is trivial and the inequality obviously holds.  

We focus on the case $x-\tilde{c}t>0$, where $\underline{w_{c,\tilde{c},A,\eta}}$ reduces to
\[
\left(t,x\right)\mapsto\textup{e}^{-\lambda\left(c\right)\left(\tilde{c}-c\right)t}\left(\textup{e}^{-\Lambda\left(c,\tilde{c}\right)\left(x-\tilde{c}t+x_{w}\right)}-K_{w}\textup{e}^{-\left(\Lambda\left(c,\tilde{c}\right)+\eta\right)\left(x-\tilde{c}t+x_{w}\right)}\right).
\]

First, differentiating, we find:
\begin{align*}
\partial_{t}\underline{w_{c,\tilde{c},A,\eta}}\left(t,x\right) & =-\lambda\left(c\right)\left(\tilde{c}-c\right)\underline{w_{c,\tilde{c},A,\eta}}\left(t,x\right)\\
 & \quad+\tilde{c}\Lambda\left(c,\tilde{c}\right)\textup{e}^{-\lambda\left(c\right)\left(\tilde{c}-c\right)t}\textup{e}^{-\Lambda\left(c,\tilde{c}\right)\left(x-\tilde{c}t+x_{w}\right)}\\
 & \quad-K_{w}\left(\tilde{c}\left(\Lambda\left(c,\tilde{c}\right)+\eta\right)\right)\textup{e}^{-\lambda\left(c\right)\left(\tilde{c}-c\right)t}\textup{e}^{-\left(\Lambda\left(c,\tilde{c}\right)+\eta\right)\left(x-\tilde{c}t+x_{w}\right)},
\end{align*}
\[
\partial_{xx}\underline{w_{c,\tilde{c},A,\eta}}=\textup{e}^{-\lambda\left(c\right)\left(\tilde{c}-c\right)t}\left(\left(\Lambda\left(c,\tilde{c}\right)\right)^{2}\textup{e}^{-\Lambda\left(c,\tilde{c}\right)\left(x-\tilde{c}t+x_{w}\right)}-K_{w}\left(\Lambda\left(c,\tilde{c}\right)+\eta\right)^{2}\textup{e}^{-\left(\Lambda\left(c,\tilde{c}\right)+\eta\right)\left(x-\tilde{c}t+x_{w}\right)}\right),
\]
so that the auxiliary function 
\[
Q:\left(t,x\right)\mapsto \textup{e}^{\lambda\left(c\right)\left(\tilde{c}-c\right)t}\left(-\partial_{t}\underline{w_{c,\tilde{c},A,\eta}}+\partial_{xx}\underline{w_{c,\tilde{c},A,\eta}}+\underline{w_{c,\tilde{c},A,\eta}}\right)\left(t,x\right)
\]
satisfies
\begin{align*}
Q\left(t,x\right) & =\left(\lambda\left(c\right)\left(\tilde{c}-c\right)-\tilde{c}\Lambda\left(c,\tilde{c}\right)+\left(\Lambda\left(c,\tilde{c}\right)\right)^{2}+1\right)\textup{e}^{-\Lambda\left(c,\tilde{c}\right)\left(x-\tilde{c}t+x_{w}\right)}\\
 & \quad-K_{w}\left(\lambda\left(c\right)\left(\tilde{c}-c\right)-\tilde{c}\left(\Lambda\left(c,\tilde{c}\right)+\eta\right)+\left(\Lambda\left(c,\tilde{c}\right)+\eta\right)^{2}+1\right)\textup{e}^{-\left(\Lambda\left(c,\tilde{c}\right)+\eta\right)\left(x-\tilde{c}t+x_{w}\right)}.
\end{align*}
Using (\ref{eq:Lambda_delta}), it follows 
\[
Q\left(t,x\right)=K_{w}\eta\left(\tilde{c}-\eta-2\Lambda\left(c,\tilde{c}\right)\right)\textup{e}^{-\left(\Lambda\left(c,\tilde{c}\right)+\eta\right)\left(x-\tilde{c}t+x_{w}\right)},
\]
that is, recalling the definition of $\Lambda_{\delta}\left(c,\tilde{c}\right)$
as well as that of $K_{w}$, 
\[
Q\left(t,x\right)\geq\left(1+aA\right)\textup{e}^{-\left(\Lambda\left(c,\tilde{c}\right)+\eta\right)\left(x-\tilde{c}t+x_{w}\right)}.
\]

Next, getting rid of all the negative terms and using 
$\textup{e}^{-\lambda\left(c\right)\left(\tilde{c}-c\right)t}\leq 1$, we find
\[
\textup{e}^{\lambda\left(c\right)\left(\tilde{c}-c\right)t}\underline{w_{c,\tilde{c},A,\eta}}\left(\underline{w_{c,\tilde{c},A,\eta}}+aA\textup{e}^{-\sigma\left(x-\tilde{c}t+x_{w}\right)}\right)\leq\textup{e}^{-\Lambda\left(c,\tilde{c}\right)\left(x-\tilde{c}t+x_{w}\right)}\left(\textup{e}^{-\Lambda\left(c,\tilde{c}\right)\left(x-\tilde{c}t+x_{w}\right)}+aA\textup{e}^{-\sigma\left(x-\tilde{c}t+x_{w}\right)}\right)
\]

Finally, using $x>\tilde{c}t$, $x_{w}=\frac{1}{\eta}\ln K_{w}\geq 0$
as well as the assumption $0 < \eta \leq \min\{ \Lambda(c,\tilde{c}), \sigma\}$, we find 
\begin{align*}
\textup{e}^{\eta\left(x-\tilde{c}t+x_{w}\right)}\left(\textup{e}^{-\Lambda\left(c,\tilde{c}\right)\left(x-\tilde{c}t+x_{w}\right)}+aA\textup{e}^{-\sigma\left(x-\tilde{c}t+x_{w}\right)}\right) & \leq\textup{e}^{-\left(\Lambda\left(c,\tilde{c}\right)-\eta\right)x_{w}}+aA\textup{e}^{-\left(\sigma-\eta\right)x_{w}}\\
 & \leq1+aA
\end{align*}
and the proof is therefore ended.
\end{proof}
With an analogous proof, we obtain directly the following lemma.
\begin{lem}
\label{lem:beta_subsolution}For all $c>2\sqrt{rd}$, $\eta\in\left(0,\min\left(\lambda_{v}\left(c\right),\frac{1}{d}\sqrt{c^{2}-4rd}\right)\right)$
and $B>0$, $\underline{\beta_{c,B,\eta}}$ satisfies, for all $\sigma\geq\eta$,
\[
-d\underline{\beta_{c,B,\eta}}''-c\underline{\beta_{c,B,\eta}}'\leq r\underline{\beta_{c,B,\eta}}\left(1-\underline{\beta_{c,B,\eta}}-bB\textup{e}^{-\sigma\left(\xi+\xi_{\beta}\right)}\right)\text{ in }\left(\mathbb{R}\backslash\left\{ 0\right\} \right).
\]
\end{lem}

\begin{lem}
\label{lem:pi_subsolution_and_estimate} For all $\delta\in\left(0,1\right)$,
$c>2\sqrt{rd}$ and $h>0$, $\underline{\pi_{\delta,c,h}}$
satisfies 
\[
-d\underline{\pi_{\delta,c,h}}''-c\underline{\pi_{\delta,c,h}}'\leq r\underline{\pi_{\delta,c,h}}\left(1-\delta-\underline{\pi_{\delta,c,h}}\right)\text{ in }\left[-\sqrt{\frac{c}{rh}},\sqrt{\frac{c}{rh}}\right].
\]

Furthermore, there exists $h^{\star}>0$ such that, for all $h\in\left(0,h^{\star}\right]$,
\[
\max_{\left[-\sqrt{\frac{c}{rh}},0\right]}\underline{\pi_{\delta,c,h}}\geq1-2\delta,
\]
\[
\max_{\left[-\sqrt{\frac{c}{rh}},0\right]}\underline{\pi_{\delta,c,h}}>\max\left(\underline{\pi_{\delta,c,h}}\left(0\right),\underline{\pi_{\delta,c,h}}\left(-\sqrt{\frac{c}{rh}}\right)\right).
\]
\end{lem}

\begin{rem*}
It should be achievable to prove that the global maximum of $\underline{\pi_{\delta,c,h}}$
in $\left(-\sqrt{\frac{c}{rh}},0\right)$ is actually unique and that
$\underline{\pi_{\delta,c,h}}$ is increasing in $\left(-\infty,\xi^{\star}\right)$
and decreasing in $\left(\xi^{\star},0\right)$ but this is really
unnecessary for our purpose.
\end{rem*}
\begin{proof}
Recalling that $\underline{\pi_{\delta,c,h}}(\xi)= \underline{\pi_{\delta,c}}(\xi) + h\xi$ (see Subsection \ref{subsec:4.1.8}), we have 
\begin{align*}
&\quad -d\underline{\pi_{\delta,c,h}}''\left(\xi\right)-c\underline{\pi_{\delta,c,h}}'\left(\xi\right) \\
& =r\underline{\pi_{\delta,c}}\left(\xi\right)\left(1-\delta-\underline{\pi_{\delta,c}}\left(\xi\right)\right)-ch\\
 & =r\underline{\pi_{\delta,c,h}}\left(\xi\right)\left(1-\delta-\underline{\pi_{\delta,c,h}}\left(\xi\right)\right)-hr\left(\xi\left(1-\delta-\underline{\pi_{\delta,c}}\left(\xi\right)\right)+\frac{c}{r}-\underline{\pi_{\delta,c,h}}\left(\xi\right)\xi\right)\\
 & =r\underline{\pi_{\delta,c,h}}\left(\xi\right)\left(1-\delta-\underline{\pi_{\delta,c,h}}\left(\xi\right)\right)-hr\left(-h\xi^{2}+\left(1-\delta-2\underline{\pi_{\delta,c}}\left(\xi\right)\right)\xi+\frac{c}{r}\right).
\end{align*}
It is easily verified that, in $\left[-\sqrt{\frac{c}{rh}},\sqrt{\frac{c}{rh}}\right]$,
\[
-h\xi^{2}+\left(1-\delta-2\underline{\pi_{\delta,c}}\left(\xi\right)\right)\xi+\frac{c}{r}>-h\xi^{2}+\frac{c}{r}\geq 0,
\]
where we used the facts
$$
\underline{\pi_{\delta,c,h}} > \frac{1-\delta}{2} \quad \text{ for }\xi <0, \quad \text{ and }\quad \underline{\pi_{\delta,c,h}} < \frac{1-\delta}{2} \quad \text{ for }\xi >0.
$$
And the stated differential inequality is established.

The maximum of $\underline{\pi_{\delta,c,h}}$ in $\left[-\sqrt{\frac{c}{rh}},0\right]$
is larger than or equal to 
\[
\underline{\pi_{\delta,c,h}}\left(-\sqrt{\frac{c}{rh}}\right)=\underline{\pi_{\delta,c}}\left(-\sqrt{\frac{c}{rh}}\right)-\sqrt{\frac{ch}{r}},
\]
which is itself larger than or equal to $1-2\delta$ if $h$ is small enough. 

Finally, since $\underline{\pi'_{\delta,c}}\left(-\sqrt{\frac{c}{rh}}\right)$
vanishes exponentially as $h\to 0$,
\[
\underline{\pi'_{\delta,c,h}}\left(-\sqrt{\frac{c}{rh}}\right)=\underline{\pi'_{\delta,c}}\left(-\sqrt{\frac{c}{rh}}\right)+h>0, \quad  \text{ and }\quad  \underline{\pi_{\delta,c,h}}'(0) =  \underline{\pi_{\delta,c}}'(0) + h <0,
\]
for all sufficiently small $h$. This implies that the values at $\xi =0$ and $-\sqrt{\frac{c}{rh}}$ are smaller than the 
aforementioned maximum.
\end{proof}

\begin{lem}
\label{lem:theta_sub-solution} For all $\delta\in\left[0,\frac{1}{2}\right)$,
$c\geq c_{\textup{LLW}}^{\delta}$ and $A>0$, $\underline{\theta_{\delta,c,A}}$
satisfies
\begin{equation}\label{eq:theta_sub-solution}
-d\underline{\theta_{\delta,c,A}}''-c\underline{\theta_{\delta,c,A}}'-r\underline{\theta_{\delta,c,A}}\left(1-\delta-b\right)=0\text{ in }\mathbb{R}.
\end{equation}
\end{lem}

\begin{proof}
Note that $\underline{\theta_{\delta,c,A}}$ is a
linear combination of $\xi\mapsto\textup{e}^{\frac{1}{2d}\left(\pm\sqrt{c^{2}+4rd(b-1+\delta)}-c\right)\xi},$
where $\frac{1}{2d}\left(\pm\sqrt{c^{2}+4rd(b-1+\delta)}-c\right)$
are the two distinct roots of the characteristic polynomial associated
with the above linear ODE (\ref{eq:theta_sub-solution}). 
\end{proof}

\begin{lem}
\label{lem:interface_0} For all $c>2\sqrt{1-a}$, $\tilde{c}\geq c$
such that 
$$
\tilde{c}>f\left(c\right), \quad \eta\in\left(0,\min\left(\sqrt{\tilde{c}^{2}-4\left(\lambda\left(c\right)\left(\tilde{c}-c\right)+1\right)}\right),\lambda_{v}\left(\tilde{c}\right)\right), \quad \text{ and }\quad 
A>0,
$$
there exists $\zeta_{0}\in\mathbb{R}$ such that the equation
\[
\underline{\chi_{c}}\left(x-ct+\zeta_{0}\right)=\underline{w_{c,\tilde{c},A,\eta}}\left(t,x\right)
\]
admits for all $t\geq0$ an isolated solution $x_{0}\left(t\right)\in\mathbb{R}$
such that 
\begin{enumerate}
\item $\underline{\chi_{c}}\left(x-ct+\zeta_{0}\right)>\underline{w_{c,\tilde{c},A,\eta}}\left(t,x\right)$
in a left-sided neighborhood of $x_{0}\left(t\right)$;
\item $\underline{\chi_{c}}\left(x-ct+\zeta_{0}\right)<\underline{w_{c,\tilde{c},A,\eta}}\left(t,x\right)$
in a right-sided neighborhood of $x_{0}\left(t\right)$;
\item $\tilde{c}t<x_{0}\left(t\right)<X_{w}+\tilde{c}t$.
\end{enumerate}
Furthermore, $x_{0}\in\mathscr{C}^{1}\left(\left[0,+\infty\right),\left(0,+\infty\right)\right)$.
\end{lem}

\begin{proof}
Recall from standard results on the KPP equation that, since $c>2\sqrt{1-a}$,
there exists $\zeta_{0,1}\in\mathbb{R}$ such that
\[
\underline{\chi_{c}}\left(x+\zeta_{0,1}\right)\sim\textup{e}^{-\lambda\left(c\right)x}\text{ as }x\to+\infty.
\]
Hence there exists $\zeta_{0}\in\mathbb{R}$ such that, for all $x\geq0$,
\[
\underline{\chi_{c}}\left(x+\zeta_{0}\right)\leq\frac{1}{2}e^{-\lambda(c)x}\max_{y\in\mathbb{R}}\underline{w_{c,\tilde{c},A,\eta}}\left(0,y\right)\leq\frac{1}{2}\max_{y\in\mathbb{R}}\underline{w_{c,\tilde{c},A,\eta}}\left(0,y\right)
\]
with $\max\limits _{y\in\mathbb{R}}\underline{w_{c,\tilde{c},A,\eta}}\left(0,y\right)$
uniquely attained at $X_{w}$.

From the intermediate value theorem and the respective strict monotonicities
of $\underline{\chi_{c}}$ in $\mathbb{R}$ and $x\mapsto\underline{w_{c,\tilde{c},A,\eta}}\left(0,x\right)$
in $\left[0,X_{w}\right]$, it clearly follows that $\underline{\chi_{c}}\left(x+\zeta_{0}\right)=\underline{w_{c,\tilde{c},A,\eta}}\left(0,x\right)$
admits a unique solution $x_{0}\left(0\right)$ in $\left(0,X_{w}\right)$.

Next, to define in the same way $x_{0}\left(t\right)$, it suffices
to verify that for all $t>0$, 
\[
\underline{w_{c,\tilde{c},A,\eta}}\left(t,X_{w}+\tilde{c}t\right)>\underline{\chi_{c}}\left(X_{w}+\left(\tilde{c}-c\right)t+\zeta_{0}\right).
\]
Since $X_{w}+\tilde{c}t\geq0$, it is \textit{a fortiori} sufficient
to verify that for all $t\geq0$,
\[
\textup{e}^{-\lambda\left(c\right)\left(\tilde{c}-c\right)t}\max_{x\in\mathbb{R}}\underline{w_{c,\tilde{c},A,\eta}}\left(0,x\right)>\frac{1}{2}\max_{x\in\mathbb{R}}\underline{w_{c,\tilde{c},A,\eta}}\left(0,x\right)\textup{e}^{-\lambda\left(c\right)\left(X_{w}+\left(\tilde{c}-c\right)t\right)}.
\]
This inequality reduces in fact to $2>\textup{e}^{-\lambda\left(c\right)X_{w}},$
which holds as $\lambda\left(c\right)$ and $X_{w}$ are both positive. The existence of $x_{0}\left(t\right)$ for
all $t>0$ follows.

Finally, the regularity of $x_{0}$ follows from the aforementioned
monotonicities and the implicit function theorem.
\end{proof}
\begin{lem}
\label{lem:interface_1} For all $\delta\in\left[0,\frac{1}{2}\right)$,
$c\geq c_{\textup{LLW}}^{\delta}$ and $\kappa\in\left(0,\delta\right]$, there
exists $\zeta_{1,\kappa}\in\mathbb{R}$ and $A_{\kappa}>0$ such that
the equation 
\[
\underline{\theta_{\delta,c,A_{\kappa}}}\left(\xi\right)=\underline{\psi_{\delta,c}}\left(\xi-\zeta_{1,\kappa}\right)
\]
admits an isolated solution $\xi_{1,\kappa}\in\mathbb{R}$ such that
\begin{enumerate}
\item $\underline{\theta_{\delta,c,A_{\kappa}}}\left(\xi\right)>\underline{\psi_{\delta,c}}\left(\xi-\zeta_{1,\kappa}\right)$ for $\xi$ 
in a left-sided neighborhood of $\xi_{1,\kappa}$;
\item $\underline{\theta_{\delta,c,A_{\kappa}}}\left(\xi\right)<\underline{\psi_{\delta,c}}\left(\xi-\zeta_{1,\kappa}\right)$ for $\xi$ 
in a right-sided neighborhood of $\xi_{1,\kappa}$;
\item $\underline{\psi_{\delta,c}}\left(\xi_{1,\kappa}-\zeta_{1,\kappa}\right)\leq\kappa$;
\item $\zeta_{1,\kappa}-\xi_{1,\kappa}\to+\infty$ as $\kappa\to0$.
\end{enumerate}
\end{lem}

\begin{proof}
Let $\delta$, $c$ and $\kappa$ be given as in the statement, define 
\[
\lambda^{-\infty}=\frac{1}{2d}\left(\sqrt{c^{2}+4rd\left(b-1+\left(b+2\right)\delta\right)}-c\right),
\]
\[
\lambda_{\theta}^{+}=\frac{1}{2d}\left(\sqrt{c^{2}+4rd\left(b-1+\delta\right)}-c\right),
\]
\[
\lambda_{\theta}^{-}=\frac{1}{2d}\left(-\sqrt{c^{2}+4rd\left(b-1+\delta\right)}-c\right),
\]
\[
\xi_{\theta}=\frac{d\ln A}{\sqrt{c^{2}+4rd\left(b-1+\delta\right)}}=\frac{\ln A}{\lambda_{\theta}^{+}-\lambda_{\theta}^{-}},
\]
and notice that 
\[
\lambda_{\theta}^{-}<0<\lambda_{\theta}^{+}<\lambda^{-\infty}.
\]

Let $\tilde{\kappa}\in\left(0,\kappa\right]$ such that $\left(1-\tilde{\kappa}\right)\lambda^{-\infty}>\lambda_{\theta}^{+}$.

In view of \lemref{Exponential_estimates_super-critical_waves},
\[
\lim_{\xi\to-\infty}\left(\frac{\underline{\psi_{\delta,c}}'\left(\xi\right)}{\underline{\psi_{\delta,c}}\left(\xi\right)}\right)=\lambda^{-\infty}.
\]
 Therefore, by monotonicity of $\underline{\psi_{\delta,c}}$, there
exists $\zeta_{\kappa}\in\mathbb{R}$ such that for all $\xi\leq0$,
\[
\underline{\psi_{\delta,c}}\left(\xi-\zeta_{\kappa}\right)\leq\kappa,
\]
\[
\left(1-\frac{\tilde{\kappa}}{2}\right)\lambda^{-\infty}\leq\frac{\underline{\psi_{\delta,c}}'\left(\xi-\zeta_{\kappa}\right)}{\underline{\psi_{\delta,c}}\left(\xi-\zeta_{\kappa}\right)}\leq\left(1+\frac{\tilde{\kappa}}{2}\right)\lambda^{-\infty}.
\]

Note that $\zeta_{\kappa}\to+\infty$ as $\kappa\to 0$. It remains to find $A>0$, $\zeta_{1}>\zeta_{\kappa}$
and $\xi_{1}\in\left(0,\zeta_{1}-\zeta_{\kappa}\right]$ such that
\[
\underline{\theta_{\delta,c,A}}\left(\xi_{1}\right)=\underline{\psi_{\delta,c}}\left(\xi_{1}-\zeta_{1}\right),
\]
\[
\frac{\underline{\theta_{\delta,c,A}}'\left(\xi_{1}\right)}{\underline{\theta_{\delta,c,A}}\left(\xi_{1}\right)}\leq\left(1-\tilde{\kappa}\right)\lambda^{-\infty}.
\]

For all $\xi\in\mathbb{R}$, 
\[
\underline{\theta_{\delta,c,A}}'\left(\xi\right)=A\lambda_{\theta}^{+}\textup{e}^{\lambda_{\theta}^{+}\left(\xi-\xi_{\theta}\right)}-\lambda_{\theta}^{-}\textup{e}^{\lambda_{\theta}^{-}\left(\xi-\xi_{\theta}\right)}>0,
\]
whence for all $\xi>0$ the condition
\[
\frac{\underline{\theta_{\delta,c,A}}'\left(\xi\right)}{\underline{\theta_{\delta,c,A}}\left(\xi\right)}<\left(1-\tilde{\kappa}\right)\lambda^{-\infty}
\]
 holds true if and only if 
\[
\left(1-\tilde{\kappa}\right)\lambda^{-\infty}>\lambda_{\theta}^{+}+\frac{\lambda_{\theta}^{+}-\lambda_{\theta}^{-}}{A\textup{e}^{\left(\lambda_{\theta}^{+}-\lambda_{\theta}^{-}\right)\left(\xi-\xi_{\theta}\right)}-1},
\]
that is if and only if 
\[
A\textup{e}^{\left(\lambda_{\theta}^{+}-\lambda_{\theta}^{-}\right)\left(\xi-\xi_{\theta}\right)}-1>\frac{\lambda_{\theta}^{+}-\lambda_{\theta}^{-}}{\left(1-\tilde{\kappa}\right)\lambda^{-\infty}-\lambda_{\theta}^{+}},
\]
that is if and only if $\xi>\xi_{1}$ where 
\[
\xi_{1}=\max\left(0,\xi_{\theta}+\frac{1}{\lambda_{\theta}^{+}-\lambda_{\theta}^{-}}\left(\ln\left(1+\frac{\lambda_{\theta}^{+}-\lambda_{\theta}^{-}}{\left(1-\tilde{\kappa}\right)\lambda^{-\infty}-\lambda_{\theta}^{+}}\right)-\ln A\right)\right).
\]
In view of the definition of $\xi_{\theta}$, 
\begin{align*}
\xi_{1} & =\max\left(0,\frac{1}{\lambda_{\theta}^{+}-\lambda_{\theta}^{-}}\ln\left(1+\frac{\lambda_{\theta}^{+}-\lambda_{\theta}^{-}}{\left(1-\tilde{\kappa}\right)\lambda^{-\infty}-\lambda_{\theta}^{+}}\right)\right)\\
 & =\frac{1}{\lambda_{\theta}^{+}-\lambda_{\theta}^{-}}\ln\left(1+\frac{\lambda_{\theta}^{+}-\lambda_{\theta}^{-}}{\left(1-\tilde{\kappa}\right)\lambda^{-\infty}-\lambda_{\theta}^{+}}\right).
\end{align*}
In particular, $\xi_{1}>0$ does not depend on $A$ and, by construction,
we have
\[
\frac{\underline{\theta_{\delta,c,A}}'\left(\xi\right)}{\underline{\theta_{\delta,c,A}}\left(\xi\right)}<\left(1-\tilde{\kappa}\right)\lambda^{-\infty}\text{ for all }\xi>\xi_{1},
\]
\[
\frac{\underline{\theta_{\delta,c,A}}'\left(\xi_{1}\right)}{\underline{\theta_{\delta,c,A}}\left(\xi_{1}\right)}=\left(1-\tilde{\kappa}\right)\lambda^{-\infty}.
\]

Now, the function $\underline{\theta_{\delta,c,A}}$ is increasing
with $\underline{\theta_{\delta,c,A}}\left(0\right)=0$ and 
\begin{align*}
\underline{\theta_{\delta,c,A}}\left(\xi_{1}\right) & =A\textup{e}^{\lambda_{\theta}^{+}\left(\xi_{1}-\xi_{\theta}\right)}-\textup{e}^{\lambda_{\theta}^{-}\left(\xi_{1}-\xi_{\theta}\right)}\\
 & =A^{1-\frac{\lambda_{\theta}^{+}}{\lambda_{\theta}^{+}-\lambda_{\theta}^{-}}}\textup{e}^{\lambda_{\theta}^{+}\xi_{1}}-A^{-\frac{\lambda_{\theta}^{-}}{\lambda_{\theta}^{+}-\lambda_{\theta}^{-}}}\textup{e}^{\lambda_{\theta}^{-}\xi_{1}}\\
 & =A^{\frac{-\lambda_{\theta}^{-}}{\lambda_{\theta}^{+}-\lambda_{\theta}^{-}}}\left(\textup{e}^{\lambda_{\theta}^{+}\xi_{1}}-\textup{e}^{\lambda_{\theta}^{-}\xi_{1}}\right).
\end{align*}
As a function of $A$, this quantity is increasing (recall $\lambda_{\theta}^{-}<0$)
and vanishes as $A\to0$. We fix now $A$ such that
\[
\underline{\theta_{\delta,c,A}}\left(\xi_{1}\right)=\underline{\psi_{\delta,c}}\left(-\zeta_{\kappa}\right)\leq\kappa.
\]
Defining $\zeta_{1}=\xi_{1}+\zeta_{\kappa}>\zeta_{\kappa}$, we obtain
indeed
\[
\underline{\theta_{\delta,c,A}}\left(\xi_{1}\right)=\underline{\psi_{\delta,c}}\left(\xi_{1}-\zeta_{1}\right)\leq\kappa,
\]
\begin{align*}
\underline{\theta_{\delta,c,A}}'\left(\xi_{1}\right) & =\left(1-\tilde{\kappa}\right)\lambda^{-\infty}\underline{\theta_{\delta,c,A}}\left(\xi_{1}\right)\\
 & =\left(1-\tilde{\kappa}\right)\lambda^{-\infty}\underline{\psi_{\delta,c}}\left(\xi_{1}-\zeta_{1}\right)\\
 & <\underline{\psi_{\delta,c}}'\left(\xi_{1}-\zeta_{1}\right),
\end{align*}
as well as the limit
\[
\lim_{\kappa\to0}\left(\zeta_{1}-\xi_{1}\right)=\lim_{\kappa\to0}\zeta_{\kappa}=+\infty.
\]
This completes the proof. 

\end{proof}
\begin{lem}
\label{lem:interface_2} There exists $\delta_0\in\left(0,\frac{1}{2}\right)$ such that, 
for all $\delta\in\left[0,\delta_0\right)$,
$c>c_{\textup{LLW}}^{\delta}$ and $\tilde{c}\geq\max\left(c,f_{\delta}\left(c\right)\right)$,
there exists $\zeta_{2}\in\mathbb{R}$ such that the equation
\[
\overline{\varphi_{\delta,c}}\left(x-ct\right)=\overline{w_{\delta,c,\tilde{c}}}\left(t,x-\zeta_{2}\right)
\]
admits for all $t\geq0$ an isolated solution $x_{2}\left(t\right)\in\mathbb{R}$
such that
\begin{enumerate}
\item $\overline{\varphi_{\delta,c}}\left(x-ct\right)>\overline{w_{\delta,c,\tilde{c}}}\left(t,x-\zeta_{2}\right)$
for all $x\in\left(x_{2}\left(t\right),+\infty\right)$;
\item $\overline{\varphi_{\delta,c}}\left(x-ct\right)<\overline{w_{\delta,c,\tilde{c}}}\left(t,x-\zeta_{2}\right)$
for all $x\in\left(-\infty,x_{2}\left(t\right)\right)$;
\item $\overline{\varphi_{\delta,c,\tilde{c}}}\left(x_{2}\left(t\right)-ct\right)\leq\frac{\delta}{b}$.
\end{enumerate}
Furthermore, 
\begin{enumerate}
\item $x_{2}\in\mathscr{C}^{1}\left(\left[0,+\infty\right),\left(\zeta_{2},+\infty\right)\right)$;
\item $x_{2}\left(t\right)=\tilde{c}t+O\left(1\right)$ as $t\to+\infty$.
\end{enumerate}
\end{lem}

\begin{rem*}
As $\delta\to0$, $f_{\delta}\left(c\right)\to f\left(c\right)$.
It can be verified that $\left(f_{\delta}\left(c\right)\right)_{\delta\in\left[0,\frac{1}{2}\right)}$
is increasing, so that the convergence occurs from above.
\end{rem*}
\begin{proof}
Recall from \lemref{Exponential_estimates_super-critical_waves} that
there exists $\zeta\in\mathbb{R}$ such that,
\[
\overline{\varphi_{\delta,c}}\left(\xi-\zeta\right)\sim\textup{e}^{-\lambda_{\delta}\left(c\right)\xi}\text{ as }\xi\to+\infty.
\]
Hence, by the intermediate value theorem, for each $t\geq0$ and each
$\zeta_{2}\in\mathbb{R}$, the equation 
\begin{align*}
\overline{\varphi_{\delta,c}}\left(x-ct\right) & =\overline{w_{\delta,c,\tilde{c}}}\left(t,x-\zeta_{2}\right)\\
 & =\textup{e}^{-\lambda_{\delta}\left(c\right)\left(\tilde{c}-c\right)t}\textup{e}^{-\Lambda_{\delta}\left(c,\tilde{c}\right)\left(x-\zeta_2-\tilde{c}t\right)}
\end{align*}
admits at least one solution $x\left(t\right)$ provided $\Lambda_{\delta}\left(c,\tilde{c}\right)>\lambda_{\delta}\left(c\right)$.
This inequality is true indeed, since it is equivalent to
\[
\tilde{c}-\sqrt{\tilde{c}^{2}-4\left(\lambda_{\delta}\left(c\right)\left(\tilde{c}-c\right)+1\right)}>2\lambda_{\delta}\left(c\right),
\]
that is to 
\[
\tilde{c}^{2}-4\lambda_{\delta}\left(c\right)\tilde{c}+4\left(\lambda_{\delta}\left(c\right)\right)^{2}>\tilde{c}^{2}-4\left(\lambda_{\delta}\left(c\right)\left(\tilde{c}-c\right)+1\right),
\]
that is to 
\[
\left(\lambda_{\delta}\left(c\right)\right)^{2}-c\lambda_{\delta}\left(c\right)+1>0,
\]
that is (recalling that $\lambda_{\delta}\left(c\right)$ is characterized
by $\left(\lambda_{\delta}\left(c\right)\right)^{2}-c\lambda_{\delta}\left(c\right)+1-a_{\delta}=0$)
to the obviously true following inequality, 
\[
a_{\delta}>0.
\]

Since $\overline{\varphi_{\delta,c}}\left(\xi\right)<1+\delta$ for all $\xi\in\mathbb{R}$,
any such solution satisfies 
\[
-\ln\left(1+\delta\right)<\lambda_{\delta}\left(c\right)\left(\tilde{c}-c\right)t+\Lambda_{\delta}\left(c,\tilde{c}\right)\left(x\left(t\right)-\zeta_{2}-\tilde{c}t\right),
\]
that is 
\[
x\left(t\right)>\zeta_{2}+\left(\tilde{c}-\frac{\lambda_{\delta}\left(c\right)\left(\tilde{c}-c\right)+\ln\left(1+\delta\right)/t}{\Lambda_{\delta}\left(c,\tilde{c}\right)}\right)t.
\]
By
\[
\lim_{\delta'\to0}\left(\tilde{c}-\frac{\lambda_{\delta'}\left(c\right)\left(\tilde{c}-c\right)+\ln\left(1+\delta'\right)/t}{\Lambda_{\delta'}\left(c,\tilde{c}\right)}\right)=\tilde{c}-\frac{\lambda\left(c\right)\left(\tilde{c}-c\right)}{\Lambda\left(c,\tilde{c}\right)}
\]
uniformly for $t \geq 1$, and, due to the preceding observation, 
\[
\frac{\lambda\left(c\right)\left(\tilde{c}-c\right)}{\Lambda\left(c,\tilde{c}\right)}<\tilde{c}-c,
\]
we deduce that $x(t)>\zeta_2+ct$ provided $\delta$ is small enough. 
Therefore the set of solutions is bounded from below and admit an infimum
$I\left(t\right)>\zeta_{2}+ct$. Back to the exponential estimates, it
is also clear that the set of solutions is bounded from above and
admits therefore a supremum $S\left(t\right)$. 

Recall that the asymptotic estimate for $\overline{\varphi_{\delta,c}}$
can be differentiated. Setting $g:\left(t,x\right)\mapsto\overline{\varphi_{\delta,c}}\left(x-ct\right)-\overline{w_{\delta,c,\tilde{c}}}\left(t,x-\zeta_{2}\right)$,
we find that for any $t\geq0$ and any solution $x\left(t\right)\in\left[I\left(t\right),S\left(t\right)\right]$,
\[
\partial_{x}g\left(t,x\left(t\right)\right)=\overline{\varphi_{\delta,c}}\left(x\left(t\right)-ct\right)\left(\left(\frac{\overline{\varphi_{\delta,c}}'}{\overline{\varphi_{\delta,c}}}\right)\left(x\left(t\right)-ct\right)+\Lambda_{\delta}\left(c,\tilde{c}\right)\right).
\]
Since 
\[
\lim_{\xi\to+\infty}\left(\frac{\overline{\varphi_{\delta,c}}'}{\overline{\varphi_{\delta,c}}}\right)\left(\xi\right)+\Lambda_{\delta}\left(c,\tilde{c}\right)=-\lambda_{\delta}\left(c\right)+\Lambda_{\delta}\left(c,\tilde{c}\right)<0,
\]
we can choose $\zeta_{2}$ large enough so that $\left(\frac{\overline{\varphi_{\delta,c}}'}{\overline{\varphi_{\delta,c}}}\right)\left(\xi\right)<0$
for all $\xi\geq\zeta_{2}$. Since $x\left(t\right)-ct>\zeta_{2}$ for all $t\geq0$, we deduce
\[
\left(\frac{\overline{\varphi_{\delta,c}}'}{\overline{\varphi_{\delta,c}}}\right)\left(x\left(t\right)-ct\right)+\Lambda_{\delta}\left(c,\tilde{c}\right)<0,
\]
whence $g$ is decreasing with respect to $x$ in a neighborhood of
$x\left(t\right)$. This implies directly
the uniqueness of $x\left(t\right)$, namely $I\left(t\right)=S\left(t\right)$.
From now on, we denote this unique solution $x_{2}\left(t\right)$.
Of course, the regularity of $x_{2}$ follows directly from the implicit
function theorem. The above yields that $x_{2}(t)-ct\geq\zeta_{2}$ for all $t\geq0$.

Provided $\zeta_{2}$ is large enough, for all $\xi\geq\zeta+\zeta_{2}$,
\[
\left(1-\delta\right)\textup{e}^{-\lambda_{\delta}\left(c\right)\xi}\leq\overline{\varphi_{\delta,c}}\left(\xi-\zeta\right)\leq\left(1+\delta\right)\textup{e}^{-\lambda_{\delta}\left(c\right)\xi}.
\]
 At $\xi=x_{2}\left(t\right)-ct+\zeta\geq\zeta+\zeta_{2}$, this reads
\[
\left(1-\delta\right)\textup{e}^{-\lambda_{\delta}\left(c\right)\left(x_{2}\left(t\right)-ct+\zeta\right)}\leq\overline{w_{\delta,c,\tilde{c}}}\left(t,x_{2}\left(t\right)-\zeta_{2}\right)\leq\left(1+\delta\right)\textup{e}^{-\lambda_{\delta}\left(c\right)\left(x_{2}\left(t\right)-ct+\zeta\right)},
\]
 that is 
\begin{align*}
\ln\left(1-\delta\right)-\lambda_{\delta}\left(c\right)\left(x_{2}\left(t\right)-ct+\zeta\right) &\leq-\lambda_{\delta}\left(c\right)\left(\tilde{c}-c\right)t-\Lambda_{\delta}\left(c,\tilde{c}\right)\left(x_{2}\left(t\right)-\zeta_{2}-\tilde{c}t\right)\\
&\qquad \qquad \leq\ln\left(1+\delta\right)-\lambda_{\delta}\left(c\right)\left(x_{2}\left(t\right)-ct+\zeta\right).
\end{align*}

The first inequality yields 
\[
x_{2}\left(t\right)\leq\frac{\tilde{c}\left(\Lambda_{\delta}\left(c,\tilde{c}\right)-\lambda_{\delta}\left(c\right)\right)t-\ln\left(1-\delta\right)+\lambda_{\delta}\left(c\right)\zeta+\Lambda_{\delta}\left(c,\tilde{c}\right)\zeta_{2}}{\Lambda_{\delta}\left(c,\tilde{c}\right)-\lambda_{\delta}\left(c\right)}
\]
and the second inequality yields 
\[
x_{2}\left(t\right)\geq\frac{\tilde{c}\left(\Lambda_{\delta}\left(c,\tilde{c}\right)-\lambda_{\delta}\left(c\right)\right)t-\ln\left(1+\delta\right)+\lambda_{\delta}\left(c\right)\zeta+\Lambda_{\delta}\left(c,\tilde{c}\right)\zeta_{2}}{\Lambda_{\delta}\left(c,\tilde{c}\right)-\lambda_{\delta}\left(c\right)}.
\]
Together these two estimates give that the asymptotic speed of $x_{2}$
is exactly $\tilde{c}$. 

Finally, using once again $x_2\left(t\right)-ct\geq\zeta_2$, we find 
\[
\overline{\varphi_{\delta,c}}\left(x_{2}\left(t\right)-ct\right)\leq\left(1+\delta\right)\textup{e}^{-\lambda_{\delta}\left(c\right)\left(x_{2}\left(t\right)-ct+\zeta\right)}\leq\left(1+\delta\right)\textup{e}^{-\lambda_{\delta}\left(c\right)\left(\zeta+\zeta_{2}\right)},
\]
and the inequality 
\[
\overline{\varphi_{\delta,c}}\left(x_{2}\left(t\right)-ct\right)\leq\frac{\delta}{b}\text{ for all }t\geq0
\]
 is indeed satisfied provided $\zeta_{2}$ is large enough.
\end{proof}

Thanks again to the intermediate value theorem and the implicit function
theorem, we can similarly establish the following lemmas. Since they involve the quantities $L$, $x_L$ and $h^{\star}$,
we recall that these are defined in \lemref{estimates_alpha} and \lemref{pi_subsolution_and_estimate} respectively.

\begin{lem}
\label{lem:interface_3_small_c_1} There exists $\delta_{1}\in\left(0,\frac{1}{2}\right)$ such that, 
for all $\delta\in\left(0,\delta_{1}\right)$, $c\in\left[c_{\textup{LLW}}^{\delta},2\right)$ and $\zeta_{3}\in\mathbb{R}$, the equation
\[
\underline{\psi_{\delta,c}}\left(x-ct\right)=\underline{\omega_{\delta,R_{\delta}}}\left(x-\left(2\sqrt{r\left(1-2\delta\right)d}-\delta\right)t-\zeta_{3}\right)
\]
 admits for all $t\geq0$ an isolated solution $x_{3}\left(t\right)\in\mathbb{R}$
such that
\begin{enumerate}
\item $\underline{\psi_{\delta,c}}\left(x-ct\right)>\underline{\omega_{\delta,R_{\delta}}}\left(x-\left(2\sqrt{r\left(1-2\delta\right)d}-\delta\right)t-\zeta_{3}\right)$ for $x$ 
in a left-sided neighborhood of $x_{3}\left(t\right)$;
\item $\underline{\psi_{\delta,c}}\left(x-ct\right)<\underline{\omega_{\delta,R_{\delta}}}\left(x-\left(2\sqrt{r\left(1-2\delta\right)d}-\delta\right)t-\zeta_{3}\right)$ for $x$ 
in a right-sided neighborhood of $x_{3}\left(t\right)$;
\item for all $t\geq0$,
\[
-R_{\delta}<x_{3}\left(t\right)-\left(2\sqrt{r\left(1-2\delta\right)d}-\delta\right)t-\zeta_{3}<x_{\delta,R_{\delta}}.
\]
\end{enumerate}
Furthermore, $x_{3}\in\mathscr{C}^{1}\left(\left[0,+\infty\right),\mathbb{R}\right)$.
\end{lem}

\begin{lem}
\label{lem:interface_3_large_c_1} For all $\delta\in\left(0,\frac{1}{2}\right)$,
$c\geq c_{\textup{LLW}}^{\delta}$, $\tilde{c}>2\sqrt{rd}$ such that $\tilde{c}\geq c$ and $h\in\left(0,h^{\star}\right)$,
there exists $\zeta_{3}^{0}\in\mathbb{R}$ such that, for all $\zeta_{3}\geq\zeta_{3}^{0}$,
the equation
\[
\underline{\psi_{\delta,c}}\left(x-ct\right)=\underline{\pi_{\delta,\tilde{c},h}}\left(x-\tilde{c}t-\zeta_{3}\right)
\]
 admits for all $t\geq0$ an isolated solution $x_{3}\left(t\right)\in\mathbb{R}$
such that
\begin{enumerate}
\item $\underline{\psi_{\delta,c}}\left(x-ct\right)>\underline{\pi_{\delta,\tilde{c},h}}\left(x-\tilde{c}t-\zeta_{3}\right)$ for $x$ 
in a left-sided neighborhood of $x_{3}\left(t\right)$;
\item $\underline{\psi_{\delta,c}}\left(x-ct\right)<\underline{\pi_{\delta,\tilde{c},h}}\left(x-\tilde{c}t-\zeta_{3}\right)$ for $x$ 
in a right-sided neighborhood of $x_{3}\left(t\right)$;
\item for all $t\geq0$,
\[
-\sqrt{\frac{c}{rh}} < x_{3}\left(t\right)-\tilde{c}t-\zeta_{3}<0.
\]
\end{enumerate}
Furthermore, $x_{3}\in\mathscr{C}^{1}\left(\left[0,+\infty\right),\mathbb{R}\right)$.
\end{lem}

\begin{rem*}
We have to point out here that the preceding two lemmas defining $x_3$ will never be used concurrently and 
no conflict of notation will occur. \lemref{interface_3_small_c_1} will be used only in the proof of 
\propref{Super-solution_compactly_supported} whereas \lemref{interface_3_large_c_1} will be used only in the proof of 
\propref{Super-sub-solution_terraces}. In other words, going back to \thmref{Compactly_supported} and 
\thmref{Generalized_terraces}, they corresponds to different values of $c_1$: 
\lemref{interface_3_small_c_1} corresponds to $c_1=2\sqrt{rd}$ whereas \lemref{interface_3_large_c_1} corresponds to $c_1>2\sqrt{rd}$. 
\end{rem*}

\begin{lem}
\label{lem:interface_4} For all $\delta\in\left[0,\frac{1}{2}\right)$,
$c>2\sqrt{rd}$, $\eta\in\left(0,\frac{1}{d}\sqrt{c^{2}-4rd}\right)$,
$B>0$, there exists $\zeta_{4}\in\mathbb{R}$ such that the equation
\[
\underline{\pi_{\delta,c,h}}\left(\xi\right)=\underline{\beta_{c,B,\eta}}\left(\xi-\zeta_{4}\right)
\]
admits an isolated solution $\xi_{4}\in\mathbb{R}$ such that
\begin{enumerate}
\item $\underline{\pi_{\delta,c,h}}\left(\xi+\zeta_{4}\right)>\underline{\beta_{c,B,\eta}}\left(\xi\right)$ for $\xi$ 
in a left-sided neighborhood of $\xi_{4}$;
\item $\underline{\pi_{\delta,c,h}}\left(\xi+\zeta_{4}\right)<\underline{\beta_{c,B,\eta}}\left(\xi\right)$ for $\xi$ 
in a right-sided neighborhood of $\xi_{4}$;
\item $\xi_{4}>0$.
\end{enumerate}
\end{lem}

\begin{lem}
\label{lem:interface_0_nonexistence} For all $c>2\sqrt{1-a}$, there
exists $\zeta_{0}\in\mathbb{R}$ such that, for all $\kappa\in\left(0,\frac{1-a}{2}\right)$,
the equation
\[
\underline{\alpha_{L}}\left(x\right)=\underline{\chi_{c}}\left(x-ct-\zeta_{0}\right)
\]
admits for all $t\geq0$ a minimal solution $x_{0,\kappa}\left(t\right)\in\mathbb{R}$
such that 
\begin{enumerate}
\item $\underline{\alpha_{L}}\left(x\right)>\underline{\chi_{c}}\left(x-ct-\zeta_{0}\right)$
for $x$ in a left-sided neighborhood of $x_{0,\kappa}\left(t\right)$;
\item $\underline{\chi_{c}}\left(x_{0,\kappa}\left(0\right)-\zeta_{0}\right)=\kappa$;
\item $x_{L}<x_{0,\kappa}\left(t\right)<L$.
\end{enumerate}
Furthermore, $x_{0,\kappa}\in\mathscr{C}^{1}\left(\left[0,+\infty\right),\left(x_{L},L\right)\right)$.
\end{lem}

Notice that in the above lemma, $x_{0,\kappa}\left(0\right)=\left(\underline{\chi_{c}}\right)^{-1}\left(\kappa\right)+\zeta_{0}$.

\begin{lem}
\label{lem:interface_1_nonexistence} For all $c>2\sqrt{1-a}$, $\tilde{c}\geq c$
such that $\tilde{c}\in\left(f\left(c\right)-4\sqrt{a},f\left(c\right)\right)$,
$\delta\in\left[0,\frac{1}{4}\left(-\tilde{c}^{2}+4\left(\lambda\left(c\right)\left(\tilde{c}-c\right)+1\right)\right)\right)$
and $\zeta>\left(\underline{\chi_{c}}\right)^{-1}\left(\frac{\delta}{2}\right)$,
the equation
\[
\underline{\chi_{c}}\left(x-ct\right)=\frac{\underline{\chi_{c}}\left(\zeta\right)}{\underline{z_{c,\tilde{c},\delta}}\left(0,X_{z}\right)}\underline{z_{c,\tilde{c},\delta}}\left(t,x-\zeta\right)
\]
admits for all $t\geq0$ an isolated solution $x_{1}\left(t\right)\in\mathbb{R}$
such that 
\begin{enumerate}
\item $\underline{\chi_{c}}\left(x-ct\right)>\frac{\underline{\chi_{c}}\left(\zeta\right)}{\underline{z_{c,\tilde{c},\delta}}\left(0,X_{z}\right)}\underline{z_{c,\tilde{c},\delta}}\left(t,x-\zeta\right)$ 
for $x$ 
in a left-sided neighborhood of $x_{1}\left(t\right)$;
\item $\underline{\chi_{c}}\left(x-ct\right)<\frac{\underline{\chi_{c}}\left(\zeta\right)}{\underline{z_{c,\tilde{c},\delta}}\left(0,X_{z}\right)}\underline{z_{c,\tilde{c},\delta}}\left(t,x-\zeta\right)$ 
for $x$ 
in a right-sided neighborhood of $x_{1}\left(t\right)$;
\item $\tilde{c}t+\zeta<x_{1}\left(t\right)<X_{z}+\tilde{c}t+\zeta$.
\end{enumerate}
Furthermore, $x_{1}\in\mathscr{C}^{1}\left(\left[0,+\infty\right),\left(\zeta,+\infty\right)\right)$.
\end{lem}

\begin{rem*}
Similarly to the third interface $x_3$ which is defined in two separate lemmas, the zeroth interface is defined concurrently by
\lemref{interface_0} and \lemref{interface_0_nonexistence} and the first interface is defined concurrently by \lemref{interface_1}
and \lemref{interface_1_nonexistence}. \lemref{interface_0} will be used only in the proof of \propref{Super-sub-solution_terraces}, 
\lemref{interface_1} will be used only in the proof of \propref{Super-sub-solution_terraces} and in that of 
\propref{Super-solution_compactly_supported}, \lemref{interface_0_nonexistence} and \lemref{interface_1_nonexistence} will be 
used only in the proof of \propref{Sub-solution_nonexistence}.
\end{rem*}

There exists a small $\delta^{\star}\in\left(0,\frac{1}{2}\right)$ such that all the lemmas of this subsection involving a parameter $\delta$ can be applied in the range $\delta\in\left(0,\delta^{\star}\right)$. By construction, all the objects depending on $\delta$ defined in the preceding subsection are also well-defined in this range.

\subsection{Construction of the super-solutions and sub-solutions for \thmref{Generalized_terraces}}
In this subsection, we prove \propref{Super-sub-solution_terraces}.

Let $c_{1}>2\sqrt{rd}$ and $c_{2}>c_{\textup{LLW}}$ such that $c_{1}>c_{2}$
and $c_{1}>f\left(c_{2}\right)$. In order to construct a satisfying approximated speed $c_2^\delta\simeq c_2$, 
we need to find $c_2^\delta$ such that:
\begin{enumerate}
	\item $c_2^\delta>c_{\textup{LLW}}^\delta$;
	\item $c_2^\delta \to c_2$ as $\delta\to 0$;
    \item $c_2<c_2^\delta<c_1$;
    \item $f_\delta\left(c_2^\delta\right)<c_1$
    \item $\Lambda_\delta\left(c_2^\delta,c_1\right)$ is well-defined;
    \item $\Lambda_\delta\left(c_2^\delta,c_1\right)\leq\Lambda\left(c_2,c_1\right)$.
\end{enumerate}
The condition $(6)$ above is equivalent to 
$\lambda_\delta\left(c_2^\delta\right)\left(c_1-c_2^\delta\right)\leq\lambda\left(c_2\right)\left(c_1-c_2\right)$,
that is to 
\[
\frac{\lambda_\delta\left(c_2^\delta\right)}{\lambda\left(c_2\right)}\leq\frac{\left(c_1-c_2\right)}{\left(c_1-c_2^\delta\right)},
\]
with a right-hand side necessarily larger than $1$ provided $(3)$ above is satisfied. Since the function 
$\left(\delta,c\right)\mapsto\lambda_\delta\left(c\right)$ is increasing with respect to $c$ and decreasing with respect to $\delta$,
the sign of $\lambda_\delta\left(c_2^\delta\right)-\lambda\left(c_2\right)$ is unclear if we only assume $c_2^\delta>c_2$. 
Hence some care is needed and we cannot simply take a rough approximation like $c_2^\delta=c_2+\delta$. 

In fact, since $a_\delta<a$ and $\lambda(c_2) < \sqrt{1-a}$, we can choose $\delta\in\left(0,\delta^{\star}\right)$ such that
\[
\lambda\left(c_{2}\right)<\sqrt{1-a_{\delta}}.
\]
Consequently, the following quantity is well-defined: 
\[
c_{2}^{\delta}=\left(\lambda_{\delta}^{-1}\circ\lambda\right)\left(c_{2}\right).
\]
 Since $\lambda$ and $\lambda_\delta$ are both decreasing functions and $\lambda\left(c_{2}\right)<\lambda_{\delta}\left(c_{2}\right)$, it follows that 
 $c_{2}^{\delta}>c_{2}$, whence 
\begin{align*}
4\left(\lambda_{\delta}\left(c_{2}^{\delta}\right)\left(c_{1}-c_{2}^{\delta}\right)+1\right) & =4\left(\lambda\left(c_{2}\right)\left(c_{1}-c_{2}^{\delta}\right)+1\right)\\
 & <4\left(\lambda\left(c_{2}\right)\left(c_{1}-c_{2}\right)+1\right)\\
 & <c_{1}^2,
\end{align*}
where the last inequality is due to $c_1-f\left(c_2\right)>0$ (see also Subsection \ref{subsec:4.1.12}). By continuity, we can further assume that $\delta$ is so small that
\[
\left\{ \begin{matrix}c_{\textup{LLW}}\leq c_{\textup{LLW}}^{\delta}<c_{2}^{\delta}\\
c_{2}<c_{2}^{\delta}<c_{1}\\
-4\sqrt{a_{\delta}}<c_{1}-f_{\delta}\left(c_{2}^{\delta}\right)
\end{matrix}\right..
\]
 It follows then, from Subsection \ref{subsec:4.1.12}, that 
\[
f_{\delta}\left(c_{2}^{\delta}\right)<f\left(c_{2}\right)<c_{1},
\]
whence the quantity $\Lambda_{\delta}\left(c_{2}^{\delta},c_{1}\right)$
is well-defined. By definition, it satisfies 
\begin{align*}
\Lambda_{\delta}\left(c_{2}^{\delta},c_{1}\right) & =\frac{1}{2}\left(c_{1}-\sqrt{c_{1}^{2}-4\left(\lambda_{\delta}\left(c_{2}^{\delta}\right)\left(c_{1}-c_{2}^{\delta}\right)+1\right)}\right)\\
 & =\frac{1}{2}\left(c_{1}-\sqrt{c_{1}^{2}-4\left(\lambda\left(c_{2}\right)\left(c_{1}-c_{2}^{\delta}\right)+1\right)}\right)\\
 & <\frac{1}{2}\left(c_{1}-\sqrt{c_{1}^{2}-4\left(\lambda\left(c_{2}\right)\left(c_{1}-c_{2}\right)+1\right)}\right),
\end{align*}
so that $\Lambda_{\delta}\left(c_{2}^{\delta},c_{1}\right)<\Lambda\left(c_{2},c_{1}\right)$. 

\subsubsection{Super-solution}\label{subsec:4.3.1}

The pair $\left(\overline{u_{\delta}},\underline{v_{\delta}}\right)$
is defined by (see Figure \ref{fig:1.6b})
\[
\overline{u_{\delta}}\left(t,x\right)=\left\{ \begin{matrix}\min\left(1,\overline{\varphi_{\delta,c_{2}^{\delta}}}\left(x-c_{2}^{\delta}t-\zeta_{1,\kappa}\right)\right) & \text{if }x<x_{2}\left(t\right)+\zeta_{1,\kappa}\\
\overline{w_{\delta,c_{2}^{\delta},c_{1}}}\left(t,x-\zeta_{1,\kappa}-\zeta_{2}\right) & \text{if }x\geq x_{2}\left(t\right)+\zeta_{1,\kappa}
\end{matrix}\right.,
\]
\[
\underline{v}_{\delta}\left(t,x\right)=\left\{ \begin{matrix}\max\left(0,\underline{\theta_{\delta,c_{2}^{\delta},A_{\kappa}}}\left(x-c_{2}^{\delta}t\right)\right) & \text{if }x<\xi_{1,\kappa}+c_{2}^{\delta}t\\
\underline{\psi_{\delta,c_{2}^{\delta}}}\left(x-c_{2}^{\delta}t-\zeta_{1,\kappa}\right) & \text{if }x\in\left[\xi_{1,\kappa}+c_{2}^{\delta}t,x_{3}\left(t\right)+\zeta_{1,\kappa}\right)\\
\underline{\pi_{\delta,c_{1},h}}\left(x-c_{1}t-\zeta_{1,\kappa}-\zeta_{3}\right) & \text{if }x\in\left[x_{3}\left(t\right)+\zeta_{1,\kappa},\xi_{4}+c_{1}t+\zeta_{1,\kappa}+\zeta_{3}\right)\\
\underline{\beta_{c_{1},B,\eta_{\beta}}}\left(x-c_{1}t-\zeta_{1,\kappa}-\zeta_{3}-\zeta_{4}\right) & \text{if }x\geq\xi_{4}+c_{1}t+\zeta_{1,\kappa}+\zeta_{3}
\end{matrix}\right.,
\]
where
\begin{itemize}
\item $\kappa\in\left(0,\delta\right]$ is fixed so small that $\zeta_{1,\kappa}-\xi_{1,\kappa} +x_2\left(0\right)$ is large enough so that for all $t\geq0$, $\xi_{1,\kappa}+c_{2}^{\delta}t<x_{2}\left(t\right)+\zeta_{1,\kappa}$
(see \lemref{interface_1}(4) and \lemref{interface_2} and use $x_2(t) \geq c_1 t + O(1) \geq c^\delta_2t + O(1)$);
\item $\zeta_{3}$ is fixed so large that, for all $t\geq0$, $x_{2}\left(t\right)<x_{3}\left(t\right)$
(by \lemref{interface_2} \lemref{interface_3_large_c_1}, $x_{2}(t)-c_{1}t$
and $x_{3}(t)-c_{1}t$ are both bounded uniformly in $t\geq0$, whence
we can translate $x_{3}(t)$ to the right by increasing $\zeta_{3}$);
\item $h=\frac{h^{\star}}{2}$;
\item $\eta_{\beta}=\frac{1}{2}\min\left(\frac{1}{d}\sqrt{c_{1}^{2}-4rd},\Lambda_{\delta}\left(c_{2}^{\delta},c_{1}\right)\right)$;
\item $B=\textup{e}^{\Lambda_{\delta}\left(c_{2}^{\delta},c_{1}\right)\xi_{\beta}}2\overline{u_{\delta}}\left(0,\zeta_{1,\kappa}+\zeta_{3}+\zeta_{4}\right)$. 
\end{itemize}
The inequality 
\[
x_{3}\left(t\right)+\zeta_{1,\kappa}<\xi_{4}+c_{1}t+\zeta_{1,\kappa}+\zeta_{3}
\]
 is guaranteed by \lemref{interface_3_large_c_1} and \lemref{interface_4}
which respectively show that $x_{3}\left(t\right)<c_{1}t+\zeta_{3}$
and $\xi_{4}>0$. In conclusion, we have 
\[
\xi_{1,\kappa}+c_{2}^{\delta}t<x_{2}\left(t\right)+\zeta_{1,\kappa}<x_{3}\left(t\right)+\zeta_{1,\kappa}<\xi_{4}+c_{1}t+\zeta_{1,\kappa}+\zeta_{3}\quad\text{ for all }t\geq0,
\]
i.e. $\underline{v}_\delta$ is well-defined for all $t \geq 0$.

\subsubsection{Sub-solution}\label{subsec:4.3.2}  First define the pair $\left(\underline{u},\overline{v}\right)$  by (see Figure \ref{fig:1.6a}) 
\[
\underline{u}\left(t,x\right)=\left\{ \begin{matrix}\underline{\chi_{c_{2}}}\left(x-c_{2}t+\zeta_{0}\right) & \text{if }x<x_{0}\left(t\right)\\
\underline{w_{c_{2},c_{1},A,\eta_{w}}}\left(t,x\right) & \text{if }x\geq x_{0}\left(t\right)
\end{matrix}\right.,
\]
\[
\overline{v}\left(t,x\right)=\min\left(1,\textup{e}^{-\lambda_{v}\left(c_{1}\right)\left(x-c_{1}t\right)}\right),
\]
where $\eta_{w}=\frac{1}{2}\min\left(\sqrt{c_{1}^{2}-4\left(\lambda\left(c_{2}\right)\left(c_{1}-c_{2}\right)+1\right)},\lambda_{v}\left(c_{1}\right)\right)$. 

The function $\underline{u}$ depends on a constant $A>0$ which will
be fixed later on. 

\subsubsection{Up to some translations, the sub-solution $(\underline{u},\overline{v})$
is initially smaller than the super-solution $(\overline{u_{\delta}},\underline{v_{\delta}})$}

First, let $\underline{V}:\mathbb{R}\to\left[0,1\right]$ be the smallest
nonincreasing continuous function such that 
\[
\underline{v_{\delta}}\left(0,x\right)\leq\underline{V}\left(x\right)\text{ for all }x\in\mathbb{R}
\]
and let $\zeta_{5}\in\mathbb{R}$ such that, for all $t\geq0$, $x\mapsto\underline{v_{\delta}}\left(t,x+c_{1}t\right)$
is $\mathscr{C}^{1}$ and nonincreasing in $\left(\zeta_{5},+\infty\right)$.
The existence of $\zeta_{5}$ follows from the fact that the last
discontinuity of $\partial_{x}\underline{v_{\delta}}$ and the last
local maximum of $\underline{v_{\delta}}$ move both at most at speed
$c_{1}$. The limit of $\underline{V}$ at $-\infty$ is smaller than
$1$ and $\underline{V}\left(x\right)=\underline{v_{\delta}}\left(0,x\right)$
if $x>\zeta_{5}$ . Therefore, since $\overline{v}$ and $\underline{v_{\delta}}$
have the same exponential decay at $+\infty$, there exists $\zeta_{6}\geq\zeta_{5}$
such that:
\begin{enumerate}
\item for all $t\geq0$, $x\mapsto\underline{v_{\delta}}\left(t,x+c_{1}t\right)$
is $\mathscr{C}^{1}$ and nonincreasing in $\left(\zeta_{6},+\infty\right)$; 
\item for all $x\in\mathbb{R}$, $\underline{v_{\delta}}\left(0,x\right)\leq\underline{V}\left(x\right)\leq\overline{v}\left(0,x-\zeta_{6}\right)$.
\end{enumerate}
Notice that with this definition of $\zeta_{5}$ and $\zeta_{6}$,
the irregularity of $\overline{v}$ is initially on the right of the
last irregularity of $\underline{v_{\delta}}$. Since the distance
between these two points is nondecreasing with respect to $t$, it
is bounded from below by the initial distance. 

Next, quite similarly, we define $\zeta_{7}\in\mathbb{R}$ such that
\[
\underline{u}\left(0,x+\zeta_{7}\right)\leq\overline{u_{\delta}}\left(0,x\right)\text{ for all }x\in\mathbb{R}.
\]
The irregularity of $\underline{u}$ moves faster than the first irregularity
of $\overline{u_{\delta}}$ (as $c_{1}>c_{2}^{\delta}$), whence it is
impossible to guarantee that they stay ordered. This is not a major
issue but some additional care will be required later on. Still, without
loss of generality, we assume that $\zeta_{7}$ is so large that the
irregularity of $\underline{u}$ and the second (last) irregularity
of $\overline{u_{\delta}}$, which both move at speed $c_{1}$, stay
ordered. 

\subsubsection{Cleansing}

Now that all required translations are done, we fix 
\[
A=2\textup{e}^{\lambda_{v}\left(c_{1}\right)x_{w}}\textup{e}^{\lambda_v\left(c_1\right)\left(\zeta_6+\zeta_7\right)},
\]
and thus there remains only one parameter: $\delta$.

From now on, all the subscripts referring to fixed parameters are
omitted. Furthermore, since all the properties of the functions $\underline{\chi}$,
$\underline{w}$, $\overline{w_{\delta}}$, $\left(\overline{\varphi_{\delta}},\underline{\psi_{\delta}}\right)$,
$\underline{\theta_{\delta}}$, $\underline{\omega_{\delta}}$, $\underline{\pi_{\delta}}$,
$\underline{\beta}$ we are interested in are invariant by translation,
we assume that these functions were correctly normalized from the
beginning, so that $\zeta_{1,\kappa}=\zeta_{3}=\zeta_{4}=\zeta_{7}=0$,
and we fix $x_{0}\left(0\right)=0$. Similarly, we define $C_{\delta}=\textup{e}^{\lambda_{v}\left(c_{1}\right)\zeta_{6}}>0$
so that $\overline{v_\delta}\left(t,x\right)=\min\left(1,C_{\delta}\textup{e}^{-\lambda_{v}\left(c_{1}\right)\left(x-c_{1}t\right)}\right)$
and $x_{w}$ and $\xi_{\beta}$ are redefined so that \lemref{w_sub-solution}
and \lemref{beta_subsolution} stay true as stated. Notice that $\underline{\chi}$, $\underline{w}$, $\underline{\beta}$,
$\underline{u}$ and $\overline{v}$ now depend on $\delta$ because of these
various normalizations (and consequently these notations come with a subscript $\delta$ from now on).

To summarize, the super- and sub-solutions are now defined as follows:
\[
\underline{u_{\delta}}\left(t,x\right)=\left\{ \begin{matrix}\underline{\chi_{\delta}}\left(x-c_{2}t\right) & \text{if }x<x_{0}\left(t\right)\\
\underline{w_{\delta}}\left(t,x\right) & \text{if }x\geq x_{0}\left(t\right)
\end{matrix}\right.,
\]
\[
\overline{v_{\delta}}\left(t,x\right)=\min\left(1,C_{\delta}\textup{e}^{-\lambda_{v}\left(c_{1}\right)\left(x-c_{1}t\right)}\right),
\]
\[
\overline{u_{\delta}}\left(t,x\right)=\left\{ \begin{matrix}\min\left(1,\overline{\varphi_{\delta}}\left(x-c_{2}^{\delta}t\right)\right) & \text{if }x<x_{2}\left(t\right)\\
\overline{w_{\delta}}\left(t,x\right) & \text{if }x\geq x_{2}\left(t\right)
\end{matrix}\right.,
\]
\[
\underline{v_{\delta}}\left(t,x\right)=\left\{ \begin{matrix}\max\left(0,\underline{\theta_{\delta}}\left(x-c_{2}^{\delta}t\right)\right) & \text{if }x<\xi_{1}+c_{2}^{\delta}t\\
\underline{\psi_{\delta}}\left(x-c_{2}^{\delta}t\right) & \text{if }x\in\left[\xi_{1}+c_{2}^{\delta}t,x_{3}\left(t\right)\right)\\
\underline{\pi_{\delta}}\left(x-c_{1}t\right) & \text{if }x\in\left[x_{3}\left(t\right),\xi_{4}+c_{1}t\right)\\
\underline{\beta_{\delta}}\left(x-c_{1}t\right) & \text{if }x\geq\xi_{4}+c_{1}t
\end{matrix}\right..
\]

Furthermore, the interfaces satisfy, for all $t\geq0$, 
\[
\left\{ \begin{matrix}x_{0}\left(t\right)<x_{2}\left(t\right)\\
\xi_{1}+c_{2}^{\delta}t<x_{2}\left(t\right)\\
x_{2}\left(t\right)<x_{3}\left(t\right)\\
\xi_{4}+c_{1}t<\frac{\ln C_{\delta}}{\lambda_{v}\left(c_{1}\right)}+c_{1}t
\end{matrix}\right..
\]

\begin{figure}
\resizebox{\hsize}{!}{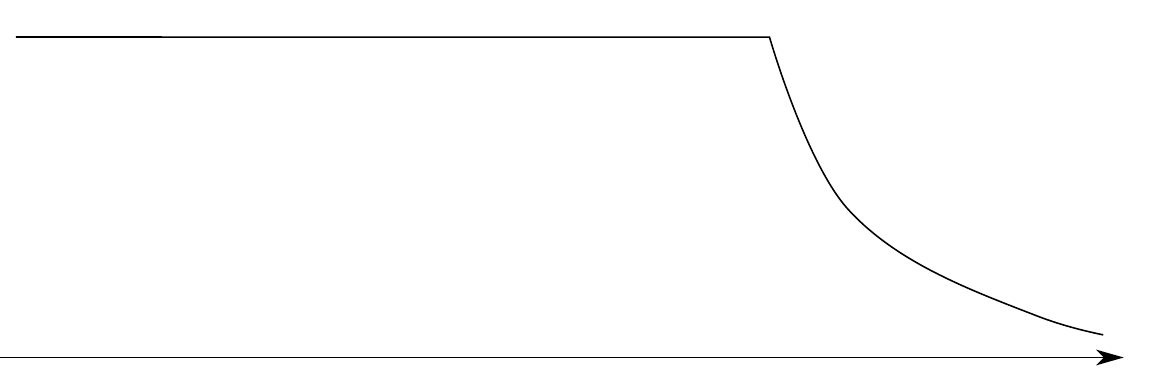}

\caption{Sub-solution $\left(\underline{u_{\delta}},\overline{v_{\delta}}\right)$
for \thmref{Generalized_terraces}}
\label{fig:1.6a}
\end{figure}
\begin{figure}
\resizebox{\hsize}{!}{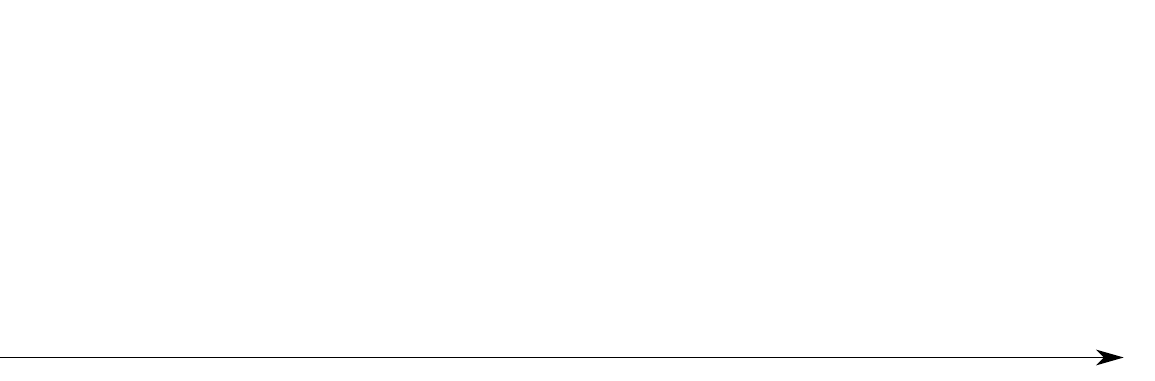}

\caption{Super-solution $\left(\overline{u_{\delta}},\underline{v_{\delta}}\right)$
for \thmref{Generalized_terraces}}
\label{fig:1.6b}
\end{figure}

\subsubsection{Verification of the differential inequalities\label{subsec:Verif_diff_inequalities}}

Let us point out that by \thmref{Generalized_super-solutions} and
\thmref{Generalized_sub-solutions} and by construction of the pairs
$\left(\overline{u_{\delta}},\underline{v_{\delta}}\right)$ and $\left(\underline{u_{\delta}},\overline{v_{\delta}}\right)$,
it suffices to verify the differential inequalities 
\begin{equation}
P\left(\overline{u_{\delta}},\underline{v_{\delta}}\right)\succeq F\left(\overline{u_{\delta}},\underline{v_{\delta}}\right)\label{eq:u_delta_v_delta_super-solution}
\end{equation}
 and 
\begin{equation}
P\left(\underline{u_{\delta}},\overline{v_{\delta}}\right)\preceq F\left(\underline{u_{\delta}},\overline{v_{\delta}}\right)\label{eq:u_delta_v_delta_sub-solution}
\end{equation}
 where the functions are regular in order to establish that $\left(\overline{u_{\delta}},\underline{v_{\delta}}\right)$
and $\left(\underline{u_{\delta}},\overline{v_{\delta}}\right)$ are
indeed a super-solution and a sub-solution of (\ref{eq:competition_diffusion_system})
 respectively. Also, the differential inequalities can also be verified before the translations are performed.
 
In what follows, for the sake of brevity, we voluntarily omit the
mentions of the points $\left(t,x\right)$, $x-c_{1}t$, $x-c_{2}t$
or $x-c_{2}^{\delta}t$ where the various functions are evaluated.
In view of the construction, it should be unambiguous.

First, we consider (\ref{eq:u_delta_v_delta_super-solution}). By
\lemref{theta_sub-solution} and \lemref{interface_1}, for all $\left(t,x\right)$ 
such that 
\[
\left(\overline{u_{\delta}},\underline{v_{\delta}}\right)\left(t,x\right)=\left(1,\underline{\theta_{\delta}}\left(x-c_{2}^{\delta}t\right)\right),
\]
we find $\underline{\theta_{\delta}}\leq\kappa\leq\delta$ and
\begin{align*}
P\left(\overline{u_{\delta}},\underline{v_{\delta}}\right)-F\left(\overline{u_{\delta}},\underline{v_{\delta}}\right) & =\left(\begin{matrix}a\underline{\theta_{\delta}}\\
-c_{2}^{\delta}\underline{\theta_{\delta}}'-d\underline{\theta_{\delta}}''-r\underline{\theta_{\delta}}\left(1-\underline{\theta_{\delta}}-b\right)
\end{matrix}\right)\\
 & =\left(\begin{matrix}a\underline{\theta_{\delta}}\\
-r\underline{\theta_{\delta}}\left(\delta-\underline{\theta_{\delta}}\right)
\end{matrix}\right)\\
 & \succeq\left(0,0\right).
\end{align*}
By definition of $\left(\overline{\varphi_{\delta}},\underline{\psi_{\delta}}\right)$,
for all $\left(t,x\right)$ such that 
\[
\left(\overline{u_{\delta}},\underline{v_{\delta}}\right)\left(t,x\right)=\left(1,\underline{\psi_{\delta}}\left(x-c_{2}^{\delta}t\right)\right),
\]
we find, using $\underline{\psi_{\delta}}\leq1+\delta$, 
\begin{align*}
P\left(\overline{u_{\delta}},\underline{v_{\delta}}\right)-F\left(\overline{u_{\delta}},\underline{v_{\delta}}\right) & =\left(\begin{matrix}a\underline{\psi_{\delta}}\\
-c_{2}^{\delta}\underline{\psi_{\delta}}'-d\underline{\psi_{\delta}}''-r\underline{\psi_{\delta}}\left(1-\underline{\psi_{\delta}}-b\right)
\end{matrix}\right)\\
 & =\left(\begin{matrix}a\underline{\psi_{\delta}}\\
-r\underline{\psi_{\delta}}\left(2\delta+b\left(\overline{\varphi_{\delta}}-1\right)\right)
\end{matrix}\right)\\
& \succeq\left(\begin{matrix}0\\
-2r\delta\underline{\psi_{\delta}}
\end{matrix}\right)\\
 & \succeq\left(0,0\right).
\end{align*}
 Similarly, for all $\left(t,x\right)$ such that 
\[
\left(\overline{u_{\delta}},\underline{v_{\delta}}\right)\left(t,x\right)=\left(\overline{\varphi_{\delta}}\left(x-c_{2}^{\delta}t\right),\underline{\psi_{\delta}}\left(x-c_{2}^{\delta}t\right)\right),
\]
we find
\begin{align*}
P\left(\overline{u_{\delta}},\underline{v_{\delta}}\right)-F\left(\overline{u_{\delta}},\underline{v_{\delta}}\right) & =\left(\begin{matrix}-c_{2}^{\delta}\overline{\varphi_{\delta}}'-\overline{\varphi_{\delta}}''-\overline{\varphi_{\delta}}\left(1-\overline{\varphi_{\delta}}-a\underline{\psi_{\delta}}\right)\\
-c_{2}^{\delta}\underline{\psi_{\delta}}'-d\underline{\psi_{\delta}}''-r\underline{\psi_{\delta}}\left(1-\underline{\psi_{\delta}}-b\overline{\varphi_{\delta}}\right)
\end{matrix}\right)\\
 & =\left(\begin{matrix}\delta\overline{\varphi_{\delta}}\\
-2r\delta\underline{\psi_{\delta}}
\end{matrix}\right)\\
 & \succeq\left(0,0\right).
\end{align*}
By \lemref{w_super-solution} and \lemref{interface_2}, for all $\left(t,x\right)$
such that 
\[
\left(\overline{u_{\delta}},\underline{v_{\delta}}\right)\left(t,x\right)=\left(\overline{w_{\delta}}\left(t,x\right),\underline{\psi_{\delta}}\left(x-c_{2}^{\delta}t\right)\right),
\]
we find, using $\overline{w_{\delta}}\leq\overline{\varphi_{\delta}}$ (\lemref{interface_2}(1)),
\begin{align*}
P\left(\overline{u_{\delta}},\underline{v_{\delta}}\right)-F\left(\overline{u_{\delta}},\underline{v_{\delta}}\right) & =\left(\begin{matrix}\partial_{t}\overline{w_{\delta}}-\partial_{xx}\overline{w_{\delta}}-\overline{w_{\delta}}\left(1-\overline{w_{\delta}}-a\underline{\psi_{\delta}}\right)\\
-c_{2}^{\delta}\underline{\psi_{\delta}}'-d\underline{\psi_{\delta}}''-r\underline{\psi_{\delta}}\left(1-\underline{\psi_{\delta}}-b\overline{w_{\delta}}\right)
\end{matrix}\right)\\
 & =\left(\begin{matrix}\overline{w_{\delta}}\left(\overline{w_{\delta}}+a\underline{\psi_{\delta}}\right)\\
-r\underline{\psi_{\delta}}\left(2\delta+b\left(\overline{\varphi_{\delta}}-\overline{w_{\delta}}\right)\right)
\end{matrix}\right)\\
 & \succeq\left(0,0\right).
\end{align*}
By \lemref{pi_subsolution_and_estimate} and \lemref{interface_2},
for all $\left(t,x\right)$ such that 
\[
\left(\overline{u_{\delta}},\underline{v_{\delta}}\right)\left(t,x\right)=\left(\overline{w_{\delta}}\left(t,x\right),\underline{\pi_{\delta}}\left(x-c_{1}t\right)\right),
\]
we find, using $\overline{w_{\delta}}\leq\frac{\delta}{b}$ (\lemref{interface_2}(3)),
\begin{align*}
P\left(\overline{u_{\delta}},\underline{v_{\delta}}\right)-F\left(\overline{u_{\delta}},\underline{v_{\delta}}\right) & =\left(\begin{matrix}\partial_{t}\overline{w_{\delta}}-\partial_{xx}\overline{w_{\delta}}-\overline{w_{\delta}}\left(1-\overline{w_{\delta}}-a\underline{\pi_{\delta}}\right)\\
-c_{1}\underline{\pi_{\delta}}'-d\underline{\pi_{\delta}}''-r\underline{\pi_{\delta}}\left(1-\underline{\pi_{\delta}}-b\overline{w_{\delta}}\right)
\end{matrix}\right)\\
 & \succeq\left(\begin{matrix}\overline{w_{\delta}}\left(\overline{w_{\delta}}+a\underline{\pi_{\delta}}\right)\\
-r\underline{\pi_{\delta}}\left(\delta+\underline{\pi_{\delta}}-b\overline{w_{\delta}}\right)
\end{matrix}\right)\\
 & \succeq\left(0,0\right).
\end{align*}
By \lemref{beta_subsolution}, \lemref{interface_2} and construction
of $B$, for all $\left(t,x\right)$ such that 
\[
\left(\overline{u_{\delta}},\underline{v_{\delta}}\right)\left(t,x\right)=\left(\overline{w_{\delta}}\left(t,x\right),\underline{\beta_{\delta}}\left(x-c_{1}t\right)\right),
\]
we find, By definition of $\underline{\beta_\delta}$, and that of $\eta_\beta$ in Subsection \ref{subsec:4.3.1},
\begin{align*}
P\left(\overline{u_{\delta}},\underline{v_{\delta}}\right)-F\left(\overline{u_{\delta}},\underline{v_{\delta}}\right) & =\left(\begin{matrix}\partial_{t}\overline{w_{\delta}}-\partial_{xx}\overline{w_{\delta}}-\overline{w_{\delta}}\left(1-\overline{w_{\delta}}-a\underline{\beta_{\delta}}\right)\\
-c_{1}\underline{\beta_{\delta}}'-d\underline{\beta_{\delta}}''-r\underline{\beta_{\delta}}\left(1-\underline{\beta_{\delta}}-b\overline{w_{\delta}}\right)
\end{matrix}\right)\\
 & \succeq\left(\begin{matrix}\overline{w_{\delta}}\left(\overline{w_{\delta}}+a\underline{\beta_{\delta}}\right)\\
rb\underline{\beta_{\delta}}\left(\overline{w_{\delta}}-B\textup{e}^{-\Lambda_{\delta}\left(c_{2}^{\delta},c_{1}\right)\left(x-c_{1}t+\xi_{\beta}\right)}\right)
\end{matrix}\right)\\
 & \succeq\left(0,0\right).
\end{align*}

Finally, we consider the differential inequalities associated with $(\underline{u_\delta}, \overline{v_\delta})$. 
By definition of $\underline{\chi_{\delta}}$, for all $\left(t,x\right)$
such that 
\[
\left(\underline{u_{\delta}},\overline{v_{\delta}}\right)\left(t,x\right)=\left(\underline{\chi_{\delta}}\left(x-c_{2}t\right),1\right),
\]
we find 
\begin{align*}
P\left(\underline{u_{\delta}},\overline{v_{\delta}}\right)-F\left(\underline{u_{\delta}},\overline{v_{\delta}}\right) & =\left(\begin{matrix}-c_{2}\underline{\chi_{\delta}}'-\underline{\chi_{\delta}}''-\underline{\chi_{\delta}}\left(1-a-\underline{\chi_{\delta}}\right)\\
rb\underline{\chi_{\delta}}
\end{matrix}\right)\\
 & =\left(\begin{matrix}0\\
rb\underline{\chi_{\delta}}
\end{matrix}\right)\\
 & \preceq\left(0,0\right).
\end{align*}
By \lemref{w_sub-solution}, \lemref{interface_0} and by construction of $A=2C_{\delta}\textup{e}^{\lambda_v\left(c_1\right)x_w}$, for all $\left(t,x\right)$
such that 
\[
\left(\underline{u_{\delta}},\overline{v_{\delta}}\right)\left(t,x\right)=\left(\underline{w_{\delta}}\left(t,x\right),1\right),
\]
we find, using $C_{\delta}\textup{e}^{-\lambda_{v}\left(c_{1}\right)\left(x-c_{1}t\right)}\geq 1$,
\begin{align*}
P\left(\underline{u_{\delta}},\overline{v_{\delta}}\right)-F\left(\underline{u_{\delta}},\overline{v_{\delta}}\right) & =\left(\begin{matrix}\partial_{t}\underline{w_{\delta}}-\partial_{xx}\underline{w_{\delta}}-\underline{w_{\delta}}\left(1-\underline{w_{\delta}}-a\right)\\
rb\underline{w_{\delta}}
\end{matrix}\right)\\
 & \preceq\left(\begin{matrix}a\underline{w_{\delta}}\left(1-A\textup{e}^{-\lambda_{v}\left(c_{1}\right)\left(x-c_{1}t+x_{w}\right)}\right)\\
0
\end{matrix}\right)\\
 & \preceq\left(\begin{matrix}a\underline{w_{\delta}}\left(1-2C_{\delta}\textup{e}^{-\lambda_{v}\left(c_{1}\right)\left(x-c_{1}t\right)}\right)\\
0
\end{matrix}\right)\\
 & \preceq\left(0,0\right).
\end{align*}
 Similarly, for all $\left(t,x\right)$ such that 
\[
\left(\underline{u_{\delta}},\overline{v_{\delta}}\right)\left(t,x\right)=\left(\underline{w_{\delta}}\left(t,x\right),C_{\delta}\textup{e}^{-\lambda_{v}\left(c_{1}\right)\left(x-c_{1}t\right)}\right),
\]
we find
\begin{align*}
P\left(\underline{u_{\delta}},\overline{v_{\delta}}\right)-F\left(\underline{u_{\delta}},\overline{v_{\delta}}\right) & =\left(\begin{matrix}\partial_{t}\underline{w_{\delta}}-\partial_{xx}\underline{w_{\delta}}-\underline{w_{\delta}}\left(1-\underline{w_{\delta}}-aC_{\delta}\textup{e}^{-\lambda_{v}\left(c_{1}\right)\left(x-c_{1}t\right)}\right)\\
rb\underline{w_{\delta}}
\end{matrix}\right)\\
 & \preceq\left(\begin{matrix}a\underline{w_\delta}\textup{e}^{-\lambda_{v}\left(c_{1}\right)\left(x-c_{1}t\right)}\left(C_\delta -A\textup{e}^{-\lambda_v\left(c_1\right)x_w}\right)\\
0
\end{matrix}\right)\\
 & \preceq\left(0,0\right).
\end{align*}

\subsection{Construction of the super-solutions for \thmref{Compactly_supported}}\label{subsec:4.4}
In this subsection, we prove \propref{Super-solution_compactly_supported}.

Let $c_{2}>\max\left(c_{\textup{LLW}},f^{-1}\left(2\sqrt{rd}\right)\right)$
and let $\delta\in\left(0,\delta^{\star}\right)$ such that $c_{\textup{LLW}}^{\delta}<c_{2}$. Define
\[
c_{1}^{\delta}=2\sqrt{r\left(1-2\delta\right)d}-\delta.
\]

Recall from \subsecref{f_delta_Lambda_delta} that, given a fixed $\tilde{c}$, the function $c\mapsto\Lambda_{\delta}\left(c,\tilde{c}\right)$ is 
decreasing and bijectively maps $\left[2\sqrt{1-a_{\delta}},+\infty\right)$ onto 
\[
\left(\frac{1}{2}\left(\tilde{c}-\sqrt{\tilde{c}^{2}-4\left(a_{\delta}+1\right)}\right),\frac{1}{2}\left(\tilde{c}-\sqrt{\tilde{c}^{2}-4\left(\tilde{c}\sqrt{1-a_{\delta}}+2a_{\delta}-1\right)}\right)\right].
\] Thus the equation
\[
\Lambda_\delta\left(c_2^\delta, c^\delta_1\right) = \Lambda\left(c_2, 2\sqrt{rd}\right).
\]
admits a unique solution $c_2^\delta$ if and only if 
\begin{equation}\label{eq:rangeee}
c^\delta_1-\sqrt{\left(c^\delta_1\right)^{2}-4\left(a_{\delta}+1\right)}<2\Lambda\left(c_2, 2\sqrt{rd}\right)\leq c^\delta_1-\sqrt{\left(c^\delta_1\right)^{2}-4\left(c^\delta_1\sqrt{1-a_{\delta}}+2a_{\delta}-1\right)}.
\end{equation}
Since  $c_2\in\left(c_{\textup{LLW}},+\infty\right)\subset\left(2\sqrt{1-a},+\infty\right)$, we have by the above discussion 
\[
2\sqrt{rd}-\sqrt{\left(2\sqrt{rd}\right)^{2}-4\left(a+1\right)}<2\Lambda\left(c_2, 2\sqrt{rd}\right)< 2\sqrt{rd}-\sqrt{\left(2\sqrt{rd}\right)^{2}-4\left(2\sqrt{rd}\sqrt{1-a}+2a-1\right)}.
\]
By the facts that  $c^\delta_1 \to 2\sqrt{rd}$ and $a_\delta \to a$ as $\delta \to 0$, we deduce that 
we can in fact assume that $\delta$ is so small that (\ref{eq:rangeee}) holds. 
Hence $c_2^\delta$ is well-defined. 

Furthermore, by continuity, $c_{2}^{\delta}$ converges to $c_{2}$ as $\delta\to 0$, and thus $c^\delta_2 > c^\delta_{\textup{LLW}}$. In summary, we can assume
that $\delta$ is so small that $c_{1}^{\delta}$ and $c_{2}^{\delta}$ are well-defined, respectively close 
to $c_1$ and $c_2$, and satisfy the following:

\begin{equation}\label{eq:c2delta}
c_{\textup{LLW}}^{\delta}<c_{2}^{\delta} \quad \text{ and }\quad 
\Lambda_{\delta}\left(c_{2}^{\delta},c_{1}^{\delta}\right)=\Lambda\left(c_{2},2\sqrt{rd}\right)
\end{equation}

\subsubsection{Super-solution}\label{subsec:4.4.1}

The pair $\left(\overline{u_{\delta}},\underline{v_{\delta}}\right)$
is defined by 
\[
\overline{u_{\delta}}\left(t,x\right)=\left\{ \begin{matrix}\min\left(1,\overline{\varphi_{\delta,c_{2}^{\delta}}}\left(x-c_{2}^{\delta}t-\zeta_{1,\kappa}\right)\right) & \text{if }x<x_{2}\left(t\right)+\zeta_{1,\kappa}\\
\overline{w_{\delta,c_{2}^{\delta},c_{1}^{\delta}}}\left(t,x-\zeta_{1,\kappa}-\zeta_{2}\right) & \text{if }x\geq x_{2}\left(t\right)+\zeta_{1,\kappa}
\end{matrix}\right.,
\]
\[
\underline{v_{\delta}}\left(t,x\right)=\left\{ \begin{matrix}\max\left(0,\underline{\theta_{\delta,c_{2}^{\delta},A_{\kappa}}}\left(x-c_{2}^{\delta}t\right)\right) & \text{if }x<\xi_{1,\kappa}+c_{2}^{\delta}t\\
\underline{\psi_{\delta,c_{2}^{\delta}}}\left(x-c_{2}^{\delta}t-\zeta_{1,\kappa}\right) & \text{if }x\in\left[\xi_{1,\kappa}+c_{2}^{\delta}t,x_{3}\left(t\right)+\zeta_{1,\kappa}\right)\\
\underline{\omega_{\delta,R_{\delta}}}\left(x-c_{1}^{\delta}t-\zeta_{1,\kappa}-\zeta_{3}\right) & \text{if }x\geq x_{3}\left(t\right)+\zeta_{1,\kappa}
\end{matrix}\right.,
\]
where 
\begin{itemize}
\item $\kappa\in\left(0,\delta\right]$ is fixed so small that, for all
$t\geq0$, $\xi_{1,\kappa}+c_{2}^{\delta}t<x_{2}\left(t\right)+\zeta_{1,\kappa}$
(see \lemref{interface_1});
\item $\zeta_{3}$ is fixed so large that, for all $t\geq0$, $x_{2}\left(t\right)<x_{3}\left(t\right)$
(see \lemref{interface_3_small_c_1}). 
\end{itemize}

Thus, we have 
\[
\xi_{1,\kappa}+c_{2}^{\delta}t<x_{2}\left(t\right)+\zeta_{1,\kappa}<x_{3}\left(t\right)+\zeta_{1,\kappa}\quad\text{ for all }t\geq0.
\]

\subsubsection{Cleansing}

Just as in the previous case, we normalize and simplify the notations
so that $x_{2}\left(0\right)=0$ and the super-solution is defined
as follows:
\[
\overline{u_{\delta}}\left(t,x\right)=\left\{ \begin{matrix}\min\left(1,\overline{\varphi_{\delta}}\left(x-c_{2}^{\delta}t\right)\right) & \text{if }x<x_{2}\left(t\right)\\
\overline{w_{\delta}}\left(t,x\right) & \text{if }x\geq x_{2}\left(t\right)
\end{matrix}\right.,
\]
\[
\underline{v_{\delta}}\left(t,x\right)=\left\{ \begin{matrix}\max\left(0,\underline{\theta_{\delta}}\left(x-c_{2}^{\delta}t\right)\right) & \text{if }x<\xi_{1}+c_{2}^{\delta}t\\
\underline{\psi_{\delta}}\left(x-c_{2}^{\delta}t\right) & \text{if }x\in\left[\xi_{1}+c_{2}^{\delta}t,x_{3}\left(t\right)\right)\\
\underline{\omega_{\delta}}\left(x-c_{1}^{\delta}t\right) & \text{if }x\geq x_{3}\left(t\right)
\end{matrix}\right..
\]

\begin{figure}
\resizebox{\hsize}{!}{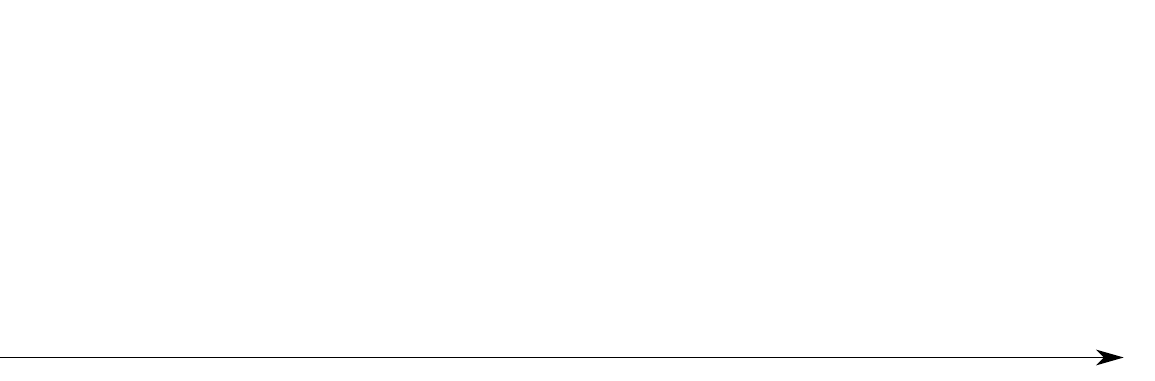}

\caption{Super-solution $\left(\overline{u_{\delta}},\underline{v_{\delta}}\right)$
for \thmref{Compactly_supported}}
\label{fig:1.5}
\end{figure}

\subsubsection{Verification of the differential inequalities}

Just as in the previous case, we verify that $\left(\overline{u_{\delta}},\underline{v_{\delta}}\right)$
is indeed a super-solution. The only new component to account for
is $\underline{\omega_{\delta}}$, which can be handled easily in view
of its definition.

\subsection{Construction of the sub-solutions for \thmref{Nonexistence}}
In this subsection, we prove \propref{Sub-solution_nonexistence}.

Let $c_{1}>2\sqrt{rd}$ and $c_{2}>c_{\textup{LLW}}$ such that $c_{1}\geq c_{2}$
and $c_{1}<f\left(c_{2}\right)$. Let $c>c_{2}$ so close to $c_{2}$
that $c_{1}<f\left(c\right)$ and let $\delta\in\left(0,\delta^{\star}\right)$ and
\[
\tilde{c}\in\left(\max\left(c_{1},f\left(c\right)-4\sqrt{a}\right),f\left(c\right)\right).
\]

\subsubsection{Sub-solution}\label{subsec:4.5.1}

The pair $\left(\underline{u_{\delta,\zeta,\kappa}},\overline{v_{\delta,\zeta}}\right)$
is defined by 
\[
\underline{u_{\delta,\zeta,\kappa}}\left(t,x\right)=\left\{ \begin{matrix}\underline{\alpha_{L}}\left(x\right) & \text{if }x<x_{0,\kappa}\left(t\right)\\
\underline{\chi_{c}}\left(x-ct-\zeta_{0}\right) & \text{if }x\in\left[x_{0,\kappa}\left(t\right),x_{1}\left(t\right)+\zeta_{0}\right)\\
\frac{\underline{\chi_{c}}\left(\zeta-\zeta_{0}\right)}{\underline{z_{c,\tilde{c},\delta}}\left(0,X_{z}-\zeta_{0}\right)}\underline{z_{c,\tilde{c},\delta}}\left(t,x-\zeta_{0}-\zeta\right) & \text{if }x\geq x_{1}\left(t\right)+\zeta_{0}
\end{matrix}\right.,
\]
\[
\overline{v_{\delta,\zeta}}\left(t,x\right)=\min\left(1,\textup{e}^{-\lambda_{v}\left(\tilde{c}\right)\left(x-y_{\delta,\zeta}-\tilde{c}t\right)}\right),
\]
where
\[
y_{\delta,\zeta}=\frac{\ln\delta-\ln\left(2a\right)}{\lambda_{v}\left(\tilde{c}\right)}+\zeta_{0}+\zeta
\]
and $\kappa\in\left(0,\min\left(\frac{1-a}{2},\frac{\delta}{2}\right)\right)$
and $\zeta>L$ are parameters.

Note that $\overline{v_{\delta,\zeta}}\left(t,x\right)\leq\frac{\delta}{2a}$
for all $x\geq\zeta_{0}+\zeta+\tilde{c}t$. 
By \lemref{interface_1_nonexistence}(3), we have $x_1(t) > \tilde{c}t + \zeta$ and thus 
$\overline{v_{\delta, \zeta}}\left(t,x\right)\leq\frac{\delta}{2a}$
for all $x\geq x_{1}\left(t\right)+\zeta_{0}$. Notice also that the
support of $x\mapsto\underline{u_{\delta,\zeta,\kappa}}\left(0,x\right)$
is included in $\left[0,L+\zeta+2R_{z}\right]$. 

\subsubsection{Cleansing}

Again, we normalize and simplify:
\[
\underline{u_{\delta,\zeta,\kappa}}\left(t,x\right)=\left\{ \begin{matrix}\underline{\alpha}\left(x\right) & \text{if }x<x_{0}\left(t\right)\\
\underline{\chi}\left(x-ct\right) & \text{if }x\in\left[x_{0}\left(t\right),x_{1}\left(t\right)\right)\\
\underline{z_{\delta}}\left(t,x-\zeta\right) & \text{if }x\geq x_{1}\left(t\right)
\end{matrix}\right.,
\]
\[
\overline{v_{\delta,\zeta}}\left(t,x\right)=\min\left(1,C_{\delta}\textup{e}^{-\lambda_{v}\left(\tilde{c}\right)\left(x-\zeta-\tilde{c}t\right)}\right).
\]

\begin{figure}
\resizebox{\hsize}{!}{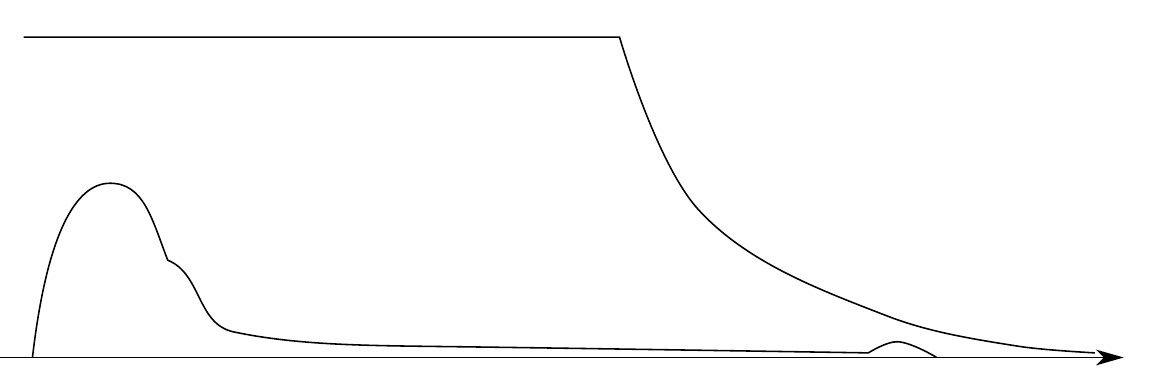}

\caption{Sub-solution $\left(\underline{u_{\delta,\zeta,\kappa}},\overline{v_{\delta,\zeta}}\right)$
for \thmref{Nonexistence}}
 \label{fig:1.4}
\end{figure}

\subsubsection{Verification of the differential inequalities}

Again, we verify that $\left(\underline{u_{\delta,\zeta,\kappa}},\overline{v_{\delta,\zeta}}\right)$
is a sub-solution. The only new components are $\underline{\alpha}$
and $\underline{z_{\delta}}$, the latter being handled with \lemref{z_sub-solution}.

\section{Discussion}

As a preliminary remark, let us point out that analogous results can
be obtained with the exact same method for the coexistence case $a<1$,
$b<1$. In that case the solutions are characterized by a profile connecting $(0,0)$ to $(0,1)$ to $\left(\frac{1-a}{1-ab}, \frac{1-b}{1-ab} \right)$.

\subsection{On the consequences of \thmref{Compactly_supported}, \thmref{Nonexistence} and \thmref{Generalized_terraces}\label{subsec:discussion_consequences}}
Consider here the Cauchy problem associated with \thmref{Compactly_supported}, namely the initial condition $u_{0}$ of the slower and stronger species has a support included in $(-\infty,0]$ while the initial condition $v_{0}$ of the faster and weaker species has compact support. Treating $2\sqrt{rd}$ as a parameter, \thmref{Compactly_supported} says that, while the species $v$ always spreads at speed $2\sqrt{rd}$ if it persists, the species $u$: 
\begin{itemize}
\item lags behind $v$ and spreads at speed $c_{\textup{LLW}}$ if $2\sqrt{rd}\geq f(c_{\textup{LLW}})$;
\item lags behind $v$ and spreads at speed $f^{-1}(2\sqrt{rd})>c_{\textup{LLW}}$ if $2<2\sqrt{rd}<f(c_{\textup{LLW}})$;
\item drives $v$ to extinction and spreads at speed $2$ if $2\sqrt{rd}< 2$.
\end{itemize}

In general, it is unclear whether $c_{\textup{LLW}}=2\sqrt{1-a}$ or not. Hence
the condition $2\sqrt{rd}\geq f\left(c_{\textup{LLW}}\right)$ might be difficult
to check in practice. However, since 
\[
\max_{c\in\left[2\sqrt{1-a},2\right]}f\left(c\right)=f(2\sqrt{1-a})=2\left(\sqrt{1-a}+\sqrt{a}\right),
\]
the condition $\sqrt{rd}>\sqrt{1-a}+\sqrt{a}$ always implies $2\sqrt{rd}>f\left(c_{\textup{LLW}}\right)$
and consequently always implies that $u$ invades at speed $c_{\textup{LLW}}$.
In particular, the maximum of $a\mapsto\sqrt{1-a}+\sqrt{a}$ in $\left(0,1\right)$
being $\sqrt{2}$, if $rd>2$, then $u$ invades at speed $c_{\textup{LLW}}$
independently of the value of $a$ and $b$. In ecological terms,
if $v$ is a sufficiently fast invader, then it decelerates optimally
any stronger and slower competitor. 

Applied to a pair $\left(u_{0},0\right)$, the nonexistence result
reduces to a well-known property of the KPP equation satisfied by
$u$ in isolation: all solutions spread at least at speed $2$. 

In view of \figref{Set_of_speeds}, it is tempting to refer to the
pair of speeds 
\[
\left(c_{2}^{\star},c_{1}^{\star}\right)=\left(\max\left(c_{\textup{LLW}},f^{-1}\left(2\sqrt{rd}\right)\right),2\sqrt{rd}\right)
\]
as a \textquotedblleft minimal pair\textquotedblright{} and to the corresponding
terrace as a \textquotedblleft minimal terrace\textquotedblright. But in our opinion,
such a terminology would be misleading. Indeed,
a very natural conjecture in view of the KPP literature is that the
propagating terraces attract initial data with appropriate exponential
decays ($\lambda_{v}\left(c_{1}\right)$ for $v_{0}$ and $\Lambda\left(c_{2},c_{1}\right)$
for $u_{0}$). Assume this conjecture is true indeed, assume $2<2\sqrt{rd}<f\left(c_{\textup{LLW}}\right)$
and fix a compactly supported or Heavyside-like $u_{0}$. Then decreasing
the decay of $v_{0}$ will increase the speed of $v$ but
decrease that of $u$ (with the obvious convention that a compactly
supported $v_{0}$ has an infinite decay). 

More generally, this paper presents several results that are complementary
to that of Lewis,
Li and Weinberger, with several surprising consequences. It shows
that $c_{\textup{LLW}}$ is not always the relevant speed when predicting the
speed of the invasion of $u$ in the territory of $v$. The initial
spatial distribution of $v$ has to be taken into account and in particular,
it can be inappropriate to approximate a very large territory by an
unbounded territory. Also, even if $c_{\textup{LLW}}$ is linearly determined
and therefore only depends on $a$, the speed of $u$ might still
depend on $rd$. 

It would be interesting to verify the existence of nonlocally pulled fronts
in real biological invasions. Indeed, at first glance, our result
might very well be described by ecological modelers as a strong case
against diffusion equations: dispersal operators preserving compact
supports, like the nonlinear diffusion of the porous form, $\partial_{t}u-\Delta\left(u^{m}\right)$,
should never lead to such a result.
Actually, Du and Wu \cite{Du_Wu_2018} studied the analogous problem in a model with 
free boundaries, first considered by Guo and Wu \cite{Guo_Wu_2015}, and showed
under appropriate assumptions on the initial conditions that the spreading speed of
the slower species is uniquely determined by the semi-wave system whose
solutions were classified by Du, Wang and Zhou \cite{Du_Wang_Zhou_2017}.
Therefore, in such a model, the second speed is never enhanced.
Let us also point out here that Li \cite{Li_2018} very recently addressed 
similar questions in the framework of integro-difference systems and did not find nonlocal pulling.

In 2014, Holzer and Scheel \cite{Holzer_Scheel} considered a partially decoupled two-species system,
which contains as a particular case the system (\ref{eq:competition_diffusion_system}) when $b=0$.
They found a sufficient condition for nonlocal pulling which is consistent
with our necessary and sufficient condition $c_{\textup{nlp}}>c_{\textup{LLW}}$.
Although we did not know their work at the time of writing of this paper, we discovered it
during the reviewing process and acknowledge that they were the first to uncover the phenomenon
of nonlocally pulled fronts. Interestingly, the same year, a more applied paper on 
horizontally transmitted hitchhiking traits, by Venegas, Allen and Evans \cite{Venegas_Ortiz_}, 
suggested a speed formula that is consistent with our expression 
for $c_{\textup{nlp}}$; yet a rigorous proof was missing. 

To the best of our knowledge, the present paper is the first one to investigate general
propagating terraces for reaction--diffusion systems with unstable intermediate steady states.
By showing that multidimensional manifolds of terraces do exist, it opens new interesting
research directions.

\subsection{On the boundary of the set of admissible pairs of speeds}

In the present paper, the question of existence at the boundary of
the set of admissible pairs is not settled. It is in fact more subtle
than expected.

Assuming only $2\sqrt{rd}>2$, this boundary is naturally partitioned
as $V\cup G\cup H\cup D$, where 
\[
V=\left\{ c_{\textup{LLW}}\right\} \times\left(\max\left(2\sqrt{rd},f\left(c_{\textup{LLW}}\right)\right),+\infty\right),
\]
\[
G=\left\{ \left(c,f\left(c\right)\right)\ |\ c\in\left[c_{\textup{LLW}},\max\left(c_{\textup{LLW}},f^{-1}\left(2\sqrt{rd}\right)\right)\right)\right\} ,
\]
\[
H=\left[\max\left(c_{\textup{LLW}},f^{-1}\left(2\sqrt{rd}\right)\right),2\sqrt{rd}\right)\times\left\{ 2\sqrt{rd}\right\} ,
\]
\[
D=\left\{ \left(c,c\right)\ |\ c\geq2\sqrt{rd}\right\} ,
\]
 and where $G$ is possibly empty whereas $V$, $H$ and $D$ are
always nonempty. 

Points on $V\cup G$ should correspond to pairs $\left(u_{0},v_{0}\right)$
with $u_{0}$ supported in a left half-line and $v_{0}$ exponentially
decaying. Using both \thmref{Generalized_terraces} and \thmref{Nonexistence} 
as well as a limiting argument and
the comparison principle, it is possible to obtain the existence of
such a terrace with a pair $\left(u_{0},v_{0}\right)$ of this form. 

However, on $H$, which corresponds naively to pairs $\left(u_{0},v_{0}\right)$
with compactly supported $v_{0}$ and exponentially decaying $u_{0}$,
such a construction seems to be impossible. A different, likely more
delicate, argument is needed to deal with $H$. Still, we believe
existence holds there.

On the contrary, on $D$, the question remains completely open.
Indeed, on $D$, propagating terraces reduce to non-monotonic traveling waves
connecting $\left(0,0\right)$ to $\left(1,0\right)$ with an intermediate bump
of $v$. To the best of our knowledge, such traveling waves have never been 
studied. Even though it might be tempting to conjecture their nonexistence, 
we prefer to remain cautious here. 

\subsection{On the proofs}

In the proof of \thmref{Compactly_supported}, the approximated speed
$c_{1}^{\delta}$ is necessary in the following sense: it is impossible to construct another
$\underline{v}$ spreading this time exactly at speed $2\sqrt{rd}$.
This is an immediate consequence of the Bramson shift for the KPP
equation \cite{Bramson_1983}: the level sets of the solution of the
KPP equation satisfied by $v$ in isolation with compactly supported
initial data are asymptotically located at $2\sqrt{rd}t+s_{\textup{Bramson}}\left(t\right)$,
with $s_{\textup{Bramson}}\left(t\right)=-\frac{3}{2\sqrt{r}}\log t+o\left(\log t\right)$.
By comparison, it is then easily verified that for the solution $\left(u,v\right)$
of our competitive system, there exists a time shift $s\left(t\right)\leq s_{\textup{Bramson}}\left(t\right)$
such that the level sets of $v$ in our problem are located at $2\sqrt{rd}t+s\left(t\right)$.

Similarly, in the proofs of \thmref{Compactly_supported} and of 
\thmref{Generalized_terraces}, we believe that
the approximated speed $c_{2}^{\delta}$ are needed to account for a time
shift $\tilde{s}\left(t\right)\neq0$ describing the position of the
level sets of $u$. The characterization of this shift is completely
open; the only hint provided by our approach is that $\tilde{s}\left(t\right)$
is asymptotically nonnegative (contrarily to $s\left(t\right)$ and
$s_{\textup{Bramson}}\left(t\right)$). 

\section*{Acknowledgments}

This research was conducted at the Ohio State University and co-funded
by the Mathematics Research Institute of the Ohio State University
and the ANR NONLOCAL project. The authors thank deeply both institutions.
The authors also thank M. Holzer for very interesting discussions and
the anonymous referees for their work.

\appendix
\section{On competition\textendash diffusion traveling waves connecting $\left(1,0\right)$ to $\left(0,1\right)$}
In this appendix, the parameters $\left(d,r,a,b\right)$ are not fixed anymore and can vary. We define 
\[
\Pi=\left(0,+\infty\right)^{2}\times\left(0,1\right)\times\left(1,+\infty\right).
\]
For all $\left(d,r,a,b\right)\in\Pi$, $c_{\textup{LLW}}^{d,r,a,b}$ denotes
the associated spreading speed of the system (\ref{eq:related_competition_diffusion_system}).
Subsequently, we define
\[
E=\left\{ \left(c,d,r,a,b\right)\in\left(0,+\infty\right)\times\Pi\ |\ c\geq c_{\textup{LLW}}^{d,r,a,b}\right\}.
\]
\subsection{Exact exponential decays}
For all $P=\left(c,d,r,a,b\right)\in E$, we define
\[
\lambda_{1,P}^{-\infty}=\frac{\sqrt{c^{2}+4}-c}{2},
\]
\[
\lambda_{2,P}^{-\infty}=\frac{\sqrt{c^{2}+4rd\left(b-1\right)}-c}{2d},
\]
\[
\lambda_{1,P}^{+\infty}=\frac{c+\sqrt{c^{2}+4rd}}{2d},
\]
\[
\lambda_{2,P}^{+\infty}=\frac{c+\sqrt{c^{2}-4\left(1-a\right)}}{2},
\]
\[
\lambda_{3,P}^{+\infty}=\frac{c-\sqrt{c^{2}-4\left(1-a\right)}}{2}.
\]
\begin{lem}
\label{lem:Exponential_decays_profiles} Let $P=\left(c,d,r,a,b\right)\in E$
and $\left(\varphi,\psi\right)$ be a profile of traveling wave solution
of (\ref{eq:competition_diffusion_system}) with speed $c$. Define $R_{P}^{-\infty}:\lambda\mapsto\lambda^{2}+c\lambda-1$
and $R_{P}^{+\infty}:\lambda\mapsto d\lambda^{2}-c\lambda-r$.

Then the asymptotic behaviors of $\left(\varphi,\psi\right)$ are
as follows.
\begin{enumerate}
\item There exist $A>0$ and $B>0$ such that, as $\xi\to-\infty$:
\begin{enumerate}
\item if $\lambda_{2,P}^{-\infty}>\lambda_{1,P}^{-\infty}$, then 
\[
\left\{ \begin{matrix}\varphi\left(\xi\right)=1-A\textup{e}^{\lambda_{1,P}^{-\infty}\xi}+\text{h.o.t.}\\
\psi\left(\xi\right)=B\textup{e}^{\lambda_{2,P}^{-\infty}\xi}+\text{h.o.t.}
\end{matrix}\right.;
\]
\item if $\lambda_{2,P}^{-\infty}<\lambda_{1,P}^{-\infty}$, then $R_{P}^{-\infty}\left(\lambda_{2,P}^{-\infty}\right)<0$
and 
\[
\left\{ \begin{matrix}\varphi\left(\xi\right)=1+\frac{a}{R_{P}^{-\infty}\left(\lambda_{2,P}^{-\infty}\right)}B\textup{e}^{\lambda_{2,P}^{-\infty}\xi}+\text{h.o.t.}\\
\psi\left(\xi\right)=B\textup{e}^{\lambda_{2,P}^{-\infty}\xi}+\text{h.o.t.}
\end{matrix}\right.;
\]
\item if $\lambda_{2,P}^{-\infty}=\lambda_{1,P}^{-\infty}$, then $c+2\lambda_{2,P}^{-\infty}=\sqrt{c^{2}+4}>0$
and 
\[
\left\{ \begin{matrix}\varphi\left(\xi\right)=1-B\left|\xi\right|\textup{e}^{\lambda_{2,P}^{-\infty}\xi}+\text{h.o.t.}\\
\psi\left(\xi\right)=\frac{c+2\lambda_{2,P}^{-\infty}}{a}B\textup{e}^{\lambda_{2,P}^{-\infty}\xi}+\text{h.o.t.}
\end{matrix}\right..
\]
\end{enumerate}
\item There exist $A\in\mathbb{R}$, $B\in\mathbb{R}$ and $C\geq0$ such
that $B>0$ if $C=0$ and, as $\xi\to+\infty$:
\begin{enumerate}
\item if $c>2\sqrt{1-a}$, 
\begin{enumerate}
\item if $\lambda_{1,P}^{+\infty}<\lambda_{3,P}^{+\infty}$, then $A>0$
and
\[
\left\{ \begin{matrix}\varphi\left(\xi\right)=B\textup{e}^{-\lambda_{2,P}^{+\infty}\xi}+C\textup{e}^{-\lambda_{3,P}^{+\infty}\xi}+\text{h.o.t.}\\
\psi\left(\xi\right)=1-A\textup{e}^{-\lambda_{1,P}^{+\infty}\xi}+\text{h.o.t.}
\end{matrix}\right.;
\]
\item if $\lambda_{1,P}^{+\infty}=\lambda_{3,P}^{+\infty}$, then $A>0$
if $C=0$ and
\[
\left\{ \begin{matrix}\varphi\left(\xi\right)=\frac{2d\lambda_{1,P}^{+\infty}-c}{a}C\textup{e}^{-\lambda_{1,P}^{+\infty}\xi}+B\textup{e}^{-\lambda_{2,P}^{+\infty}\xi}+\text{h.o.t.}\\
\psi\left(\xi\right)=1-\left(A+C\xi\right)\textup{e}^{-\lambda_{1,P}^{+\infty}\xi}+\text{h.o.t.}
\end{matrix}\right.;
\]
\item if $\lambda_{1,P}^{+\infty}\in\left(\lambda_{3,P}^{+\infty},\lambda_{2,P}^{+\infty}\right)$,
then $R_{P}^{+\infty}\left(\lambda_{3,P}^{+\infty}\right)<0$, $A>0$
if $C=0$ and
\[
\left\{ \begin{matrix}\varphi\left(\xi\right)=B\textup{e}^{-\lambda_{2,P}^{+\infty}\xi}+C\textup{e}^{-\lambda_{3,P}^{+\infty}\xi}+\text{h.o.t.}\\
\psi\left(\xi\right)=1-A\textup{e}^{-\lambda_{1,P}^{+\infty}\xi}+\frac{rb}{R_{P}^{+\infty}\left(\lambda_{3,P}^{+\infty}\right)}C\textup{e}^{-\lambda_{3,P}^{+\infty}\xi}+\text{h.o.t.}
\end{matrix}\right.;
\]
\item if $\lambda_{1,P}^{+\infty}=\lambda_{2,P}^{+\infty}$, then $R_{P}^{+\infty}\left(\lambda_{3,P}^{+\infty}\right)<0$
and 
\[
\left\{ \begin{matrix}\varphi\left(\xi\right)=\frac{2d\lambda_{1,P}^{+\infty}-c}{a}B\textup{e}^{-\lambda_{1,P}^{+\infty}\xi}+C\textup{e}^{-\lambda_{3,P}^{+\infty}\xi}+\text{h.o.t.}\\
\psi\left(\xi\right)=1-B\xi\textup{e}^{-\lambda_{1,P}^{+\infty}\xi}+\frac{rb}{R_{P}^{+\infty}\left(\lambda_{3,P}^{+\infty}\right)}C\textup{e}^{-\lambda_{3,P}^{+\infty}\xi}+\text{h.o.t.}
\end{matrix}\right.;
\]
\item if $\lambda_{1,P}^{+\infty}>\lambda_{2,P}^{+\infty}$, then $R_{P}^{+\infty}\left(\lambda_{2,P}^{+\infty}\right)<0$,
$R_{P}^{+\infty}\left(\lambda_{3,P}^{+\infty}\right)<0$ and
\[
\left\{ \begin{matrix}\varphi\left(\xi\right)=B\textup{e}^{-\lambda_{2,P}^{+\infty}\xi}+C\textup{e}^{-\lambda_{3,P}^{+\infty}\xi}+\text{h.o.t.}\\
\psi\left(\xi\right)=1+\frac{rb}{R_{P}^{+\infty}\left(\lambda_{2,P}^{+\infty}\right)}B\textup{e}^{-\lambda_{2,P}^{+\infty}\xi}+\frac{rb}{R_{P}^{+\infty}\left(\lambda_{3,P}^{+\infty}\right)}C\textup{e}^{-\lambda_{3,P}^{+\infty}\xi}+\text{h.o.t.}
\end{matrix}\right.;
\]
\end{enumerate}
\item if $c=2\sqrt{1-a}$,
\begin{enumerate}
\item if $\lambda_{1,P}^{+\infty}<\lambda_{2,P}^{+\infty}$, then $A>0$
and 
\[
\left\{ \begin{matrix}\varphi\left(\xi\right)=\left(B+C\xi\right)\textup{e}^{-\lambda_{2,P}^{+\infty}\xi}+\text{h.o.t.}\\
\psi\left(\xi\right)=1-A\textup{e}^{-\lambda_{1,P}^{+\infty}\xi}+\text{h.o.t.}
\end{matrix}\right..
\]
\item if $\lambda_{1,P}^{+\infty}=\lambda_{2,P}^{+\infty}$, then $2d\lambda_{1,P}^{+\infty}-c=\sqrt{c^{2}+4rd}>0$
and
\[
\left\{ \begin{matrix}\varphi\left(\xi\right)=\frac{2d\lambda_{1,P}^{+\infty}-c}{a}\left(B+C\xi\right)\textup{e}^{-\lambda_{1,P}^{+\infty}\xi}+\text{h.o.t.}\\
\psi\left(\xi\right)=1-\left(B+\frac{1}{2}C\xi\right)\xi\textup{e}^{-\lambda_{1,P}^{+\infty}\xi}+\text{h.o.t.}
\end{matrix}\right.;
\]
\item if $\lambda_{1,P}^{+\infty}>\lambda_{2,P}^{+\infty}$, then $R_{P}^{+\infty}\left(\lambda_{2,P}^{+\infty}\right)<0$
and 
\[
\left\{ \begin{matrix}\varphi\left(\xi\right)=\left(B+C\xi\right)\textup{e}^{-\lambda_{2,P}^{+\infty}\xi}+\text{h.o.t.}\\
\psi\left(\xi\right)=1+\frac{rb}{R_{P}^{+\infty}\left(\lambda_{2,P}^{+\infty}\right)}\left(B+C\xi\right)\textup{e}^{-\lambda_{2,P}^{+\infty}\xi}+\text{h.o.t.}
\end{matrix}\right.;
\]
\end{enumerate}
\end{enumerate}
\end{enumerate}
\end{lem}

\begin{proof}
This result follows from a standard yet lengthy phase-plane analysis.
The detailed proof can be found for instance in Kan-on
\cite{Kan_On_1997} or in Morita\textendash Tachibana \cite{Morita_Tachibana}. 
\end{proof}

Compiling these estimates, we obtain the following two corollaries.

\begin{cor}
\label{cor:decays_near_unstable_state} Let $P=\left(c,d,r,a,b\right)\in E$
and $\left(\varphi,\psi\right)$ be a profile of traveling wave solution
of the corresponding system with speed $c$. Then there exist $i\in\left\{ 2,3\right\} $,
$C>0$, $D>0$ and $\left(i_{+},j_{+}\right)\in\left\{ 0,1\right\} \times\left\{ 0,1,2\right\} $
such that, as $\xi\to+\infty$,
\[
\left\{ \begin{matrix}\varphi\left(\xi\right)=C\xi^{i_{+}}\textup{e}^{-\lambda_{i,P}^{+\infty}\xi}+\text{h.o.t.}\\
\psi\left(\xi\right)=1-D\xi^{j_{+}}\textup{e}^{-\min\left(\lambda_{1,P}^{+\infty},\lambda_{i,P}^{+\infty}\right)\xi}+\text{h.o.t.}
\end{matrix}\right..
\]
\end{cor}

\begin{cor}
\label{cor:decays_near_stable_state}Let $P=\left(c,d,r,a,b\right)\in E$
and $\left(\varphi,\psi\right)$ be a profile of traveling wave solution
of the corresponding system with speed $c$. Let $i_{-}=2-\#\left\{ \lambda_{1,P}^{-\infty},\lambda_{2,P}^{-\infty}\right\} $. 

Then there exist $A>0$ and $B>0$ such that, as $\xi\to-\infty$,
\[
\left\{ \begin{matrix}\varphi\left(\xi\right)=1-A\left|\xi\right|^{i_{-}}\textup{e}^{\min\left(\lambda_{1,P}^{-\infty},\lambda_{2,P}^{-\infty}\right)\xi}+\text{h.o.t.}\\
\psi\left(\xi\right)=B\textup{e}^{\lambda_{2,P}^{-\infty}\xi}+\text{h.o.t.}
\end{matrix}\right..
\]
\end{cor}

\subsection{Component-wise monotonicity of the profiles}
Thanks to \corref{decays_near_unstable_state} and a sliding argument, we can show the component-wise monotonicity. 

\begin{prop}
\label{prop:monotonicity_of_the_profiles} Let $P=\left(c,d,r,a,b\right)\in E$. 
Let $\left(\varphi,\psi\right)\in\mathscr{C}^2\left(\mathbb{R},\left[0,1\right]^2\right)$ be a
profile of traveling wave solution of (\ref{eq:competition_diffusion_system}) with speed
$c$ connecting $\left(1,0\right)$ to $\left(0, 1\right)$. 

Then $\left(\varphi,\psi\right)$ is component-wise strictly monotonic, i.e. 
$$
\left(\varphi,\psi\right)(\xi_1)\succ \left(\varphi,\psi\right)(\xi_2) \quad \text{ whenever} \quad \xi_1 < \xi_2.
$$
\end{prop}

\begin{proof}
The proof relies upon a sliding argument. 

The sliding argument for monostable problems has three main steps: first, showing that if two profiles are correctly 
ordered at some point far on the left, then they remain correctly ordered everywhere on the left of this point; next, showing 
thanks to 
the first step and the exponential estimates at $+\infty$ that, up to some translation, the two profiles are globally ordered; 
finally, showing by optimizing the aforementioned translation that the two profiles actually coincide.

Notice that since the exponential estimates of \lemref{Exponential_decays_profiles} can be differentiated, they imply the 
component-wise strict monotonicity of $\left(\varphi,\psi\right)$ near $\pm\infty$. Thus we can define $R>0$ such that 
$\left(\varphi,\psi\right)$ is component-wise strictly monotonic in $\mathbb{R}\backslash\left[-R,R\right]$. In particular, we can assume that 
\begin{equation}\label{eq:monotone1}
(\varphi,\psi)(-R) \succ (\varphi,\psi) (\xi) \succ (\varphi,\psi)(R)\text{ for all } \xi \in (-R, R).
\end{equation}

\textbf{Step 1:} We claim that there is $\tau_1>0$ such that for all $\tau \geq \tau_1$, 
\begin{equation}\label{eq:monotone2}
\left(\varphi,\psi\right)\left(\xi-\tau\right)\succ\left(\varphi,\psi\right)\left(\xi\right) \quad \text{ for all }\xi \in \mathbb{R}.
\end{equation}
In view of the monotonicity of $(\varphi,\psi)$ in $\mathbb{R} \setminus(-R,R)$, and (\ref{eq:monotone1}), the claim clearly holds once we take $\tau_1 = 2R$.

\textbf{Step 2:} Define $\tau^{\star}$ to be the infimum of all $\tau\in (0, 2R]$ such that (\ref{eq:monotone2}) holds true. It remains to show that $\tau^{\star} = 0$. Suppose to the contrary that $\tau^{\star}>0$. By construction, 
\[
\left(\varphi,\psi\right)\left(\xi-\tau^{\star}\right)\succeq\left(\varphi,\psi\right)\left(\xi\right)\text{ for all }\xi\in\mathbb{R}.
\]
Moreover, by (\ref{eq:monotone1}) and monotonicity of $(\varphi,\psi)$ in $\mathbb{R}\setminus(-R,R)$, we see that 
for each $\tau \in \left[ \frac{\tau^{\star}}{2}, 2\tau^{\star}\right]$, 
\[
\left(\varphi,\psi\right)\left(\xi-\tau\right) \succ \left(\varphi,\psi\right)\left(\xi\right)\text{ for all } \xi \in \mathbb{R}\setminus (-R+\tau, R), 
\]
and in particular for all $\xi \in \mathbb{R} \setminus (-R + \tau^{\star}/2, R)$. 
By the minimality of $\tau^{\star}>0$, there exists $\xi^{\star} \in [-R + \tau^{\star}/2, R]$ 
such that equality holds for at least one of the components. The strong comparison principle yields 
\[
(\varphi,\psi)(\xi - \tau^{\star}) = (\varphi,\psi)(\xi) \quad \text{ for all }\xi \in \mathbb{R}.
\]
This implies $(\varphi,\psi)$ is periodic with period $\tau^{\star}$, and contradicts $(\varphi,\psi)(-\infty) = (1,0)$ and $(\varphi,\psi)(+\infty) = (0,1)$.
Hence $\tau^\star=0$ and, subsequently, for all $\tau>0$, we have 
\[
\left(\varphi,\psi\right)\left(\xi-\tau\right)\succ\left(\varphi,\psi\right)\left(\xi\right)\text{ for all }\xi\in\mathbb{R},
\]
which exactly means that $\left(\varphi,\psi\right)$ is component-wise strictly monotonic.
\end{proof}

\subsection{Ordering of the decays}
By a similar proof, we can characterize more precisely the decays. We recall that 
Roques\textendash Hosono\textendash Bonnefon\textendash Boivin \cite{Roques_Hosono} 
showed that the slow or fast decay problem is
related to the pulled or pushed front problem.

\begin{lem}
\label{lem:ordering_of_the_decays_of_the_profiles} Let $p=\left(d,r,a,b\right)\in\Pi$, $c\geq c_{\textup{LLW}}^{p}$ and $\hat{c} \geq c$. Define $P=\left(c,p\right)\in E$ and $\hat{P}=\left(\hat{c},p\right)\in E$.

Let $\left(\varphi,\psi\right)\in\mathscr{C}^2\left(\mathbb{R},\left[0,1\right]^2\right)$ and 
$\left(\hat{\varphi},\hat{\psi}\right)\in\mathscr{C}^2\left(\mathbb{R},\left[0,1\right]^2\right)$ be two
profiles of traveling wave solution of (\ref{eq:competition_diffusion_system}) with speed
$c$ and $\hat{c}$ respectively. Denote $\left(i,C,D,i_+,j_+\right)$ and $\left(\hat{i},\hat{C},\hat{D},\hat{i_+},\hat{j_+}\right)$
the quantities given by \corref{decays_near_unstable_state} when applied to $\left(\varphi,\psi\right)$ and $\left(\hat{\varphi},\hat{\psi}\right)$ respectively.

Then at least one of the following estimates fails:
\[
\hat{C}\xi^{\hat{i_{+}}}\textup{e}^{-\lambda_{\hat{i},\hat{P}}^{+\infty}\xi}=o\left(C\xi^{i_{+}}\textup{e}^{-\lambda_{i,P}^{+\infty}\xi}\right)\text{ as }\xi\to+\infty,
\]
\[
\hat{D}\xi^{\hat{j_{+}}}\textup{e}^{-\min\left(\lambda_{1,\hat{P}}^{+\infty},\lambda_{\hat{i},\hat{P}}^{+\infty}\right)\xi}=o\left(D\xi^{j_{+}}\textup{e}^{-\min\left(\lambda_{1,P}^{+\infty},\lambda_{i,P}^{+\infty}\right)\xi}\right)\text{ as }\xi\to+\infty.
\]
\end{lem}

\begin{proof}
The proof is by contradiction: we assume from now on that, on the contrary, the above 
two asymptotic estimates are satisfied. This means that, near $+\infty$, any translation of $\left(\varphi,\psi\right)$ dominates 
 $\left(\hat{\varphi},\hat{\psi}\right)$ (in the sense of the competitive ordering).

Here are the three steps of the sliding argument of this proof.

\textbf{Step 1:} choose $\xi_0\in\mathbb{R}$ sufficiently close to $-\infty$ and such that for all $\xi \leq \xi_0$,
\begin{equation}\label{eq:lemA5step1}
\left(\hat{\varphi},\hat{\psi}\right)\left(\xi\right) \succeq \left( \frac{3}{4}, \frac{1}{4}\right)\text{ and }\hat{\varphi}(\xi) \geq \max\left(\frac{5-a}{8-4a},  \frac{3+b}{4b} + \hat{\psi}(\xi)\right).  
\end{equation}
Notice that such a $\xi_0$ exists indeed, since $(\hat\varphi, \hat\psi)(-\infty)=(1,0)$, and  $\max\left(\frac{5-a}{8-4a},\frac{3+b}{4b}\right)<1$ with $a<1$ and $b>1$. We claim that if there exists $\tau \in \mathbb{R}$ such that
\[
\left(\varphi,\psi\right)\left(\xi_0-\tau\right)\succ\left(\hat{\varphi},\hat{\psi}\right)\left(\xi_0\right),
\]
then
\[
\left(\varphi,\psi\right)\left(\xi-\tau\right)\succ\left(\hat{\varphi},\hat{\psi}\right)\left(\xi\right)\text{ for all }\xi\leq\xi_0.
\]
Clearly, there exists $\varepsilon\in\left(0,\frac{1}{4}\right]$ such that 
\[
\left(\varphi,\psi\right)\left(\xi-\tau\right)\succeq\left(\hat{\varphi},\hat{\psi}\right)\left(\xi\right)+\varepsilon\left(-1,1\right)\text{ for all }\xi\leq\xi_0.
\]

Now, let $\varepsilon^{\star}\in\left[0,\frac{1}{4}\right]$ be the infimum of all these $\varepsilon$ and assume by contradiction that 
$\varepsilon^{\star}>0$. In view of the limiting values at $-\infty$ and of the inequality at $\xi_0$, there exists 
$\xi^{\star}\in\left(-\infty,\xi_0\right)$ such that 
\[
\left(\varphi,\psi\right)\left(\xi^{\star}-\tau\right)\succeq\left(\hat{\varphi},\hat{\psi}\right)\left(\xi^{\star}\right)+\varepsilon^{\star}\left(-1,1\right)
\]
with, most importantly, equality for at least one of the components. 
Let us verify that $\left(\varphi_{\varepsilon^{\star}},\psi_{\varepsilon^{\star}}\right)=\left(\hat{\varphi},\hat{\psi}\right)+\varepsilon^{\star}\left(-1,1\right)$ is a sub-solution. Since 
$\left(\varphi,\psi\right)$ satisfies by definition
\[
\left\{ \begin{matrix}-\hat{\varphi}''-c\hat{\varphi}'= (\hat{c} - c) \hat{\varphi}' + \hat{\varphi}\left(1-\hat{\varphi}-a\hat{\psi}\right)\\
-d\hat{\psi}''-c\hat{\psi}'=(\hat{c} - c) \hat{\psi}' +  r\hat{\psi}\left(1-\hat{\psi}-b\hat{\varphi}\right)
\end{matrix}\right.,
\]
we find (note that, by \propref{monotonicity_of_the_profiles}, $(\hat{\varphi}', \hat{\psi}') \preceq (0,0)$)
\[
\left\{ \begin{matrix}-\varphi_{\varepsilon^{\star}}''-c\varphi_{\varepsilon^{\star}}'-\varphi_{\varepsilon^{\star}}\left(1-\varphi_{\varepsilon^{\star}}-a\psi_{\varepsilon^{\star}}\right) < \varepsilon^{\star}\left(1-\left(2-a\right)\varphi_{\varepsilon^{\star}}-\left(1-a\right)\varepsilon^{\star}-a\psi_{\varepsilon^{\star}}\right)\\
-d\psi_{\varepsilon^{\star}}''-c\psi_{\varepsilon^{\star}}'-r\psi_{\varepsilon^{\star}}\left(1-\psi_{\varepsilon^{\star}}-b\varphi_{\varepsilon^{\star}}\right) >-r\varepsilon^{\star}\left(1-\left(2-b\right)\psi_{\varepsilon^{\star}}-\left(b-1\right)\varepsilon^{\star}-b\varphi_{\varepsilon^{\star}}\right)
\end{matrix}\right..
\]
From 
\[
\left(\begin{matrix}1-\left(2-a\right)\varphi_{\varepsilon^{\star}}-\left(1-a\right)\varepsilon^{\star}-a\psi_{\varepsilon^{\star}} \\
-\left(1-\left(2-b\right)\psi_{\varepsilon^{\star}}-\left(b-1\right)\varepsilon^{\star}-b\varphi_{\varepsilon^{\star}}\right)
\end{matrix}\right)
\preceq\left(\begin{matrix}1-\left(2-a\right)\hat{\varphi}+\frac{1-a}{4}-a\hat{\psi} \\
-\left(1-\left(2-b\right)\hat{\psi}+\frac{b-1}{4}-b\hat{\varphi}\right)
\end{matrix}\right),
\]
we deduce by (\ref{eq:lemA5step1}) that 
\[
\left(\begin{matrix}1-\left(2-a\right)\varphi_{\varepsilon^{\star}}-\left(1-a\right)\varepsilon^{\star}-a\psi_{\varepsilon^{\star}} \\
-\left(1-\left(2-b\right)\psi_{\varepsilon^{\star}}-\left(b-1\right)\varepsilon^{\star}-b\varphi_{\varepsilon^{\star}}\right)
\end{matrix}\right)\preceq\left(\begin{matrix}0 \\ 0 \end{matrix}\right)\text{ for all }\xi\leq\xi_0.
\]

We are now in position to apply the strong comparison principle of \thmref{Competitive_comparison_principle} and 
deduce from the existence of $\xi^{\star}$ a contradiction. Hence $\varepsilon^{\star}=0$, that is 
\[
\left(\varphi,\psi\right)\left(\xi-\tau\right)\succeq\left(\hat{\varphi},\hat{\psi}\right)\left(\xi\right)\text{ for all }\xi\leq\xi_0.
\]
Finally, by strong comparison principle the strict inequality must hold for any $\xi \leq \xi_0$.

\textbf{Step 2:} in this step, we show the existence of $\tau_1$ such that, for all $\tau\geq\tau_1$, 
\[
\left(\varphi,\psi\right)\left(\xi-\tau\right)\succ\left(\hat{\varphi},\hat{\psi}\right)\left(\xi\right)\text{ for all }\xi\in\mathbb{R}.
\]
To this end, we fix $\xi_0$ as in (\ref{eq:lemA5step1}) and choose $\tau_0>0$ large so that 
\[
\left(\varphi,\psi\right)\left(\xi_0 - \tau\right) \succ \left(\hat{\varphi}, \hat{\psi}\right)\left(\xi_0\right)\text{ for all }\tau \geq \tau_0.
\]
By Step 1, we deduce that for all $\tau \geq \tau_0$, 
\[
\left(\varphi,\psi\right)\left(\xi-\tau\right)\succeq\left(\hat{\varphi},\hat{\psi}\right)\left(\xi\right)\text{ for all }\xi\leq\xi_0.
\]
Next, we use the asymptotic behavior of $(\varphi,\psi)$ and $(\hat{\varphi},\hat{\psi})$ at $+\infty$ to 
choose $\tau_1 \geq \tau_0$ such that for all $\tau \geq \tau_1$, 
\[
\left(\varphi,\psi\right)\left(\xi-\tau\right)\succeq\left(\hat{\varphi},\hat{\psi}\right)\left(\xi\right)\text{ for all }\xi\geq\xi_0.
\]
The above two inequalities complete Step 2.

\textbf{Step 3:} define $\tau^{\star}$ as the infimum of all $\tau$ such that the preceding inequality holds true. By construction, 
\[
\left(\varphi,\psi\right)\left(\xi-\tau^{\star}\right)\succeq\left(\hat{\varphi},\hat{\psi}\right)\left(\xi\right)\text{ for all }\xi\in\mathbb{R}.
\]
It suffices to show  the existence of $\xi^{\star}\in\mathbb{R}$ such that 
$\left(\varphi,\psi\right)\left(\xi^{\star}-\tau^{\star}\right)\succeq\left(\varphi,\psi\right)\left(\xi^{\star}\right)$ 
with equality for at least one component. Granted, then the strong comparison principle yields 
\[
\left(\varphi,\psi\right)\left(\xi-\tau^{\star}\right)=\left(\hat{\varphi},\hat{\psi}\right)\left(\xi\right)\text{ for all }\xi\in\mathbb{R},
\]
and the proof is ended. Suppose by contradiction that such a $\xi^\star$ does not exist, that is 
\[
\left(\varphi,\psi\right)\left(\xi-\tau^{\star}\right)\succ\left(\hat{\varphi},\hat{\psi}\right)\left(\xi\right)\text{ for all }\xi\in\mathbb{R}.
\]
Now, the asymptotic behavior assumed at the beginning of this proof implies  
\[
\lim_{\xi\to+\infty}\frac{\varphi\left(\xi-\tau^\star\right)}{\hat{\varphi}\left(\xi\right)}= +\infty, \quad \text{ and }
\quad 
\lim_{\xi\to+\infty}\frac{1-\psi\left(\xi-\tau^\star\right)}{1-\hat{\psi}\left(\xi\right)}= +\infty.
\]
Hence, there exists $\xi_1 >0$ large and $\delta>0$ small such that for all $\tau \in (\tau^{\star} - \delta, \tau^{\star})$, 
\[
\left(\varphi,\psi\right)\left(\xi-\tau\right)\succ\left(\hat{\varphi},\hat{\psi}\right)\left(\xi\right)\text{ for all }\xi \geq \xi_1.
\]
By taking $\delta>0$ small, we have also that, for all $\tau \in (\tau^{\star} - \delta, \tau^{\star})$, 
\[
\left(\varphi,\psi\right)\left(\xi-\tau\right)\succ\left(\hat{\varphi},\hat{\psi}\right)\left(\xi\right) \quad \text{ for all }\xi \in [\xi_0, \xi_1].
\]
Finally, the result in Step 1 implies that for all $\tau \in (\tau^{\star} - \delta, \tau^{\star})$,
\[
\left(\varphi,\psi\right)\left(\xi-\tau\right)\succ\left(\hat{\varphi},\hat{\psi}\right)\left(\xi\right)\text{ for all }\xi \in \mathbb{R}.
\]
This contradicts the minimality of $\tau^{\star}$.
\end{proof}

From the preceding lemma, \lemref{Exponential_decays_profiles} and the respective monotonicities of 
$c\mapsto\lambda_{1,P}^{+\infty}$, $c\mapsto\lambda_{2,P}^{+\infty}$ and $c\mapsto\lambda_{3,P}^{+\infty}$, we deduce the 
following corollary which is a refinement of \lemref{Exponential_decays_profiles}. Basically, it discards the possibility 
of solutions having a fast decay and a super-critical speed. 

\begin{cor}\label{cor:Precise_decays_near_unstable_state}Let $P=\left(c,d,r,a,b\right)\in E$ with $c>c_{\textup{LLW}}^{d,r,a,b}$ and $\left(\varphi,\psi\right)$ be a profile 
of traveling wave solution of (\ref{eq:competition_diffusion_system}) with speed $c$. 

Then there exist $A>0$ and $C>0$ such that, as $\xi\to+\infty$:
\begin{enumerate}
\item if $\lambda_{1,P}^{+\infty}<\lambda_{3,P}^{+\infty}$, then
\[
\left\{ \begin{matrix}\varphi\left(\xi\right)=C\textup{e}^{-\lambda_{3,P}^{+\infty}\xi}+\text{h.o.t.}\\
\psi\left(\xi\right)=1-A\textup{e}^{-\lambda_{1,P}^{+\infty}\xi}+\text{h.o.t.}
\end{matrix}\right.;
\]
\item if $\lambda_{1,P}^{+\infty}=\lambda_{3,P}^{+\infty}$, then
\[
\left\{ \begin{matrix}\varphi\left(\xi\right)=\frac{2d\lambda_{1,P}^{+\infty}-c}{a}C\textup{e}^{-\lambda_{1,P}^{+\infty}\xi}+\text{h.o.t.}\\
\psi\left(\xi\right)=1-C\xi\textup{e}^{-\lambda_{1,P}^{+\infty}\xi}+\text{h.o.t.}
\end{matrix}\right.;
\]
\item if $\lambda_{1,P}^{+\infty}\in\left(\lambda_{3,P}^{+\infty},\lambda_{2,P}^{+\infty}\right)$,
then $R_{P}^{+\infty}\left(\lambda_{3,P}^{+\infty}\right)<0$ and
\[
\left\{ \begin{matrix}\varphi\left(\xi\right)=C\textup{e}^{-\lambda_{3,P}^{+\infty}\xi}+\text{h.o.t.}\\
\psi\left(\xi\right)=1+\frac{rb}{R_{P}^{+\infty}\left(\lambda_{3,P}^{+\infty}\right)}C\textup{e}^{-\lambda_{3,P}^{+\infty}\xi}+\text{h.o.t.}
\end{matrix}\right.;
\]
\item if $\lambda_{1,P}^{+\infty}=\lambda_{2,P}^{+\infty}$, then $R_{P}^{+\infty}\left(\lambda_{3,P}^{+\infty}\right)<0$ and 
\[
\left\{ \begin{matrix}\varphi\left(\xi\right)=C\textup{e}^{-\lambda_{3,P}^{+\infty}\xi}+\text{h.o.t.}\\
\psi\left(\xi\right)=1+\frac{rb}{R_{P}^{+\infty}\left(\lambda_{3,P}^{+\infty}\right)}C\textup{e}^{-\lambda_{3,P}^{+\infty}\xi}+\text{h.o.t.}
\end{matrix}\right.;
\]
\item if $\lambda_{1,P}^{+\infty}>\lambda_{2,P}^{+\infty}$, then $R_{P}^{+\infty}\left(\lambda_{3,P}^{+\infty}\right)<0$ and
\[
\left\{ \begin{matrix}\varphi\left(\xi\right)=C\textup{e}^{-\lambda_{3,P}^{+\infty}\xi}+\text{h.o.t.}\\
\psi\left(\xi\right)=1+\frac{rb}{R_{P}^{+\infty}\left(\lambda_{3,P}^{+\infty}\right)}C\textup{e}^{-\lambda_{3,P}^{+\infty}\xi}+\text{h.o.t.}
\end{matrix}\right.;
\]
\end{enumerate}
\end{cor}

\begin{rem*}
We emphasize that there exists a unique translation of the profile such that the normalization $C=1$ holds. The remaining degree of freedom in the first case above ($A$ can still take any positive value \textit{a priori}) is the main difficulty regarding uniqueness.
\end{rem*}

\subsection{Uniqueness and continuity}
We are now in position to establish the following uniqueness result.

\begin{prop}
\label{prop:uniqueness_of_the_profiles} Let $P=\left(c,d,r,a,b\right)\in E$ such that 
$\lambda_{1,P}^{+\infty}\geq\lambda_{3,P}^{+\infty}$.

Let $\left(\varphi,\psi\right)\in\mathscr{C}^2\left(\mathbb{R},\left[0,1\right]^2\right)$ and 
$\left(\hat{\varphi},\hat{\psi}\right)\in\mathscr{C}^2\left(\mathbb{R},\left[0,1\right]^2\right)$ be two
profiles of traveling wave solution of (\ref{eq:competition_diffusion_system}) with speed
$c$. 

Then $\left(\varphi,\psi\right)$ and $\left(\hat{\varphi},\hat{\psi}\right)$ coincide up to translation.
\end{prop}

\begin{proof}
The proof relies upon a sliding argument again.

In view of \corref{Precise_decays_near_unstable_state}, if $c>c_{\textup{LLW}}^{d,r,a,b}$, the assumption 
$\lambda_{1,P}^{+\infty}\geq\lambda_{3,P}^{+\infty}$ immediately yields that the two profiles can be normalized so that 
they have the same decay at $+\infty$. Similarly, in view of \lemref{Exponential_decays_profiles}, if $c=c_{\textup{LLW}}^{d,r,a,b}$, 
then the two profiles can be normalized so that their decays either coincide or are well-ordered. In all cases, we can fix 
\textit{a priori} the roles of the two profiles so that $\left(\varphi,\psi\right)$ dominates 
$\left(\hat{\varphi},\hat{\psi}\right)$ near $+\infty$. 
By following the first two steps of the proof of 
\lemref{ordering_of_the_decays_of_the_profiles},   we can assume without loss of generality the existence 
of $\tau_0\in\mathbb{R}$ such that, for all $\tau\geq\tau_0$, 
\[
\left(\varphi,\psi\right)\left(\xi-\tau\right)\succeq\left(\hat{\varphi},\hat{\psi}\right)\left(\xi\right)\text{ for all }\xi \in \mathbb{R}.
\]

Next, define $\tau^{\star} \in \mathbb{R}$ as the infimum of all $\tau$ such that the preceding inequality holds true. 
It again suffices to show that there exists $\xi^{\star} \in \mathbb{R}$ where equality holds for one of the components. Assume 
on the contrary that no such $\xi^{\star}$ exists. Thus the preceding inequality is strict for both components for all 
$\xi\in\mathbb{R}$. Now, note that
\[
\lim_{\xi\to+\infty}\frac{\varphi\left(\xi-\tau^\star\right)}{\hat{\varphi}\left(\xi\right)}\geq 1\quad \text{ and }\quad \lim_{\xi\to+\infty}\frac{1-\psi\left(\xi-\tau^\star\right)}{1-\hat{\psi}\left(\xi\right)}\geq 1.
\]
Next, we claim that 
\[
\lim_{\xi\to+\infty}\frac{\varphi\left(\xi-\tau^\star\right)}{\hat{\varphi}\left(\xi\right)}= 1\quad \text{ or }\quad \lim_{\xi\to+\infty}\frac{1-\psi\left(\xi-\tau^\star\right)}{1-\hat{\psi}\left(\xi\right)}=1.
\]
Otherwise we may further reduce $\tau^{\star}$, 
just as in the proof of \lemref{ordering_of_the_decays_of_the_profiles}. Notice that this equality directly yields $\tau^\star=0$, 
that is
\[
\left(\varphi,\psi\right)\left(\xi\right)\succeq\left(\hat{\varphi},\hat{\psi}\right)\left(\xi\right)\text{ for all }\xi \in \mathbb{R}.
\]

Next, from the fact that $\lambda^{+\infty}_{1,P} \geq \lambda^{+\infty}_{3,P}$ and, depending on $c$, \corref{Precise_decays_near_unstable_state} or \lemref{Exponential_decays_profiles}, both of the above limits are equal to 1.

Since the decay rate at $+\infty$ of both profiles coincide, we can reverse the profiles and repeat the proof. This leads to 
\[
\left(\hat{\varphi},\hat{\psi}\right)\left(\xi\right)\succeq\left(\varphi,\psi\right)\left(\xi\right)\text{ for all }\xi\in\mathbb{R}.
\]
Hence the two profiles actually coincide, which directly contradicts the assumption of nonexistence of $\xi^\star$. 

In the end, $\xi^\star$ exists indeed and, by virtue of the strong comparison principle, the two normalized profiles coincide. 
In other words, the two profiles coincide up to translation. 
\end{proof}

\begin{cor}\label{cor:uniqueness_of_the_profiles_for_all_speeds}
Let $\left(d,r,a,b\right)\in\Pi$ such that $d\leq 2+\frac{r}{1-a}$. Then each speed $c\geq c_{\textup{LLW}}^{d,r,a,b}$ is associated with a unique profile (up to translation).
\end{cor}

\begin{proof}
It suffices to prove that, for all $c\geq c_{\textup{LLW}}^{d,r,a,b}$, 
$\lambda_{1,\left(c,d,r,a,b\right)}^{+\infty}\geq\lambda_{3,\left(c,d,r,a,b\right)}^{+\infty}$.

Noticing that this inequality is equivalent to 
$R_{\left(c,d,r,a,b\right)}^{+\infty}\left(\lambda_{3,\left(c,d,r,a,b\right)}^{+\infty}\right)\leq 0$, 
we find that we just have to prove that, for all $c\geq c_{\textup{LLW}}^{d,r,a,b}$, 
\[
d\leq\frac{c\lambda_{3,\left(c,d,r,a,b\right)}^{+\infty}+r}{\left(\lambda_{3,\left(c,d,r,a,b\right)}^{+\infty}\right)^2}
\]
and, using the polynomial equation satisfied by $\lambda_{3,\left(c,d,r,a,b\right)}^{+\infty}=\frac{c-\sqrt{c^{2}-4\left(1-a\right)}}{2}$, this reads 
\[
d\leq\frac{\left(\lambda_{3,\left(c,d,r,a,b\right)}^{+\infty}\right)^2+1-a+r}{\left(\lambda_{3,\left(c,d,r,a,b\right)}^{+\infty}\right)^2}=1+\frac{1-a+r}{\left(\lambda_{3,\left(c,d,r,a,b\right)}^{+\infty}\right)^2}
\]

It only remains to show that 
\[
\inf_{c\geq c_{\textup{LLW}}^{d,r,a,b}} \frac{1-a+r}{\left(\lambda_{3,\left(c,d,r,a,b\right)}^{+\infty}\right)^2}\geq 1+\frac{r}{1-a}.
\]

The above inequality follows actually quite easily: 
\[
\inf_{c\geq  c_{\textup{LLW}}^{d,r,a,b}} \frac{1-a+r}{\left(\lambda_{3,\left(c,d,r,a,b\right)}^{+\infty}\right)^2} \geq \inf_{c\geq 2\sqrt{1-a}} \frac{1-a+r}{\left(\lambda_{3,\left(c,d,r,a,b\right)}^{+\infty}\right)^2} = \frac{1-a+r}{\sup\limits_{c\geq 2\sqrt{1-a}} \left(\lambda_{3,\left(c,d,r,a,b\right)}^{+\infty}\right)^2}
\]
and, by monotonicity,
\[
\sup\limits_{c\geq 2\sqrt{1-a}} \left(\lambda_{3,\left(c,d,r,a,b\right)}^{+\infty}\right)^2=\left(\lambda_{3,\left(2\sqrt{1-a},d,r,a,b\right)}^{+\infty}\right)^2=1-a.
\]

\end{proof}

Finally, as a consequence of the uniqueness, we also have the continuity of the profiles with respect to the parameters.

\begin{prop}\label{prop:continuity_of_the_profiles} Let
\[
E_u=\left\{P\in E\ |\ \lambda_{1,P}^{+\infty}\geq\lambda_{3,P}^{+\infty}\right\}.
\]

For all $P\in E_u$, let $\left(\Phi^{P},\Psi^{P}\right)$ be the unique profile
of traveling wave solution of (\ref{eq:competition_diffusion_system}) with speed $c$ satisfying $\Psi^{P}\left(0\right)=\frac{1}{2}$.

Then $P\mapsto\left(\Phi^{P},\Psi^{P}\right)$ is in $\mathscr{C}\left(\interior{E_u},\mathscr{C}_{b}\left(\mathbb{R},\mathbb{R}^{2}\right)\right)$.
\end{prop}

\begin{proof}
    Let $P_{\infty}\in\overline{\interior{E_u}}$ and $\left(P_{n}\right)_{n\in\mathbb{N}}\in\left(\interior{E_u}\right)^{\mathbb{N}}$
such that $\lim\limits _{n\to+\infty}P_{n}=P_{\infty}$. By standard
elliptic estimates (see Gilbarg\textendash Trudinger \cite{Gilbarg_Trudin}),
the sequence $\left(\left(\Phi^{P_{n}},\Psi^{P_{n}}\right)\right)_{n\in\mathbb{N}}$
converges, up to a diagonal extraction, in $\mathscr{C}_{loc}^{2}$.
\textit{A fortiori} it converges pointwise in $\mathbb{R}$. The limit
$\left(\Phi_{\infty},\Psi_{\infty}\right)$ is continuous, monotonic,
and satisfies $\Psi_{\infty}\left(0\right)=\frac{1}{2}$. Using standard
elliptic estimates to study the asymptotic behaviors, we find easily
\[
\lim_{-\infty}\left(\Phi_{\infty},\Psi_{\infty}\right)\in\left\{ \left(1,0\right),\left(0,0\right)\right\} \text{ and }\lim_{+\infty}\left(\Phi_{\infty},\Psi_{\infty}\right)=\left(0,1\right).
\]
 If $\Phi_{\infty}$ is null in $\mathbb{R}$, then $\xi\mapsto\Psi_{\infty}\left(-\xi\right)$
is a KPP traveling wave with negative speed, which is impossible.
Therefore the limit at $-\infty$ of $\left(\Phi_{\infty},\Psi_{\infty}\right)$
is $\left(1,0\right)$. This shows that the sequence of monotonic
functions $\left(\left(\Phi^{P_{n}},\Psi^{P_{n}}\right)\right)_{n\in\mathbb{N}}$
converges pointwise in $\left[-\infty,+\infty\right]$, whence by
a variant of the Dini theorem it converges uniformly in $\mathbb{R}$.
In view of the preceding uniqueness result, the limit is exactly $\left(\Phi^{P_{\infty}},\Psi^{P_{\infty}}\right)$.
Finally, a classical uniqueness and compactness argument shows that
the previous diagonal extraction was not necessary and the sequence
$\left(\left(\Phi^{P_{n}},\Psi^{P_{n}}\right)\right)_{n\in\mathbb{N}}$
converges indeed in $\mathscr{C}_{b}\left(\mathbb{R},\mathbb{R}^{2}\right)$
to $\left(\Phi^{P_{\infty}},\Psi^{P_{\infty}}\right)$.
\end{proof}

\bibliographystyle{plain}
\bibliography{ref}

\end{document}